\documentclass[12pt]{amsart}%
\usepackage{geometry}
\usepackage{amsmath}
\usepackage{amsfonts}
\usepackage{mathtools}
\usepackage{tikz-cd}
\usepackage{amssymb}
\usepackage{stmaryrd}
\usepackage{graphicx}
\usepackage{tikz-cd}
\usepackage[toc]{appendix}%
\providecommand{\U}[1]{\protect\rule{.1in}{.1in}}
\newtheorem{theorem}{Theorem}[section]
\theoremstyle{plain}

\newtheorem{question}{Question}

\newtheorem{claim}{Claim}
\newtheorem*{fact}{Fact}

\newtheorem{corollary}[theorem]{Corollary}

\newtheorem{lemma}[theorem]{Lemma}

\newtheorem{proposition}[theorem]{Proposition}

\theoremstyle{definition}
\newtheorem{definition}[theorem]{Definition}
\newtheorem{remark}[theorem]{Remark}
\newtheorem{notation}{Notation}
\newtheorem{example}{Example}
\numberwithin{equation}{section}

\begin{document}
\title[Boundaries for semidirect products with $%
\mathbb{Z}
$]{Group boundaries for semidirect products with $%
\mathbb{Z}
$}
\author{Craig R. Guilbault}
\address{Department of Mathematical Sciences\\
University of Wisconsin-Milwaukee, Milwaukee, WI 53201}
\email{craigg@uwm.edu}
\author{Brendan Burns Healy}
\address{Department of Mathematical Sciences\\
University of Wisconsin-Milwaukee, Milwaukee, WI 53201}
\email{healyb@uwm.edu}
\author{Brian Pietsch}
\address{The Adecco Group}
\email{brian.pietsch@adeccogroup.com}
\thanks{This research was supported in part by Simons Foundation Grant 427244, CRG}
\date{\today}

\begin{abstract}
Bestvina's notion of a $\mathcal{Z}$-structure provides a general framework
for group boundaries that includes Gromov boundaries of hyperbolic groups and
visual boundaries of CAT(0) groups as special cases. A refinement, known as an
$E\mathcal{Z}$-structure has proven useful in attacks on the Novikov
Conjecture and related problems. Characterizations of groups admitting a
$\mathcal{Z}$- or $E\mathcal{Z}$-structure are longstanding open problems. In
this paper, we examine groups of the form $G\rtimes_{\phi}%
\mathbb{Z}
$. For example, we show that, if $G$ is torsion-free and admits a
$\mathcal{Z}$-structure, then so does every semidirect product of this type.
We prove a similar theorem for $E\mathcal{Z}$-structures, under an additional hypothesis.

As applications, we show that all closed 3-manifold groups admit $\mathcal{Z}%
$-structures, as do all strongly polycyclic groups and all groups of
polynomial growth. In those latter cases our $\mathcal{Z}$-boundaries are
always spheres. This allows one to make strong conclusions about the group
cohomology and end invariants of those groups. In another direction, we expand
upon the notion of an $E\mathcal{Z}$-structure and discuss new applications to
the Novikov Conjecture.

\end{abstract}
\maketitle
\tableofcontents

\setcounter{tocdepth}{1}

\section{Introduction\label{Section: Introduction}}

Bestvina \cite{Bes96} introduced the notion of a $\mathcal{Z}$\emph{-boundary}
for a group as a means to both unify the study of Gromov boundaries of
hyperbolic groups with visual boundaries of CAT(0) groups and to provide a
framework for assigning boundaries to a wider class of groups. Initially
restricted to torsion-free groups, work by Dranishnikov \cite{Dra06} expanded
Bestvina's definition to allow for groups with torsion. Roughly speaking, a
groups $G$ admits a $\mathcal{Z}$\emph{-structure} if it acts nicely (properly
and cocompactly) on a nice space $X$ (an absolute retract or AR) which admits
a nice compactification $\overline{X}$ (a $\mathcal{Z}$\emph{-}set
compactification) such that compact subsets of $X$ vanish in $\overline{X}$ as
they are pushed toward the $\mathcal{Z}$\emph{-boundary}, $Z:=\overline{X}-X$,
by elements of $G$. Work from those papers, as well as \cite{GeOn07} and
\cite{GuMo19}, highlights some of the useful properties of a $\mathcal{Z}$-boundary.

In order to admit a $\mathcal{Z}$-boundary, $G$ must first admit a proper
cocompact action on an absolute retract $X$. This rules out many groups. For
example, a torsion free group admits this type of action if and only if it is
\emph{type F. }The general class of groups which admit such an action have
been defined to be \emph{type F}$_{\text{AR}}^{\ast}$. (When $X$ is a CW
complex and the action is cellular, it is simply called \emph{type F}$^{\ast}%
$.) A significant open question asks whether all type F or type F$_{\text{AR}%
}^{\ast}$ groups admit $\mathcal{Z}$-structures. Current progress has focused
on special classes of groups. For example, in addition to hyperbolic and
CAT(0) groups, all systolic and all generalized Baumslag-Solitar groups have
been shown to admit $\mathcal{Z}$-boundaries. In a more general direction, all
free and direct products of groups which admit $\mathcal{Z}$-boundaries are
known to admit $\mathcal{Z}$-boundaries (references to be provided later). One
of the main results in this paper is in that spirit.

\begin{theorem}
\label{Theorem A} If a torsion-free group $G$ admits a $\mathcal{Z}$-structure
with boundary $Z$, then every semidirect product $G\rtimes_{\phi}${$%
\mathbb{Z}
$} admits a $\mathcal{Z}$-structure with boundary equal to the suspension of
$Z$.
\end{theorem}

A significant special case of Theorem \ref{Theorem A} includes all free-by-$%
\mathbb{Z}
$ groups---a well-studied class that is notable because it contains some
hyperbolic groups \cite{BeFe92}, some CAT(0) groups, and \emph{many} groups
which are neither. Hyperbolic group experts might be surprised that, for this
collection of groups, our $\mathcal{Z}$-boundary is \emph{always} a suspended
Cantor set---an illustration that even hyperbolic groups need not admit unique
$\mathcal{Z}$-boundaries.

With the aid of Theorem \ref{Theorem A} we obtain a variety of new results.

\begin{theorem}
\label{Theorem C}Every closed 3-manifold group admits a $\mathcal{Z}$-structure.
\end{theorem}

\begin{theorem}
\label{Theorem B}Every strongly polycyclic and every finitely generated
nilpotent group $G$ admits a $\mathcal{Z}$-structure with a $k$-sphere as
boundary, for some $k\geq-1$.
\end{theorem}

By combining Theorem \ref{Theorem B} with Gromov's theorem on groups of
polynomial growth \cite{Gro81}, the main result of \cite{KMN09}, and a
boundary swapping trick, Theorem \ref{Theorem B} can be pushed further to obtain:

\begin{theorem}
\label{Theorem E}Every group $G$ of polynomial growth admits a $\mathcal{Z}%
$-structure with a $k$-sphere as boundary, for some $k\geq-1$.
\end{theorem}

By applying standard properties of $\mathcal{Z}$-structures (see \cite{Bes96})
and some recent extensions (see \cite{GuMo21}), one may immediately deduce:

\begin{corollary}
\label{Corollary: applications to group cohomology}Every group $G$ that is
strongly polycyclic or of polynomial growth, has the same $%
\mathbb{Z}
G$-cohomology as $%
\mathbb{Z}
^{n}$, for some $n$. In addition, $G$ is semistable and has the same
pro-homotopy groups at infinity as $%
\mathbb{Z}
^{n}$. If $G$ is torsion-free, it is a Poincar\'{e} duality group.
\end{corollary}

Several of the assertions covered by Corollary
\ref{Corollary: applications to group cohomology} are known by other methods.
Details and references will be provided in Section \ref{nilpotent}.

The reader will notice that the above results include groups with torsion.
Indeed, we will prove a more general version of Theorem \ref{Theorem A} which
can be applied whenever there exists a cocompact $\underline{E}G$-space. As a
corollary, if $G$ is hyperbolic, CAT(0), or systolic, then $G\rtimes_{\phi}${$%
\mathbb{Z}
$} admits a $\mathcal{Z}$-structure; and by applying our own Theorem
\ref{Theorem E}, $G\rtimes_{\phi}${$%
\mathbb{Z}
$} admits a $\mathcal{Z}$-structure whenever $G$ is of polynomial growth. We
save the most general statements for Section
\ref{sec: proof with torsion permitted}.\medskip

Farrell and Lafont \cite{FaLa05} defined $E\mathcal{Z}$-structures---a
refinement of $\mathcal{Z}$-structures which adds a requirement that the
$G$-action on $X$ extends to $\overline{X}$. Their main application was to
show that each torsion-free group that admits an $E\mathcal{Z}$-structure
satisfies the famous Novikov Conjecture. For that and other reasons, we will
expend significant effort proving the existence of $E\mathcal{Z}$-structures
whenever possible. The following is a corollary of a more general theorem that
will be proved in Section \ref{Section: EZ-structures}.

\begin{theorem}
\label{Theorem D}Suppose $G$ is a hyperbolic group, a finitely generated
abelian group, or a CAT(0) group with the isolated flats property. Then, for
any $\phi\in\operatorname*{Aut}(G)$, $G\rtimes_{\phi}${$%
\mathbb{Z}
$} admits an $E\mathcal{Z}$-structure with boundary equal to the suspension of
the Gromov or visual boundary of $G$.
\end{theorem}

By \cite{FaLa05}, Theorem \ref{Theorem D} implies the Novikov Conjecture, for
the groups covered there, whenever they are torsion-free. (Other proofs of
Novikov are known in these particular cases. See Section \ref{subsec:novikov}
for details.)

A related approach to the Novikov Conjecture, that can be applied to groups
with torsion, has been developed by Rosenthal \cite{Ros04}, \cite{Ros06},
\cite{Ros12}. In order to apply that work, we develop a refinement of
$E\mathcal{Z}$-structures which we call $\underline{E\mathcal{Z}}$-structures.
The additional requirement is that, for each finite $H\leq G$, the fixed set
$X^{H}$ is an absolute retract whose closure in $\overline{X}$ is a
$\mathcal{Z}$-compactification of $X^{H}$. By existing work, to be detailed
later, all hyperbolic, CAT(0), and systolic groups support
$\underline{E\mathcal{Z}}$-structures. We will prove that, under appropriate
hypotheses, an $\underline{E\mathcal{Z}}$-structure on $G$ implies the
existence of an $\underline{E\mathcal{Z}}$-structure on $G\rtimes_{\phi}%
\mathbb{Z}
$. Special cases include all of the groups mentioned in Theorem
\ref{Theorem D}.\medskip

The layout of this paper is as follows. In Section \ref{sec:background} we
provide much of the necessary background, including definitions and a variety
of notational conventions to be used throughout the paper. In Section
\ref{Section: Mapping tori and telescopes as classifying spaces} we explain
the role of mapping tori and mapping telescopes in the study of $G\rtimes
_{\phi}%
\mathbb{Z}
$ (thereby explaining why the groups themselves are sometimes called mapping
tori). All of this works best when $G$ is torsion-free, so that is our focus
in this section. From there we turn to the proof of Theorem \ref{Theorem A}
which occupies Sections \ref{sec:controlled} and \ref{sec:proof}. This is the
heart of the paper and also the most technical portion. In Section
\ref{sec: proof with torsion permitted} material from the previous sections is
generalized---to the extent possible---to allow for groups with torsion. In
Sections \ref{Section: EZ-structures} and \ref{Section: EZ-bar structures} we
develop and discuss $E\mathcal{Z}$-structures and $\underline{E\mathcal{Z}}%
$-structures. From there we turn to applications of our main theorems. In
particular, Section \ref{sec:applications} contains a detailed discussion of
implications to the Novikov Conjecture. From there, we move on to proofs of
Theorems \ref{Theorem C}-\ref{Theorem E}. We close the paper with a discussion
of some open questions. In addition, we include an appendix which looks at an
interesting special case of the work presented here---the integral Heisenberg
group. That discussion can be viewed as motivation for the more general
methods used elsewhere. Before delving into the more abstract constructions
and proofs, the reader might benefit from a quick look at this appendix.

This paper is an expansion of work begun in the dissertation of the third
author, Brian Pietsch, written under the direction of the first author at the
University of Wisconsin-Milwaukee \cite{Pie18}. All three authors wish to
acknowledge Chris Hruska, Boris Okun, Hoang Nguyen, Mike Mihalik, Ian Leary,
and an anonymous referee for useful comments and suggestions during and at the
completion of this project. Theorem \ref{Theorem A} was anticipated by
Bestvina in \cite[Ex.3.1]{Bes96}. A portion of the work presented in this
paper can be viewed as a completion of the argument sketched there.



\section{Background}

\label{sec:background}


\subsection{Semidirect products with $\mathbb{Z}$}

Here we establish some conventions regarding semidirect products with $%
\mathbb{Z}
$. First note that a group $\Gamma$ is a semidirect product of $G$ with $%
\mathbb{Z}
$ if and only if there exists a short exact sequence%
\[
1\rightarrow G\rightarrow\Gamma\rightarrow%
\mathbb{Z}
\rightarrow1
\]
In other words (in this special case) a semidirect product is the same as a
group extension. Viewed as a semidirect product, $\Gamma$ is determined by a
single element $\phi\in\operatorname*{Aut}\left(  G\right)  $ which then
determines the homomorphism $\Phi:%
\mathbb{Z}
\rightarrow\operatorname*{Aut}\left(  G\right)  $ taking $n$ to $\phi^{n}$.
This group is also an HNN\ extension. It can be defined by the presentation%
\[
G\rtimes_{\phi}%
\mathbb{Z}
=\left\langle G,t\mid t^{-1}gt=\phi\left(  g\right)  \text{ }\forall g\in
G\right\rangle
\]
Alternatively, if $G=\left\langle S\mid R\right\rangle $, then
\[
G\rtimes_{\phi}%
\mathbb{Z}
=\left\langle S,t\mid R,\ t^{-1}st=\phi\left(  s\right)  \text{ }\forall s\in
S\right\rangle
\]
Some authors use an alternative convention that uses relators $tgt^{-1}%
=\phi(g)$. We prefer the above presentation for geometric reasons that will
become apparent as we proceed. Those geometric reasons also explain why some
authors (\cite{BeFe92}, \cite{FeHa99}) refer to this group as a \emph{mapping
torus}. We reserve that terminology for the associated topological construction.

\subsection{Absolute retracts, group actions, and $\phi$-variant maps}


A locally compact separable metric space $X$ is an \emph{absolute neighborhood
retract (ANR)} if, whenever $X$ is embedded as a closed subset of a space $Y$,
some neighborhood of $X$ retracts onto $X$. A contractible ANR is called an
\emph{absolute retract (AR)}. Note that all ARs and ANRs in this paper are
assumed to be locally compact, separable, and metrizable.

\begin{definition}
Let $G$ be a group and $X$ a topological space [respectively AR]. We say that
$X$ is a $G$\emph{-space} [respectively $G$\emph{-AR}] when there is a
specified homomorphism $i:G\rightarrow\operatorname*{Homeo}(X)$. When no
confusion can arise, we write $g\cdot x$ to denote $i\left(  g\right)  \left(
x\right)  $ (omitting the notation $i$).
\end{definition}

\begin{definition}
The action associated to a $G$-space is \emph{cocompact} if the quotient space
$G\backslash X$ is compact, where $x\sim g\cdot x$ for all $x,g$. The action
is \emph{proper} if, for any compact set $K\subset X$, $\{g\in G\mid g\cdot
K\cap K\neq\emptyset\}$ is a finite set.
\end{definition}

Note that our definition of \textquotedblleft proper\textquotedblright%
\ matches the more general definition for topological groups whenever $G$ is
discrete---a condition that will be satisfied throughout this paper.

\begin{definition}
A $G$-space is called a $G$\emph{-complex} if it is a $CW$-complex and the
action is cellular. It is said to be a \emph{rigid} $G$-complex if the action
has the property that if $\sigma$ is a cell for which $g\cdot\sigma=\sigma$,
then $\sigma$ is point-wise fixed by $g$.
\end{definition}

We are also interested in maps between $G$-spaces which respect the action, or
which are compatible with a `twisted' version of the action induced by an automorphism.

\begin{definition}
\label{def:phivar} Let $X,Y$ be $G$-spaces. A map $f:X\rightarrow Y$ is
\emph{$G$-equivariant} (or just \emph{equivariant} when the group is
understood) if for all $x\in X$, $f(g\cdot x)=g\cdot f(x)$.

More generally, for $\phi\in\operatorname*{Aut}\left(  G\right)  $, we say $f$
is $\phi$\emph{-variant} if
\[
f(g\cdot x)=\phi(g)\cdot f(x)
\]
for all $g\in G$ and $x\in X$.
\end{definition}

\begin{remark}
Clearly a $G$-equivariant map is just a $\phi$-variant map where
$\phi=\operatorname*{id}_{G}$. Conversely, a $\phi$-variant map can be viewed
as a $G$-equivariant map, where the $G$-action on $Y$ has been twisted by
precomposing with $\phi$. For our purposes, $\phi$-variance is an extremely
useful concept.
\end{remark}


\subsection{$\mathcal{Z}$-sets and $\mathcal{Z}$%
-structures\label{Subsection: Z-sets and Z-structures}}


A collection $\mathcal{A}$ of subsets of a compact space $Y$ is a \emph{null
family} if, for any open cover $\mathcal{U}$ of $Y$, there exists a finite
subcollection $\mathcal{B}\subseteq\mathcal{A}$ such that for all
$A\in\ \mathcal{A}-\mathcal{B}$, there exists $U\in\mathcal{U}$ such that
$A\subseteq U$. By the Lebesgue Covering Theorem, if $Y$ is metrizable,
$\mathcal{A}$ is a null family if and only if, for any $\varepsilon>0$, there
exists a finite subcollection $\mathcal{B\subseteq A}$ such that
$\operatorname*{diam}(A)<\varepsilon$ for all $A\in\ \mathcal{A}-\mathcal{B}$.
This condition is independent of the metric chosen, but different metrics may
require different choices of $\mathcal{B}$.

A closed set $Z\subseteq Y$ is a \textbf{$\mathcal{Z}$}\emph{-set} if, there
exists a homotopy $\alpha:Y\times\lbrack0,1]\rightarrow Y$ such that
$\alpha_{0}=\operatorname*{id}_{Y}$ and $\alpha_{t}(Y)\subseteq Y-Z$ for all
$t>0$. In this case, we call $\alpha$ a \textbf{$\mathcal{Z}$}\emph{-set
homotopy. }A compactification $\overline{X}$ of a space $X$ is a
\textbf{$\mathcal{Z}$}\emph{-compactification }if $Z:=\overline{X}-X$ is a
\textbf{$\mathcal{Z}$}-set in $\overline{X}$. In that case, we call $Z$ a
\textbf{$\mathcal{Z}$}\emph{-boundary }for $X$ (keeping in mind this boundary
may not coincide with other boundaries, such as the visual boundary of a
CAT(0) or Gromov hyperbolic space).

By an application of Hanner's Theorem \cite{Han51}, a $\mathcal{Z}%
$-compactification of an ANR is always an ANR (see \cite{Tir11} for a detailed
discussion); hence, a \textbf{$\mathcal{Z}$}-compactification of an AR is an
AR. When $X$ is an AR, we can choose a \textbf{$\mathcal{Z}$}-set homotopy
which contracts $\overline{X}$ to a point $x_{0}\in X$, keeping $x_{0}$ fixed
throughout \cite[Lemma 1.10]{Tir11}. For a more detailed discussion of ANRs
and \textbf{$\mathcal{Z}$}-sets, see \cite{GuMo19}.


\begin{definition}
\label{DefineZStructure} A \textbf{$\mathcal{Z}$}\emph{-structure} on a group
$G$ is a pair of spaces $(\overline{X},Z)$ satisfying:

\begin{enumerate}
\item $\overline{X}$ is a compact AR,

\item $Z$ is a $\mathcal{Z}$-set in $\overline{X}$,

\item \label{Defn: Z-structure Item3}$X=\overline{X}-Z$ admits a proper,
cocompact action by $G$, and

\item (nullity condition) for any compact set $K\subseteq X$, the collection
$\{gK\ |\ g\in G\}$ of subsets of $X$ is a null family in $\overline{X}$.
(Informally, compact subsets of $X$ get small when translated toward $Z$.)
\end{enumerate}

When $(\overline{X},Z)$ is \textbf{$\mathcal{Z}$}-structure on $G$, we call
$Z$ a \textbf{$\mathcal{Z}$}\emph{-boundary} for $G$. A \textbf{$\mathcal{Z}$%
}-structure for which the $G$-action on $X$ extends to an action on
$\overline{X}$ is called an $E$\textbf{$\mathcal{Z}$}\emph{-structure.}
\end{definition}

\begin{fact}
\label{fact:isometric} Without loss of generality, we can require that the
$G$-action in item (\ref{Defn: Z-structure Item3}) to be geometric with
respect to some proper metric on $X$. See \cite[Remark 8]{GuMo19}.
\end{fact}

The question of which groups admit $(E)$\textbf{$\mathcal{Z}$}-structures is
very much open, but a variety of special cases are now understood, including
hyperbolic groups \cite{BeMe91}, CAT(0) groups, Baumslag-Solitar groups
\cite{GMT19} and \cite{GMS22}, systolic groups \cite{OsPr09}, and certain
relatively hyperbolic groups \cite{Dah03}. Work of Tirel \cite{Tir11}
demonstrates that this class is closed under direct and free products. The
latter of these results can also be obtained from Dahmani's theorem.


\subsection{Mapping cylinders, mapping tori, and mapping
telescopes\label{Subsection: Mapping cylinders, mapping tori, and mapping telescopes}%
}


The following definitions, notation, and background material will play a
primary role in our main constructions.

\begin{definition}
\label{Defn: Mapping cylinder, torus and telescope}Let $X$ be a space and
$f:X\rightarrow X$ a continuous map.

\begin{enumerate}
\item \label{Defn: Item 1 Mapping cylinder}The \emph{mapping cylinder of }$f$
based on the interval $\left[  a,b\right]  $ is the quotient space
$\mathcal{M}_{[a,b]}(f)=(X\times\lbrack a,b])\sqcup X/\sim$ where $\sim$ is
generated by the rule $(x,b)\sim f(x)$. For each $t\in\lbrack a,b)$, the
quotient map $q_{[a,b]}:(X\times\lbrack a,b])\sqcup X\rightarrow
\mathcal{M}_{[a,b]}(f)$ restricts to an embedding of $X\times
\{t\}\hookrightarrow\mathcal{M}_{[a,b]}(f)$ whose image will be denoted
$X_{t}$. The quotient map also restricts to an embedding on the disjoint copy
of $X$; we denote its image in $\mathcal{M}_{[a,b]}(f)$ by $X_{b}$ and refer
to it as the \emph{range end} of $\mathcal{M}_{[a,b]}(f)$. Similarly, $X_{a}$
is the \emph{domain end} of $\mathcal{M}_{[a,b]}(f)$.

\item \label{Defn: Item 2 Mapping torus}The \emph{mapping torus of }$f$ is the
quotient space $\operatorname*{Tor}_{f}(X)=X\times\lbrack0,1]/\sim$ where
$\sim$ is generated by the rule $(x,1)\sim(f(x),0)$.

\item \label{Defn: Item 3 Mapping telescope}The \emph{bi-infinite mapping
telescope} corresponding to $f$ is the space
\[
\operatorname*{Tel}\nolimits_{f}(X)=\dots\cup\mathcal{M}_{[-1,0]}%
(f)\cup\mathcal{M}_{[0,1]}(f)\cup\mathcal{M}_{[1,2]}(f)\cup\mathcal{M}%
_{[2,3]}(f)\dots
\]
Implicit in this notation is the identification of the range end of each
$\mathcal{M}_{[k-1,k]}(f)$ with the domain end of $\mathcal{M}_{[k,k+1]}(f)$,
both denoted by $X_{k}$. Using the notation from
(\ref{Defn: Item 1 Mapping cylinder}), there is a quotient map $\lambda
:\operatorname*{Tel}\nolimits_{f}(X)\rightarrow%
\mathbb{R}
$ such that $\lambda^{-1}(r)=X_{r}$ for each $r$.
\end{enumerate}
\end{definition}

\begin{remark}
The reader is warned that the \textquotedblleft direction\textquotedblright%
\ of our mapping cylinders and telescopes is the reverse of what is found in
some of the literature, such as \cite{Gui14} and \cite{Gui16}.
\end{remark}

\begin{definition}
\label{Defn: Suspension}The \emph{suspension}\textbf{ }of a space $Z$ is the
quotient space $SZ=Z\times\left[  -\infty,\infty\right]  /\sim$ where
$(z,\infty)\sim(z^{\prime},\infty)$ and $(z,-\infty)\sim(z^{\prime},-\infty)$
for all $z,z^{\prime}\in X$.
\end{definition}

\begin{notation}
\label{Notation: points in a telescope}We frequently find ourselves working
with a Cartesian product $X\times%
\mathbb{R}
$, a mapping telescope $\operatorname*{Tel}\nolimits_{f}(X)$, and a related
suspension $SZ$, all at the same time. To aid in distinguishing between points
of these spaces, we adopt the following notational conventions.

\begin{enumerate}
\item A point in $X\times%
\mathbb{R}
$ is represented in the usual way, as an ordered pair $\left(  x,r\right)  $
enclosed in ordinary parentheses.

\item A point in $\operatorname*{Tel}\nolimits_{f}(X)$ is represented by
$\lceil x,r \rceil$ when it is the equivalence class of $\left(  x,r\right)  $
in $\mathcal{M}_{[k,k+1]}(f)\subseteq\operatorname*{Tel}\nolimits_{f}(X)$ and
$k\leq r<k+1$. In particular, when a point projects to $k\in%
\mathbb{Z}
$ under $\lambda:\operatorname*{Tel}\nolimits_{f}(X)\rightarrow%
\mathbb{R}
$, it takes its coordinates from the \emph{domain end} of $\mathcal{M}%
_{[k,k+1]}(f)$. This gives a bijective correspondence between symbols
$\left\{  \left\lceil x,r\right\rceil \mid x\in X\text{ and }r\in%
\mathbb{R}
\right\}  $ and points of $\operatorname*{Tel}\nolimits_{f}(X)$. Under this
convention, a sequence $\left\{  \left\lceil x,k+\frac{1}{i+1}\right\rceil
\right\}  _{i=1}^{\infty}$ converges to $\left\lceil x,k\right\rceil $ while
$\left\{  \left\lceil x,k+\frac{i}{i+1}\right\rceil \right\}  _{i=1}^{\infty}$
converges to $\left\lceil f\left(  x\right)  ,k+1\right\rceil $.

\item A point in $SZ$ will be represented by $\left\langle z,r\right\rangle $
when it is the equivalence class of $\left(  z,r\right)  $. As such,
equivalence classes $Z\times\left\{  \infty\right\}  $ and $Z\times\left\{
-\infty\right\}  $ have non-unique representations which we sometimes
abbreviate to $\left\langle \infty\right\rangle $ and $\left\langle
-\infty\right\rangle $. In either case, the delimiters $\left\langle
\ ,\ \right\rangle $ indicates an element of a suspension.
\end{enumerate}
\end{notation}

For later use, we compile some basic facts about mapping cylinders, mapping
tori, mapping telescopes, and suspensions, tailored to our present needs.

\begin{lemma}
\label{Lemma: basic facts about cylinders and telescopes} Let $f:X\rightarrow
X$ be a proper self map of a locally compact, separable metrizable space. Then

\begin{enumerate}
\item \label{Item 1: Lemma on basic facts on cylinders and telescopes}%
$\mathcal{M}_{[a,b]}(f)$, $\operatorname*{Tor}_{f}(X)$, and
$\operatorname*{Tel}\nolimits_{f}(X)$ are locally compact, separable, and metrizable,

\item \label{Item 2: Lemma on basic facts on cylinders and telescopes}if $X$
is contractible, then $\mathcal{M}_{[a,b]}(f)$ and $\operatorname*{Tel}%
\nolimits_{f}(X)$ are contractible,

\item \label{Item 3: Lemma on basic facts on cylinders and telescopes}if $X$
is an ANR, then $\mathcal{M}_{[a,b]}(f)$, $\operatorname*{Tor}_{f}(X)$, and
$\operatorname*{Tel}\nolimits_{f}(X)$ are ANRs,

\item \label{Item 4: Lemma on basic facts on cylinders and telescopes}if $X$
is an AR, then $\mathcal{M}_{[a,b]}(f)$ and $\operatorname*{Tel}%
\nolimits_{f}(X)$ are ARs,

\item \label{Item 5: Lemma on basic facts on cylinders and telescopes}if $X$
is a locally finite CW complex and $f$ is a cellular map, then $\mathcal{M}%
_{[a,b]}(f)$, $\operatorname*{Tor}_{f}(X)$, and $\operatorname*{Tel}%
\nolimits_{f}(X)$ can be endowed with corresponding cell structures making
each a locally finite CW complex.
\end{enumerate}
\end{lemma}

\begin{proof}
Item \ref{Item 1: Lemma on basic facts on cylinders and telescopes} is an
exercise in general topology. In each case, the space in question is endowed
with the quotient topology induced by a map $q:Y\rightarrow Y/\sim$, where the
space $Y$ has all of the desired properties. By using the properness of $f$
one may view the quotient space as the result of an upper semicontinuous
decomposition of $Y$ in the sense of \cite{Dav86}. From there, the desired
conclusions can be deduced from results found in the first three sections of
that book.

Since $\mathcal{M}_{[a,b]}(f)$ strong deformation retracts onto its domain end
$X_{a}$, the first assertion of Item
\ref{Item 2: Lemma on basic facts on cylinders and telescopes} is clear. By
applying that deformation retraction inductively, one sees that any finite
subtelescope of $\operatorname*{Tel}\nolimits_{f}(X)$ is contractible. By
thickening these finite telescopes to contractible open subsets of
$\operatorname*{Tel}\nolimits_{f}(X)$, one may apply \cite{AnEd16} to deduce
the contractibility of $\operatorname*{Tel}\nolimits_{f}(X)$.

Items \ref{Item 3: Lemma on basic facts on cylinders and telescopes} and
\ref{Item 4: Lemma on basic facts on cylinders and telescopes} follow from
\cite[p.178]{Hu65}. Item
\ref{Item 5: Lemma on basic facts on cylinders and telescopes} is
standard---see, for example, \cite{FrPi90}.
\end{proof}

\begin{notation}
Finally, when working in a metric space $\left(  X,d\right)  $, we will let
$B_{d}\left(  x,\varepsilon\right)  $ and $B_{x}\left[  x,\varepsilon\right]
$ denote open and closed $\varepsilon$-balls centered at $x$.
\end{notation}



\section{Mapping tori and telescopes as classifying spaces for $G\rtimes
_{\phi}%
\mathbb{Z}
$}

\label{Section: Mapping tori and telescopes as classifying spaces}

The first step in placing a $\mathcal{Z}$-structure on any group is to find an
AR on which that group acts properly and cocompactly. In this paper, we most
frequently begin with a $\mathcal{Z}$-structure $(\overline{X},Z)$ on $G$ and
look to place one on $G\rtimes_{\phi}${$%
\mathbb{Z}
$}. As such, our first step is to construct an AR $Y$ on which $G\rtimes
_{\phi}${$%
\mathbb{Z}
$} acts properly and cocompactly. Due to basic covering space theory, this is
accomplished most easily and most intuitively when $G$ is torsion-free. For
that reason we deal only with the torsion-free case in this section. We will
return to the cases where $G$ has torsion in Section
\ref{sec: proof with torsion permitted}.

Throughout this section, prime symbols will be used for maps occurring
`downstairs', while their lifts to universal covers will be denoted without primes.

Beginning with a proper and cocompact $G$-action on an AR $X$, the assumption
that $G$ is torsion-free ensures that the quotient map $q:X\rightarrow
G\backslash X$ is a covering projection. Since the action is cocompact and $X$
is contractible, $G\backslash X$ is compact and aspherical. Furthermore, since
being an ANR is a local property, $G\backslash X$ is an ANR. We now use some
trickery to replace our generic AR $X$ with a locally finite contractible CW
complex (and a cellular $G$-action). This step is not strictly necessary for
much of what follows, but it leads to a more intuitive version of a
$(G\rtimes_{\phi}{%
\mathbb{Z}
})$-space $Y$. In addition, a CW structure is useful in some applications.

By a famous theorem of West \cite{Wes77}, $G\backslash X$ is homotopy
equivalent to a finite CW complex $K$, which is a $K(G,1)$ complex. The
universal cover of $K$, which we temporarily label as $X^{\ast}$, admits a
proper, cocompact, cellular $G$-action, and by a well-known boundary swapping
trick (see \cite{Bes96} or \cite{GuMo19}), we may attach $Z$ to $X^{\ast}$ to
obtain an alternative $\mathcal{Z}$-structure $(\overline{X^{\ast}},Z)$ on
$G$. For ease of notation we replace the $X$ with $X^{\ast}$ and omit the
star---assuming from now on that $X$ is this CW complex.

Without loss of generality, we may assume that the $2$-skeleton of $K$ is a
presentation $2$-complex for $G=\left\langle S\mid R\right\rangle $ with
vertex $v_{0}$. By asphericity, there is a cellular map $f^{\prime}%
:(K,v_{0})\rightarrow(K,v_{0})$ such that $f_{\ast}^{\prime}=\phi:G\rightarrow
G$. Since $\phi$ is an automorphism, $f^{\prime}$ is a homotopy equivalence.
The mapping torus $\operatorname*{Tor}_{f^{\prime}}(K)$ has fundamental group
$G\rtimes_{\phi}${$%
\mathbb{Z}
$}. In fact, if we give $\operatorname*{Tor}_{f^{\prime}}(K)$ the usual cell
structure consisting of $K$ together with a new $(k+1)$-cell for each $k$-cell
of $K$, then the $2$-skeleton of $\operatorname*{Tor}_{f^{\prime}}(K)$ is a
presentation $2$-complex for $G\rtimes_{\phi}{%
\mathbb{Z}
=}\left\langle S,\ t\mid R,\ t^{-1}st=\phi(s)\ \forall s\in S\right\rangle $.
The new $1$-cell (corresponding to the $0$-cell $v_{0}$) gives rise to $t$,
and each new $2$-cell (one for each $s\in S$) results in a relator
$t^{-1}st=\phi(s)$.

It is easy to see that the mapping telescope $\operatorname*{Tel}%
\nolimits_{f^{\prime}}(K)$ is an infinite cyclic cover of $\operatorname*{Tor}%
_{f^{\prime}}(K)$ with fundamental group corresponding to $G\trianglelefteq
G\rtimes_{\phi}${$%
\mathbb{Z}
$}. From there, the universal cover may be viewed as the mapping telescope
$\operatorname*{Tel}\nolimits_{f}(X)$ where $f:X\rightarrow X$ is a lift of
$f^{\prime}$. More specifically, let $x_{0}\in q^{-1}(v_{0})$ be a preferred
basepoint, where $q:X\rightarrow K$ is the covering projection, and choose $f$
taking $x_{0}$ to $x_{0}$. By lifting the CW structure of $\operatorname*{Tor}%
_{f^{\prime}}(K)$ to $\operatorname*{Tel}\nolimits_{f}(X)$, we can realize the
Cayley graph and Cayley $2$-complex of $G\rtimes_{\phi}${$%
\mathbb{Z}
$} as the $1$- and $2$-skeleta of $\operatorname*{Tel}\nolimits_{f}(X)$. By
Lemma \ref{Lemma: basic facts about cylinders and telescopes},
$\operatorname*{Tel}\nolimits_{f}(X)$ is contractible, and since locally
finite CW complexes are ANRs, $\operatorname*{Tel}\nolimits_{f}(X)$ is an AR.
This is the space $Y$ we set out to construct.

Before proceeding, we pause to examine the action of $G\rtimes_{ \phi}${$%
\mathbb{Z}
$} on $Y=\operatorname*{Tel}\nolimits_{f}(X)$. First we focus on the
$1$-skeleton, which is $\operatorname*{Cay}(G\rtimes_{\phi}{%
\mathbb{Z}
},S\cup\left\{  t\right\}  )$. For each integer $i$, $X_{i}$ is a copy of $X$,
so its $1$-skeleton is a copy of $\operatorname*{Cay}(G,S)$. By placing the
identity element of $G\rtimes_{\phi}${$%
\mathbb{Z}
$} at $x_{0}\in X_{0}$ we identify $X_{0}^{(1)}$ as the subgraph of
$\operatorname*{Cay}(G\rtimes_{\phi}{%
\mathbb{Z}
},S\cup\left\{  t\right\}  )$ corresponding to $G\trianglelefteq
G\rtimes_{\phi}${$%
\mathbb{Z}
$}. All edges of $\operatorname*{Cay}(G\rtimes_{\phi}{%
\mathbb{Z}
},S\cup\left\{  t\right\}  )$ not in some $X_{i}^{(1)}$ are labeled by $t$ and
are oriented in the positive $%
\mathbb{R}
$-direction with respect to the projection $\lambda:$ $\operatorname*{Tel}%
\nolimits_{f}(X)\rightarrow%
\mathbb{R}
$. As a homeomorphism of $Y$, $t$ maps each $X_{i}$ onto $X_{i+1}$ via the
\textquotedblleft identity\textquotedblright. As such, each subcomplex
$X_{i}^{(1)}$ corresponds to the left coset $t^{i}G$, and each vertex
$t^{i}g\in t^{i}G$ is taken to its counterpart $t^{i+1}g\in t^{i}G$. Left
multiplication by an element of $G$ is more interesting. For each vertex $g\in
X_{0}^{(1)}$, the outgoing $t$-edge ends at $gt=t\phi(g)$. So, if $ta\in
tG=X_{1}^{(0)}$, it is the terminal vertex of a $t$-edge emanating from
$\phi^{-1}(a)$. Left-multiplication by $g$ moves that edge to one emanating
from $g\phi^{-1}(a)$ and ending at $\phi(g\phi^{-1}(a))=\phi(g)a$. In other
words, from the perspective of $X_{1}^{(1)}$, left multiplication by $g$ is
replaced by left multiplication by $\phi(g)$. More generally,
left-multiplication by $g$ shows up in $X_{i}^{(1)}$ as left-multiplication by
$\phi^{i}(g)$.

\begin{figure}[th]
\begin{center}
\includegraphics[scale=0.5]{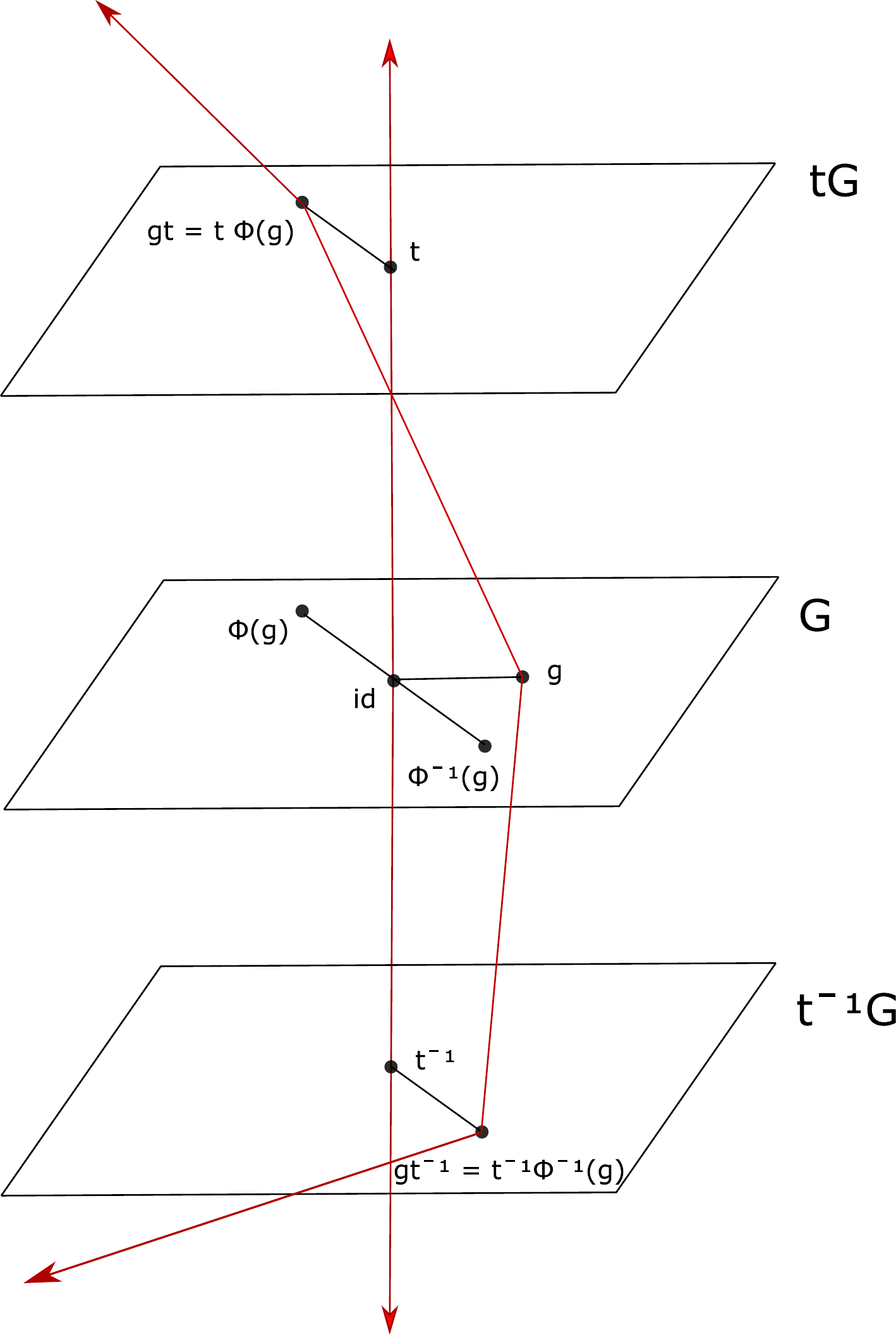} \label{fig:cayley}
\end{center}
\caption{A snapshot of the Cayley graph of a semidirect product with the
integers. Red edges are labeled by t, black by elements of G.}%
\label{fig:semidirect}%
\end{figure}

As a homeomorphism of $Y$, $\left.  g\right\vert _{X_{i}}=\phi^{i}( g) $, with
mapping cylinder lines being taken to mapping cylinder lines.

For another interesting perspective, note that the left coset $g\left\langle
t\right\rangle $ of the infinite cyclic subgroup generated by $t$ is the line
in the Cayley graph whose vertex set is the bi-infinite sequence $\left\{
\phi^{i}( g) \right\}  _{i=-\infty}^{\infty}$, and whose edges are all labeled
by $t$.

Let us now consider the geometry of $G\rtimes_{\phi}${$%
\mathbb{Z}
$}, again focusing on $\operatorname*{Cay}(G\rtimes_{\phi}{%
\mathbb{Z}
},S\cup\left\{  t\right\}  )=Y^{(1)}$. If we ignore the $t$-edges, this Cayley
graph is identical to that of the direct product $G\times${$%
\mathbb{Z}
$}; we have a copy $X_{i}^{(1)}$ of $\operatorname*{Cay}(G,S)$ over each
integer $i\in%
\mathbb{R}
$. But, unlike the Cayley graph of the direct product, the edges between
vertices of $X_{i}^{(1)}$ and $X_{i+1}^{(1)}$ are not \textquotedblleft
vertical\textquotedblright. Instead the $t$-edge emanating from a vertex
$t^{i}a$ has terminal point $t^{i+1}\phi(a)$ where $a$ and $\phi(a)$ can be
far apart as elements of $G$.
In addition, elements $t^{i}a$ and $t^{i}b$ which are far apart in
$X_{i}^{(1)}$ can be close together in $\operatorname*{Cay}(G\rtimes_{\phi}{%
\mathbb{Z}
},S\cup\left\{  t\right\}  )$, due to shortcuts made available by the
$t$-edges. In particular, depending on the automorphism $\phi$, the geometry
of $G\rtimes_{\phi}${$%
\mathbb{Z}
$} can be very different from that of $G\times%
\mathbb{Z}
$. See Figure~\ref{fig:semidirect}.

We have already noted that $Y$ is contractible, but we can be more precise. We
can build a proper homotopy equivalence between $Y$ and $X\times%
\mathbb{R}
$.

Since $f^{\prime}:(K,v_{0})\rightarrow(K,v_{0})$ is a homotopy equivalence,
there exists a cellular map $h^{\prime}:(K,v_{0})\rightarrow(K,v_{0})$ and
cellular homotopies $A^{\prime},B^{\prime}:K\times\left[  0,1\right]
\rightarrow K$ such that $A_{0}^{\prime}=\operatorname*{id}_{K}$,
$A_{1}^{\prime}=h^{\prime}f^{\prime}$, $B_{0}^{\prime}=\operatorname*{id}_{K}%
$, and $B_{1}=f^{\prime}h^{\prime}$. Since these maps are proper (all spaces
being compact), so are their lifts (see \cite[Theorem 10.1.23]{Geo08}). Hence
we obtain proper maps $f,h:(X,x_{0})\rightarrow(X,x_{0})$ and proper
homotopies $A,B:X\times\left[  0,1\right]  \rightarrow X$ such that
$A_{0}=\operatorname*{id}_{X}$, $A_{1}=hf$, $B_{0}=\operatorname*{id}_{X}$,
and $B_{1}=fh$. Moreover, $f$ is $\phi$-variant under the standard covering
action on $(X,x_{0})$ and $h$ is $\phi^{-1}$-variant.

To obtain the desired proper homotopy equivalence $v:Y\rightarrow X\times%
\mathbb{R}
$, we will first construct a proper homotopy equivalence $v^{\prime
}:\operatorname*{Tel}_{f^{\prime}}(K)\rightarrow K\times%
\mathbb{R}
$. That map will be lifted to universal covers to obtain $v$.

The building blocks of $v^{\prime}$ will be maps $v_{n}^{\prime}%
:\mathcal{M}_{[n,n+1]}(f^{\prime})\rightarrow K\times\left[  n,n+1\right]  $.
For the purposes of assuring continuity, recall that $\mathcal{M}%
_{[n,n+1]}(f^{\prime})=K\times\lbrack n,n+1]\sqcup K/\sim$ where $(x,n+1)\sim
f^{\prime}(x)$. We will first describe continuous maps of $K\times\lbrack
n,n+1]\sqcup K$ into $K\times\lbrack n,n+1]$, then observe that they respect
$\sim$ thereby inducing continuous maps on $\mathcal{M}_{[n,n+1]}(f^{\prime})$.

For integers $n\geq0$, $v_{n}^{\prime}$ is induced by the map%
\[%
\begin{tabular}
[c]{cccc}%
$\left(  x,r\right)  $ & $\mapsto$ & $\left(  h^{\prime n}\left(
A_{r-n}^{\prime}\left(  x\right)  \right)  ,r\right)  $ & for$\text{ }\left(
x,r\right)  \in K\times\lbrack n,n+1]$\\
$x$ & $\mapsto$ & $\left(  h^{\prime\left(  n+1\right)  }\left(  x\right)
,n+1\right)  $ & for$\text{ }x$ in range copy of $K$%
\end{tabular}
\
\]

\noindent Since each $\left(  x,n+1\right)  \in K\times\lbrack n,n+1]$ is sent
to $\left(  h^{\prime n}\left(  h^{\prime}f^{\prime}\left(  x\right)  \right)
,n+1\right)  =\left(  h^{\prime\left(  n+1\right)  }\left(  f^{\prime}\left(
x\right)  \right)  ,n+1\right)  $, which is also where it sends the point
$f^{\prime}\left(  x\right)  \in K$, we get a well-defined continuous map
$v_{n}^{\prime}:\mathcal{M}_{[n,n+1]}(f^{\prime})\rightarrow K\times\left[
n,n+1\right]  $. For integers $n<0$, use the following similar (but simpler)
rule.
\[%
\begin{tabular}
[c]{cccc}%
$\left(  x,r\right)  $ & $\mapsto$ & $(f^{\prime\left\vert n\right\vert
}(x),r)$ & for$\text{ }\left(  x,r\right)  \in K\times\lbrack n,n+1]$\\
$x$ & $\mapsto$ & $f^{\prime\left\vert n\right\vert -1}(x)$ & for$\text{ }x$
in range copy of $K$%
\end{tabular}
\ \
\]

Notice that, for each integer $n$, $v_{n-1}^{\prime}$ and $v_{n}^{\prime}$
agree on $K_{n}$, so these maps can be pasted together to obtain $v^{\prime
}:\operatorname*{Tel}_{f^{\prime}}(K)\rightarrow K\times%
\mathbb{R}
$. An argument similar to the one in \cite{Gui14} , shows that $v^{\prime}$ is
a proper homotopy equivalence with a proper homotopy inverse $u^{\prime
}:K\times%
\mathbb{R}
\rightarrow\operatorname*{Tel}_{f^{\prime}}(K)$ which takes $K\times\left\{
n\right\}  $ into $K_{n}$ and $K\times\left[  n,n+1\right]  $ into
$\mathcal{M}_{[n,n+1]}(f^{\prime})$ for each integer $n$. We say that
$v^{\prime}$ is \emph{level-preserving }and $u^{\prime}$ is \emph{nearly
level-preserving. }(A level-preserving version of $u^{\prime}$ can be provided
if need arises.)

For use later in this paper, we simplify the description of $v^{\prime}$. We
will make use of the \emph{floor function} for real numbers $\left\lfloor
r\right\rfloor $, and define $\mathring{r}=r-\left\lfloor r\right\rfloor
\in\left[  0,1\right]  $. Using the notation established in Section
\ref{Subsection: Mapping cylinders, mapping tori, and mapping telescopes} we have:%

\begin{equation}
v^{\prime}\left\lceil x,r\right\rceil =\left\{
\begin{tabular}
[c]{cc}%
$\left(  h^{\prime\left\lfloor r\right\rfloor }\left(  A_{\mathring{r}%
}^{\prime}\left(  x\right)  \right)  ,r\right)  $ & $\text{if }r\geq0$\\
$\left(  f^{\prime\left\vert \left\lfloor r\right\rfloor \right\vert }\left(
x\right)  ,r\right)  $ & $\text{if }r<0$%
\end{tabular}
\ \ \ \right.  \label{Definition of v'}%
\end{equation}
By associating the fundamental groups of $\operatorname*{Tel}_{f^{\prime}}(K)$
and $K\times%
\mathbb{R}
$ with $G$ via the inclusions of $K\times\left\{  0\right\}  $, and noting
that $v^{\prime}$ restricts to the identity on these subspaces, we may view
$v^{\prime}$ as inducing $\operatorname*{id}_{G}$ on fundamental groups. As
such, the lift to universal covers $v:\operatorname*{Tel}_{f}(X)\rightarrow
X\times%
\mathbb{R}
$ is a $G$-equivariant proper homotopy equivalence, where $X$ is the universal
cover of $K$ and $f:X\rightarrow X$ is the lift of $f^{\prime}$. This map is
level-preserving with nearly level-preserving $G$-equivariant proper homotopy
inverse $u:X\times%
\mathbb{R}
\rightarrow\operatorname*{Tel}_{f}(X)$. Letting $h:X\rightarrow X$ and
$A:X\times\left[  0,1\right]  \rightarrow X$ be the appropriately chosen lifts
of $f^{\prime}$ and $A^{\prime}$, and adapting the conventions used above, we
may specify $v$ by the formula%
\begin{equation}
v\left\lceil x,r\right\rceil =\left\{
\begin{tabular}
[c]{cc}%
$\left(  h^{\left\lfloor r\right\rfloor }\left(  A_{\mathring{r}}\left(
x\right)  \right)  ,r\right)  $ & $\text{if }r\geq0$\\
$\left(  f^{\left\vert \left\lfloor r\right\rfloor \right\vert }\left(
x\right)  ,r\right)  $ & $\text{if }r<0$%
\end{tabular}
\ \ \ \right.  \label{Definition of v}%
\end{equation}
For future use, let $H^{\prime}:\operatorname*{Tel}_{f^{\prime}}%
(K)\times\left[  0,1\right]  \rightarrow\operatorname*{Tel}_{f^{\prime}}(K)$
and $J^{\prime}:K\times%
\mathbb{R}
\rightarrow X$ be the near-level preserving proper homotopies $u^{\prime}\circ
v^{\prime}\overset{H^{\prime}}{\simeq}\operatorname*{id}_{\operatorname*{Tel}%
_{f^{\prime}}(K)}$ and $v^{\prime}\circ u^{\prime}\overset{J^{\prime}}{\simeq
}\operatorname*{id}_{K}$ promised above, and let $H$ and $J$ be their
$G$-equivariant, near level-preserving lifts $u\circ v\overset{H}{\simeq
}\operatorname*{id}_{\operatorname*{Tel}_{f}(X)}$ and $v\circ
u\overset{J}{\simeq}\operatorname*{id}_{X}$.

\begin{remark}
A benefit of constructing $u,v,H$, and $J$ as lifts of maps between
$\operatorname*{Tel}_{f^{\prime}}(K)$ and $K\times%
\mathbb{R}
$ (and related spaces) is that they are $G$-equivariant. One might ask: Why
not start even lower, i.e., begin with homotopy equivalences between
$\operatorname*{Tor}_{f^{\prime}}(K)$ and $K\times\mathbb{S}^{1}$, so as to
end up with $(G\rtimes_{\phi}%
\mathbb{Z}
)$-equivariant maps? In all but the simplest cases, that is impossible since
$\operatorname*{Tor}_{f^{\prime}}(K)$ and $K\times\mathbb{S}^{1}$ have
non-isomorphic fundamental groups. The spaces under consideration become
homotopy equivalent only after the \textquotedblleft$t$%
-factor\textquotedblright\ is unfurled.
\end{remark}

For later reference, we provide a summary of the many spaces and maps
introduced in this section.%

\[%
\begin{tabular}
[c]{|l|l|}\hline
\textbf{Spaces\medskip} & \\\hline
$K$ & a finite $K(G,1)$ complex\\\hline
$\operatorname*{Tor}_{f^{\prime}}(K)$ & a finite $K(G\rtimes_{\phi}%
\mathbb{Z}
)$ complex\\\hline
$\operatorname*{Tel}_{f^{\prime}}(K)$ & an infinite cyclic cover of
$\operatorname*{Tor}_{f^{\prime}}(K)$\\\hline
$X$ & the universal cover of $K$\\\hline
$Y=\operatorname*{Tel}\nolimits_{f}(X)$ & universal cover of
$\operatorname*{Tor}_{f^{\prime}}(K)$ and $\operatorname*{Tel}_{f^{\prime}%
}(K)$\\\hline
$SZ$ & The suspension of a space $Z$\\\hline
\end{tabular}
\ \ \
\]%
\[
\
\begin{tabular}
[c]{|l|l|}\hline
\textbf{Maps\medskip} & \\\hline
$\phi:G\rightarrow G$ & a group isomorphism with inverse $\overleftarrow{\phi
}$\\\hline
$f^{\prime}:(K,v_{0})\rightarrow(K,v_{0})$ & cellular map inducing $\phi$ on
fundamental groups\\\hline
$h^{\prime}:(K,v_{0})\rightarrow(K,v_{0})$ & a homotopy inverse for
$f$\\\hline
$A^{\prime}$ \ and \ $B^{\prime}$ & homotopies: $A^{\prime}:h^{\prime}\circ
f^{\prime}\simeq\operatorname*{id}_{K}$; \ $B^{\prime}:f^{\prime}\circ
h^{\prime}\simeq\operatorname*{id}_{K}$\\\hline
$v^{\prime}:\operatorname*{Tel}_{f^{\prime}}(K)\rightarrow K\times%
\mathbb{R}
$ & a level-preserving proper homotopy equivalence\\\hline
$u^{\prime}:K\times%
\mathbb{R}
\rightarrow\operatorname*{Tel}_{f^{\prime}}(K)$ & a nearly level-preserving
proper homotopy inverse for $v^{\prime}$\\\hline
$H^{\prime}$ \ and $\ J^{\prime}$ & homotopies; $H^{\prime}:u^{\prime}\circ
v^{\prime}\simeq\operatorname*{id}_{\operatorname*{Tel}_{f^{\prime}}(K)}$;
$J^{\prime}:v^{\prime}\circ u^{\prime}\simeq\operatorname*{id}_{K\times%
\mathbb{R}
}$\\\hline
$f,h:(X,x_{0})\rightarrow(X,x_{0})$ & lifts of $f^{\prime}$ and $h^{\prime}%
$which are $\phi$-variant\\
& and $\phi^{-1}$-variant, respectively\\\hline
$A$ \ and $\ B$ & lifts of $A^{\prime}$ and $B^{\prime}$; $\ G$-equivariant
homotopies\\
& $A:$ $h\circ f\simeq\operatorname*{id}_{X}$; $\ B:f\circ h\simeq
\operatorname*{id}_{X}$\\\hline
$v:Y\rightarrow X\times%
\mathbb{R}
$ & lift of $v$; a level-preserving\\
& $G$-equivariant proper homotopy equivalence\\\hline
$u:X\times%
\mathbb{R}
\rightarrow Y$ & lift of $u$; a $G$-equivariant proper homotopy inverse for
$v$\\\hline
$H$ \ and $\ J$ & lifts of $H^{\prime}$ and $J^{\prime}$; $\ G$-equivariant
homotopies\\
& $H:$ $u\circ v\simeq\operatorname*{id}_{Y}$; $J:v\circ u\simeq
\operatorname*{id}_{X\times%
\mathbb{R}
}$\\\hline
\end{tabular}
\ \
\]
\bigskip

It is now possible to give a rough outline of the proof of Theorem
\ref{Theorem A}.

\begin{enumerate}
\item Using the hypothesis that $G$ admits a $\mathcal{Z}$-structure $(
\overline{X},Z) $ and $\phi:G\rightarrow G$ is an isomorphism, obtain a nice
($\phi$-variant) continuous quasi-isometry $f:X\rightarrow X$ and use this map
to construct an AR, $Y=\operatorname*{Tel}\nolimits_{f}(X)$, on which
$G\rtimes_{\phi}${$%
\mathbb{Z}
$} acts properly and cocompactly.

\item Build a carefully controlled proper homotopy equivalence $v:Y\rightarrow
X\times${$%
\mathbb{R}
$}.

\item Using the $\mathcal{Z}$-compactifiability of $X$ we may $\mathcal{Z}%
$-compactify $X\times${$%
\mathbb{R}
$} by the addition of $SZ$. This compactification is not unique. Delicate
techniques from \cite{Tir11}, allow us to choose a compactification for which
specific collections of compact subsets of $X\times${$%
\mathbb{R}
$} become null families in $\overline{X\times%
\mathbb{R}
}$. In particular we are interested in the $v$-images of $(G\rtimes_{\phi}${$%
\mathbb{Z}
$}${\mathbb{)}}$-translates of compact subsets of $Y$.

\item Finally, we use \textquotedblleft boundary swapping\textquotedblright%
\ techniques developed in \cite{GuMo19} to pull back the above boundary onto
$Y$. Additional controls must be built into Step 3 to ensure that
$\overline{Y}=Y\sqcup SZ$ is an AR and that the nullity condition is satisfied.
\end{enumerate}

When successful with the above, we can also ask whether the $(G\rtimes_{\phi}%
${$%
\mathbb{Z}
$}${\mathbb{)}}$-action on $Y$ can be extended to $\overline{Y}$. That is the
topic of Section \ref{Section: EZ-structures}.


\section{A controlled $\mathcal{Z}$-compactification of $X\times%
\mathbb{R}
\label{sec:controlled}$}


Our proof of Theorem \ref{Theorem A} requires a $\mathcal{Z}$-compactification
of the $(G\rtimes_{\phi}{%
\mathbb{Z}
})$-space $Y=\operatorname*{Tel}_{f}(X)$ that satisfies the nullity condition
of Definition \ref{DefineZStructure}. The strategy is indirect. We will first
$\mathcal{Z}$-compactify $X\times%
\mathbb{R}
$, then use the map $v:Y\rightarrow X\times%
\mathbb{R}
$, constructed above, to swap the boundary onto $Y$ using a process like the
one described in \cite{GuMo19}. Unfortunately, our setup does not exactly
match the hypotheses found there--- $v$ is not a coarse equivalence and the
homotopies $H$ and $J$ are not bounded. To compensate, we show how to impose
some extreme controls on the $\mathcal{Z}$-compactification of $X\times%
\mathbb{R}
$ which can be adjusted to ensure that the boundary swapping procedure succeeds.

\begin{definition}
A \emph{controlled compactification} of a proper metric space $(Y,d)$ is a
compactification $\overline{Y}$ satisfying the following property:

For every $R>0$ and every open cover $\mathcal{U}$ of $\overline{Y}$, there is
a compact set $C\subset Y$ so that if $A\subset Y\setminus C$ and
$\operatorname*{diam}_{d}(A)<R$, then $A\subset U$ for some $U\in\mathcal{U}$.

A \emph{controlled }$\mathcal{Z}$\emph{-compactification }is one that is
simultaneously a controlled compactification and a $\mathcal{Z}$-compactification
\end{definition}

Whenever $\left(  \overline{Y},Z\right)  $ is a $\mathcal{Z}$-structure for a
group $G$ and $d$ is a metric under which the corresponding $G$-action on $Y$
is by isometries, $\overline{Y}$ is a controlled $\mathcal{Z}$%
-compactification of $\left(  Y,d\right)  $ (see Lemma~6.4 in \cite{GuMo19}).
Other examples (not requiring a group action) include the addition of the
Gromov boundary to a proper $\delta$-hyperbolic space or the visual boundary
to a proper CAT(0) space. For the remainder of this section, we set aside
group actions and focus on the following metric/topological goal.

\begin{theorem}
\label{Theorem: Controlled Z-compactification of X x R}Let $\overline
{X}=X\sqcup Z$ be a controlled $\mathcal{Z}$-compactification of a
contractible proper metric space $(X,d)$ and let $\eta:[0,\infty
)\rightarrow\lbrack0,\infty)$ be an arbitrary monotone increasing function
with $\lim_{r\rightarrow\infty}\eta(r)=\infty$. Then there is a $\mathcal{Z}%
$-compactification $\overline{X\times%
\mathbb{R}
}$ of $X\times%
\mathbb{R}
$ with boundary $SZ$ satisfying the following control condition:

\noindent(\ddag)$\quad$For each open cover $\mathcal{U}$ of $\overline
{X\times{%
\mathbb{R}
}}$ and $k\in%
\mathbb{Z}
$, there exists a compact set $Q\times\left[  -N,N\right]  \subseteq X\times%
\mathbb{R}
$, such that every set $B_{d}[x,\eta(\left\vert k\right\vert )]\times\lbrack
k,k+1]$ lying outside $Q\times\left[  -N,N\right]  $ is contained in some
$U\in\mathcal{U}.\medskip$
\end{theorem}

Observe that it will suffice to demonstrate this result under the assumption
that $\eta$ is a continuous function. Indeed, we may always replace an
arbitrary such function by a continuous function, which is also monotone
increasing and is larger than or equal to our original function. One benefit
of continuity is that $\eta$ surjects onto $[\eta(0),\infty)$.

The first order of business is to place an appropriate topology on $X\times%
\mathbb{R}
\sqcup SZ$. The following list of (semi-arbitrary) choices will be used:

\begin{enumerate}
\item Choose metrics $\overline{d}$ for $\overline{X}$ and $\overline{\sigma}$
for $\left[  -\infty,\infty\right]  =%
\mathbb{R}
\cup\left\{  \pm\infty\right\}  $.

\item Fix a basepoint $x_{0}\in X$ and choose a monotone increasing continuous
function $\lambda:[0,\infty)\rightarrow\lbrack0,\infty)$ with the property
that, for every $s>0$:

\begin{enumerate}
\item $\overline{d}( x,Z) \,<1/s$ for all $x\in X-B_{d}\left[  x_{0},\lambda(
s) \right]  $, and

\item every ball $B_{d}\left[  x,s\right]  $ in $( X,d) $ lying outside
$B\left[  x_{0},\lambda( s) \right]  $ has diameter $\leq1/s$ in $(
\overline{X},\overline{d}) $.
\end{enumerate}

\item \label{Item: Defn of psi}Let $\psi:[0,\infty)\rightarrow\lbrack
0,\infty)$ be a continuous monotone increasing function satisfying:

\begin{enumerate}
\item $\psi(s)\geq\max\left\{  \eta(s),\lambda( s) \right\}  $ , and

\item $\psi(s+1)\geq3\psi(s)$ for all $s\geq0$
\end{enumerate}

\item Define $p:X\rightarrow\lbrack0,\infty)$ by
\begin{equation}
p(x)=\log(\psi^{-1}(d(x,x_{0})+\psi(0))+1) \label{Defn: p(x)}%
\end{equation}

\item Finally, we arrive at the \emph{slope function} $\mu:X\times%
\mathbb{R}
\rightarrow\left[  -\infty,\infty\right]  $ defined by%
\begin{equation}
\mu(x,r)=\left\{
\begin{tabular}
[c]{ll}%
$\frac{r}{p(x)}$ & \quad if $p(x)>0$\\
$\infty$ & \quad if $p(x)=0$ and $r\geq0$\\
$-\infty$ & \quad if $p(x)=0$ and $r<0$%
\end{tabular}
\ \ \right.  \label{Defn: slope function}%
\end{equation}

\end{enumerate}

\begin{notation}
In what follows, an undecorated $x$ will indicate a point in $X$, while $z$
will denote a point in $Z$. Open and closed ball notation, $B_{d}(x,r)$ and
$B_{d}[x,r]$, will be used primarily for open and closed balls in $X$.

To differentiate points of $X\times${$%
\mathbb{R}
$} from those of $SZ $, we will use $(x,r)$ for the former and $\langle\bar
{x},\mu\rangle$ for the latter (recall our convention that $\mu\in\left[
-\infty,\infty\right]  $.) Equivalence classes of $\langle\bar{x}%
,-\infty\rangle$ and $\langle\bar{x},\infty\rangle$ will often (but not
always) be abbreviated to $\langle-\infty\rangle$ and $\langle\infty\rangle$, respectively.
\end{notation}

\begin{definition}
[The topology on $\overline{X\times{%
\mathbb{R}
}}$]\label{DefineProductBoundaryTopology} Let $\overline{X\times{%
\mathbb{R}
}}=(X\times${$%
\mathbb{R}
)$}$\sqcup SZ$ with the topology generated by all open subsets of $X\times${$%
\mathbb{R}
$} together with sets of the form:

\begin{itemize}
\item[i)] For $\langle z,\mu\rangle\in SX$ with $\mu\neq\pm\infty$ and
$\varepsilon>0$,
\[
U(\langle z,\mu\rangle,\varepsilon)=\{(x,r)\ |\ \overline{d}(x,z)<\varepsilon
,\ |\mu(x,r)-\mu|<\varepsilon\}\cup\{\langle z^{\prime},\mu^{\prime}%
\rangle\ |\ \overline{d}(z^{\prime},z)<\varepsilon,\ |\mu^{\prime}%
-\mu|<\varepsilon\}
\]

\item[ii)] For each $\varepsilon>0$,%
\begin{align*}
U(\langle\infty\rangle,\varepsilon)  &  =\{(x,r)|\ r>\frac{1}{\varepsilon
},\ \mu(x,r)>\frac{1}{\varepsilon}\}\cup\{\langle z^{\prime},\mu^{\prime
}\rangle|\ \mu^{\prime}>\frac{1}{\varepsilon}\}\text{, and\medskip}\\
U(\langle-\infty\rangle,\varepsilon)  &  =\{(x,r)\ |\ r<\frac{-1}{\varepsilon
},\ \mu(x,r)<\frac{-1}{\varepsilon}\}\cup\{\langle z^{\prime},\mu^{\prime
}\rangle\ |\ \mu^{\prime}<\frac{-1}{\varepsilon}\}
\end{align*}

\end{itemize}
\end{definition}

Proofs of the following two lemmas are nearly identical to their analogs in
\cite{Tir11}.

\begin{lemma}
\label{ProductWithJoinIsACompactification} $\overline{X\times{%
\mathbb{R}
}}$ is a compactification of $X\times${$%
\mathbb{R}
$.}
\end{lemma}

\begin{lemma}
\label{LebesgueNumberOfBoundary} For any open cover $\mathcal{U}$ of
$\overline{X\times{%
\mathbb{R}
}}$ there exists $\delta>0$ such that for each $\langle z,\mu\rangle\in SZ$,
there is an element of $\mathcal{U}$ containing $U(\langle z,\mu\rangle
,\delta)$
\end{lemma}

\begin{remark}
We allow for the possibility that $\mu=\pm\infty$, i.e. $\langle z,\mu
\rangle=\left\langle \pm\infty\right\rangle $, in Lemma
\ref{LebesgueNumberOfBoundary}; so $U(\langle z,\mu\rangle,\varepsilon)$ can
denote a basic open set of type (i) or (ii).
\end{remark}

For the rest of this section, fix an open cover $\mathcal{U}$ of
$\overline{X\times{%
\mathbb{R}
}}$, and let $\delta$ be the value promised by Lemma
\ref{LebesgueNumberOfBoundary}.

Fix a function $\eta$ as in the statement of Theorem
\ref{Theorem: Controlled Z-compactification of X x R}. Then the following
three propositions verify the control condition ($\ddagger$). The goal is to
find a compact set $Q\times\left[  -N,N\right]  \subseteq X\times%
\mathbb{R}
$ such that all sets of the form $B_{d}\left[  x,\eta(|k|)\right]
\times\lbrack k,k+1]$ which lie outside $Q\times\left[  -N,N\right]  $ are
contained in some basic open set $U(\langle z,\mu\rangle,\delta)$.

\begin{proposition}
\label{prop:1of3} \label{JxPj} For each compact set of the form $Q=B_{d}%
[x_{0},\eta(M)]$ and $\delta>0$, there exists $N_{Q}\in%
\mathbb{N}
$ such that if $(B_{d}[x,\eta(|k|)]\times\lbrack k,k+1])\cap(Q\times\left[
-N_{Q},N_{Q}\right]  )=\varnothing$ and $Q\cap B_{d}[x,\eta(|k|)]\neq
\varnothing$, then $B_{d}[x,\eta(|k|)]\times\lbrack k,k+1])\subseteq
U(\langle\pm\infty\rangle,\delta)$.
\end{proposition}

\begin{figure}[th]
\begin{center}
\includegraphics[scale=0.5]{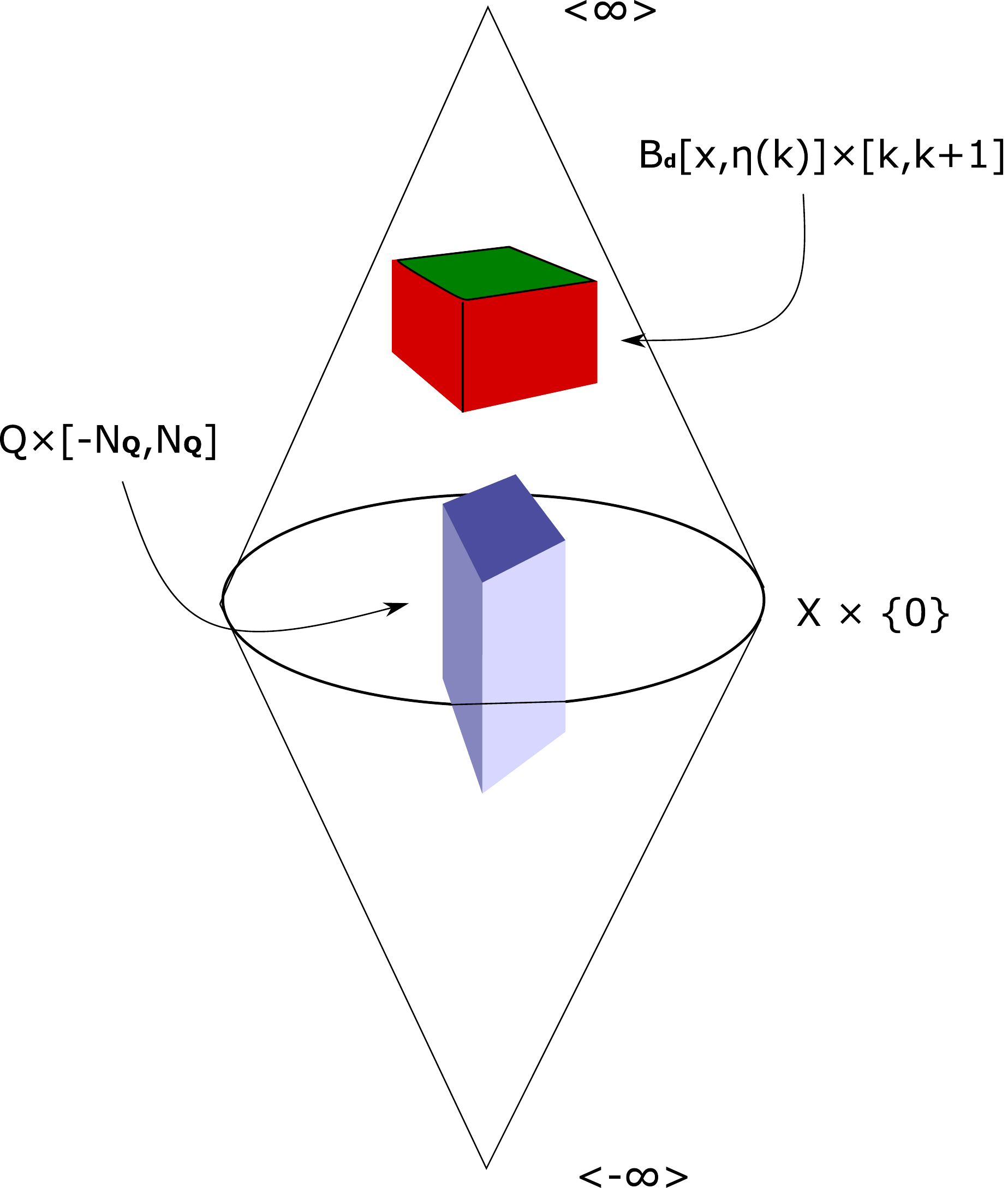} \label{fig:boxes}
\end{center}
\caption{A set meeting the criterion of the Proposition~\ref{prop:1of3}.}%
\label{fig:Qbox}%
\end{figure}

\begin{proof}
Choose $N_{Q}\geq M\text{ so large that }\frac{N_{Q}}{\log(N_{Q}+2)}>\frac
{1}{\delta}$. Let $k$ be any value such that $Q\cap B_{d}[x,\eta
(k)]\neq\varnothing$ and observe that it must be true that $|k|>N_{Q}$. Assume
first that $k>N_{Q}$. Then%
\begin{align*}
\min\{\mu(x^{\prime},r^{\prime}) &  \mid(x^{\prime},r^{\prime})\in
B(x,\eta(k))\times\lbrack k,k+1]\}\\
&  \geq\frac{\min\{r^{\prime}\mid r^{\prime}\in\lbrack k,k+1]\}}%
{\max\{p(x^{\prime})\mid x^{\prime}\in B(x,\eta(k))\}}\\
&  \geq\frac{k}{\log(\psi^{-1}(\eta(M)+2\eta(k))+1)}\\
&  \geq\frac{k}{\log(\psi^{-1}(3\eta(k))+1)}\\
&  \geq\frac{k}{\log(\psi^{-1}(3\psi(k))+1)}\\
&  \geq\frac{k}{\log(\psi^{-1}(\psi(k+1))+1)}>\frac{1}{\delta}%
\end{align*}
\newline Thus, $B(x,\eta(k))\times\lbrack k,k+1]\subseteq U(\left\langle
\infty\right\rangle ,\delta)$. The case where $k<-N_{Q}$ is similar, with a
conclusion that $B(x,\eta(k))\times\lbrack k,k+1]\subseteq U(\left\langle
-\infty\right\rangle ,\delta)$.
\end{proof}

\begin{proposition}
\label{QkxK} For each $\left[  -N,N\right]  \subseteq%
\mathbb{R}
$, there exists a compact set $Q_{N}\subseteq X$ such that if $(B_{d}%
[x,\eta(|k|)]\times\lbrack k,k+1])\cap(Q_{N}\times\left[  -N,N\right]
)=\varnothing$ and $\left[  -N,N\right]  \cap\lbrack k,k+1]\neq\varnothing$,
then there exists $z\in Z$ such that $B[x,\eta(k)]\times\lbrack
k,k+1]\subseteq U(\langle z,0\rangle,\delta)$.
\end{proposition}

\begin{proof}
Observe that it follows from the hypotheses that $Q_{N}$ must be disjoint from
$B_{d}[x,\eta(|k|)]$. Thus we may choose $Q_{N}$ sufficiently large that
$\overline{d}(B_{d}[x,\eta(k)],Z)<\frac{\delta}{2}$ and $\operatorname*{diam}%
_{\overline{d}}(B_{d}[x,\eta(k)])<\frac{\delta}{2}$. For example, let
$Q_{N}=B_{d}\left[  x_{0},R\right]  $ where $R\geq\lambda(2/\delta)$.

Suppose now that $(B[x,\eta(k)]\times\lbrack k,k+1])\cap(Q_{N}\times\left[
-N,N\right]  )=\varnothing$ and $\left[  -N,N\right]  \cap\lbrack
k,k+1]\neq\varnothing$. Choose $z\in Z$ for which $\overline{d}(B_{d}%
[x,\eta(k)],z)<\delta/2$, and assume for the moment that $k\geq0$. Note that
for any $(x^{\prime},r^{\prime})\in B(x,\eta(k))\times\lbrack k,k+1]$,%
\[
\mu(x^{\prime},r^{\prime})\leq\frac{N+1}{\log(\psi^{-1}(R+\psi(0))+1)}%
\]

So, by choosing $R$ sufficiently large, we can ensure that $\mu(x^{\prime
},r^{\prime})<\delta$. In that case, $B[x,\eta(k)]\times\lbrack
k,k+1]\subseteq U(\langle z,0\rangle,\delta)$

The case where $k<0$ is similar.
\end{proof}

Now let $Q^{\prime}=B_{d}[x_{0},\eta(S)]$ where $S\gg0$ is so large that, for
all $x\in X-Q^{\prime}$:

\begin{itemize}
\item $p(x)>\frac{2}{\delta}$,

\item $\overline{d}(x,Z)<\frac{\delta}{2}$, and

\item $\log(\frac{\psi^{-1}(3\eta(S))+1}{\psi^{-1}(\eta(S))+1})<\delta$.
\end{itemize}

\noindent Then choose $N^{\prime}\gg0$ so large that:

\begin{itemize}
\item $\frac{1}{N^{\prime}}<\frac{\delta}{2}$,

\item $\frac{N^{\prime}}{\log(N^{\prime}+2)}>\frac{1}{\delta}$ (so $\frac
{|r|}{\log(|r|+2)}>\frac{1}{\delta}$ for all $\left\vert r\right\vert
>N^{\prime}$), and

\item $\overline{\sigma}(r,\left\{  \pm\infty\right\}  )<\frac{\delta}{2}$ for
all $r\in{%
\mathbb{R}
}-[-N^{\prime},N^{\prime}]$ (so $\operatorname*{diam}_{\overline{\sigma}%
}([k,k+1])<\frac{\delta}{2}$ whenever $[k,k+1]\cap\left[  -N^{\prime
},N^{\prime}\right]  =\varnothing$).
\end{itemize}

\noindent Given $Q^{\prime}\subseteq X$ as chosen above, choose $N_{Q^{\prime
}}>0$ in accordance with Proposition \ref{JxPj}; and given $N^{\prime}>0$ as
chosen above, choose $Q_{N^{\prime}}\subseteq X$ in accordance with
Proposition \ref{QkxK}. Let $N=\max\left\{  N^{\prime},N_{Q^{\prime}}\right\}
$ and $Q=Q^{\prime}\cup Q_{N^{\prime}}$

\begin{proposition}
\label{MainNullityProposition} If $(B(x,\eta(k))\times\lbrack k,k+1])\cap
(Q\times\left[  -N,N\right]  )=\varnothing$, then there exists $\langle
z,\mu\rangle\in SZ$ such that $B(x,\eta(k))\times\lbrack k,k+1]\subseteq
U(\langle z,\mu\rangle,\delta)$.
\end{proposition}

\begin{proof}
If $B(x,\eta(k))\cap Q$ or $[k,k+1]\cap\lbrack-N,N]$ is nonempty, the
conclusion follows from Proposition \ref{JxPj} or \ref{QkxK}, so we assume
that $B(x,\eta(k))\cap Q=\varnothing=[k,k+1]\cap\lbrack-N,N]$. For
convenience, assume also that $k>0$.

Let $M\in\lbrack0,\infty)$ be such that $d(x,x_{0})-\eta(k)=\eta(M)$. Such an
$M$ exists since $B(x,\eta(k))\cap Q=\varnothing$ and we assumed $\eta$ was
continuous and $\eta\rightarrow\infty$.\medskip

\noindent\textbf{Case 1.}\textit{ There exists }$( x^{\prime},r^{\prime}) \in
B(x,\eta(k))\times\lbrack k,k+1]$\textit{ such that }$\mu( x^{\prime},
r^{\prime}) \leq\frac{1}{\delta}$.\medskip

By the choice of $S$, there exists $z\in Z$ such that $\overline{d}(x^{\prime
},z)<\frac{\delta}{2}$. By the choice of $N$ (and since $k>0$), we know that
$\overline{\sigma}(r^{\prime},\infty)<\frac{\delta}{2}$.

For any other $(x^{\prime\prime},r^{\prime\prime})\in B(x,\eta(k))\times
\lbrack k,k+1]$,%

\begin{align*}
\left\vert \mu(x^{\prime\prime},r^{\prime\prime})-\mu(x^{\prime},r^{\prime
})\right\vert  &  =|\mu(x^{\prime\prime},r^{\prime\prime})-\mu(x^{\prime
\prime},r^{\prime})+\mu(x^{\prime\prime},r^{\prime})-\mu(x^{\prime},r^{\prime
})|\\
&  =\left\vert \frac{r^{\prime\prime}}{p(x^{\prime\prime})}-\frac{r^{\prime}%
}{p(x^{\prime\prime})}+\frac{r^{\prime}}{p(x^{\prime\prime})}-\frac{r^{\prime
}}{p(x^{\prime})}\right\vert \\
&  \leq\frac{1}{p(x^{\prime\prime})}\left\vert r^{\prime\prime}-r^{\prime
}\right\vert +\mu(x^{\prime},r^{\prime})\frac{|p(x^{\prime})-p(x^{\prime
\prime})|}{p(x^{\prime\prime})}\\
&  <\frac{\delta}{2}+\frac{1}{\delta}\cdot\frac{|p(x^{\prime})-p(x^{\prime
\prime})|}{2/\delta}%
\end{align*}
So if we can show that $\left\vert p(x^{\prime})-p(x^{\prime\prime
})\right\vert <\delta$, we may conclude that slopes of points in
$B(x,\eta(k))\times\lbrack k,k+1]$ differ by no more than $\delta$.\medskip

\noindent\emph{Claim.} $\eta(M)>\psi(k)$\newline Slopes of points in
$B(x,\eta(k))\times\lbrack k,k+1]$ are bounded below by
\[
\mu_{\text{min}}=\frac{k}{\log(\psi^{-1}(\eta(M)+2\eta(k))+1)}%
\]
By construction, $\psi(k)\geq\eta(k)$. Suppose that $\psi(k)\geq\eta(M)$.
Then,%
\begin{align*}
\mu_{\text{min}}  &  \geq\frac{k}{\log(\psi^{-1}(3\psi(k))+1)}\\
&  \geq\frac{k}{\log(\psi^{-1}(\psi(k+1))+1)}>\frac{1}{\delta}%
\end{align*}
contradicting the existence of $(x^{\prime},r^{\prime})\in B(x,\eta
(k))\times\lbrack k,k+1]$ with $\mu(x^{\prime},r^{\prime})\leq\frac{1}{\delta
}$. The claim follows.

Now%
\begin{align*}
\left\vert p(x^{\prime})-p(x^{\prime\prime})\right\vert  &  \leq\log(\psi
^{-1}(\eta(M)+2\eta(k))+1)-\log(\psi^{-1}(\eta(M))+1)\\
&  <\log(\frac{\psi^{-1}(3\eta(M))+1}{\psi^{-1}(\eta(M))+1})<\delta
\end{align*}
Since $\eta(M)>\psi(k)$, we are guaranteed that $\operatorname*{diam}%
_{\overline{d}}(B(x,\eta(k)))<\frac{1}{k}<\frac{\delta}{2}$; so by the
triangle inequality $B(x,\eta(k))\times\lbrack k,k+1]\subseteq U(\langle
\bar{x},\mu(x^{\prime},r^{\prime})\rangle,\delta)$.\medskip

\noindent\textbf{Case 2.}\textit{ There exists no }$( x^{\prime},r^{\prime})
\in B(x,\eta(k))\times\lbrack k,k+1]$\textit{ such that }$\mu( x^{\prime},
r^{\prime}) \leq\frac{1}{\delta}$.\medskip

Then $\mu(x^{\prime},r^{\prime})>\frac{1}{\delta}$ for all $(x^{\prime
},r^{\prime})\in B(x,\eta(k))\times\lbrack k,k+1]$. Since all of the slopes
are greater than $\frac{1}{\delta}$, the choice of $N$ guarantees that
$B(x,\eta(k))\times\lbrack k,k+1]\subseteq U(\langle\infty\rangle,\delta)$.
\end{proof}

To complete the proof of Theorem
\ref{Theorem: Controlled Z-compactification of X x R}, we need only prove the following:

\begin{proposition}
\label{ProductWithJoinIsANR} Given $X$, $\overline{X}$, and $\overline
{X\times{%
\mathbb{R}
}}$ as defined above, $SZ$ is a $\mathcal{Z}$-set in $\overline{X\times{%
\mathbb{R}
}}$. If $X$ is an AR, then so are $X\times%
\mathbb{R}
$, $\overline{X}$, and $\overline{X\times{%
\mathbb{R}
}}$.
\end{proposition}

To prove Proposition \ref{ProductWithJoinIsANR}, we will construct a
contraction $\gamma:\overline{X\times{%
\mathbb{R}
}}\times\left[  0,1\right]  \rightarrow\overline{X\times{%
\mathbb{R}
}}$ with $\gamma_{0}=\operatorname*{id}_{\overline{X\times{%
\mathbb{R}
}}}$, $\gamma_{1}(\overline{X\times{%
\mathbb{R}
}})=\left\{  x_{0}\right\}  $, and $\gamma(\overline{X\times{%
\mathbb{R}
}}\times(0,1])\subseteq X\times{%
\mathbb{R}
}$. From there, the second sentence of the proposition follows from the
discussion of $\mathcal{Z}$-compactifications in Section
\ref{Subsection: Z-sets and Z-structures}. Our construction of $\gamma$
involves relatively minor modifications to the analogous construction in
Section~3 of \cite{Tir11}. The first modification is entirely superficial,
owing to the fact that she is working with a general product $X\times Y$ and
join $Z_{X}\ast Z_{Y}$ while we are working with a special case: $X\times%
\mathbb{R}
$ and $SZ$. The second modification is due to the more delicate nature of our
choice of the function $p:X\rightarrow\lbrack0,\infty)$, used to define the
slope function.

The strategy for contracting $\overline{X\times%
\mathbb{R}
}$ is motivated by the contraction of the visual compactification
$\overline{Y}$ of a proper CAT(0) space $Y$ to a basepoint $y_{0}$, whereby
interior points and boundary points are slid toward $y_{0}$ along geodesic
segments and geodesic rays, respectively. There it is useful to view---and
parameterize---the geodesic segments as \textquotedblleft eventually
constant\textquotedblright\ nonproper rays. In this way, one associates to the
points of $\overline{Y}$, a continuously varying family of rays $\gamma
_{y}:[0,\infty)\rightarrow Y$ beginning at $y_{0}$ and ending at $y\in Y$ or
else determining a point $y$ in the visual boundary by virtue of being a
geodesic ray. The definition of the cone topology allows us to extend each to
$\gamma_{y}:[0,\infty]\rightarrow\overline{Y}$ without losing continuity. The
contraction is obtained by applying a deformation retraction of $[0,\infty]$
onto $\left\{  0\right\}  $ simultaneously to all (extended) rays.

In our setting, we will identify a continuous family of preferred rays in
$X\times%
\mathbb{R}
$, all emanating from $(x_{0},0)$. Some will be eventually constant, ending at
a point $(x,r)\in X\times%
\mathbb{R}
%
$; others will limit to a point $\langle z,\mu\rangle\in SZ$. By choosing
these rays in coordination with the topology on $\overline{X\times%
\mathbb{R}
}$, we will be able to mimic the above strategy.

Let $\alpha:\overline{X}\times\left[  0,1\right]  \rightarrow\overline{X}$ be
a $\mathcal{Z}$-set homotopy that contracts $\overline{X}$ to $x_{0}\in X$
keeping $x_{0}$ fixed throughout, and $\beta:\left[  -\infty,\infty\right]
\times\left[  0,1\right]  \rightarrow\left[  -\infty,\infty\right]  $ be a
$\mathcal{Z}$-set homotopy that contracts $\left[  -\infty,\infty\right]  $ to
$0$ keeping $0$ fixed throughout.

The following lemma is inspired by \cite[Lemmas 3.6 \& 3.8]{Tir11}.

\begin{lemma}
\label{ReparameterizedRays} There are reparameterizations $\widehat{\alpha}$
and $\widehat{\beta}$ of $\alpha$ and $\beta$ so that $p(\widehat{\alpha
}(z,t))\in\lbrack\frac{1}{t}-1,\frac{1}{t}+2]$ and $\left\vert \widehat{\beta
}(\pm\infty,t)\right\vert \in\lbrack\frac{1}{t}-1,\frac{1}{t}+2]$ for all
$t\in(0,1]$ and for all $z\in Z$.
\end{lemma}

\begin{proof}
Our approach differs slightly from that taken in \cite{Tir11}. Whereas she
\emph{chose} her function $p:X\rightarrow\lbrack0,\infty)$ to have the
property:\medskip

(\dag\dag) \ For some sequence $1=t_{0}>t_{1}>\dots>0,$ $p(\alpha
(Z\times\lbrack t_{i},t_{i-1})))\subseteq(i-1,i+1]$\medskip

\noindent we, instead, begin with the function $p:X\rightarrow\lbrack
0,\infty)$ arrived at in defining our slope function, then arrange condition
(\dag\dag) by reparameterizing $\alpha$. This can be accomplished by using
methods similar to those used by Tirel to define her function $p$.

Once $\alpha$ has been adjusted to satisfy property (\dag\dag), we can
implement the proof of \cite[Lemma 3.8]{Tir11} to obtain $\widehat{\alpha}$.
Obtaining $\widehat{\beta}$ is much simpler.
\end{proof}

\begin{remark}
\label{Remark: defining a Z-homotopy}Notice that, aside from properness, no
specific properties of $p:X\rightarrow\lbrack0,\infty)$ are used in the above
proof. This fact will be useful in Section \ref{Section: EZ-bar structures}.
\end{remark}

In order to turn homotopy tracks into rays, we invert and stretch the
\textquotedblleft time\textquotedblright\ interval. Define
\[%
\begin{tabular}
[c]{ccc}%
$\xi:[0,\infty]\rightarrow\lbrack0,1]$ & by & $\xi(t)=\left\{
\begin{tabular}
[c]{cc}%
$\frac{1}{1+t}$ & if $t\in\lbrack0,\infty)$\\
0 & if $t=\infty$%
\end{tabular}
\ \right.  $\\
$\alpha^{\prime}:\overline{X}\times\lbrack0,\infty]\rightarrow\overline{X}$ &
by & $\alpha^{\prime}(w,t)=\widehat{\alpha}(w,\xi(t))$\\
$\beta^{\prime}:\left[  -\infty,\infty\right]  \times\lbrack0,\infty
]\rightarrow\overline{%
\mathbb{R}
}$ & by & $\beta^{\prime}(r,t)=\widehat{\beta}(r,\xi(t))$%
\end{tabular}
\ \ \
\]

\noindent where we allow $w$ to denote an element of either $X$ or $Z$. Notice
that for any $t\in\lbrack0,\infty)$ and $z\in Z$, $p(\alpha^{\prime}%
(z,t))\in(t-1,t+3)$; and similarly, $\left\vert \beta^{\prime}(\pm
\infty,t)\right\vert \in(t-1,t+3)$.

Define%
\[
\gamma^{\prime}:\overline{X\times{%
\mathbb{R}
}}\times\lbrack0,\infty)\rightarrow X\times{%
\mathbb{R}
}%
\]
by%
\[%
\begin{tabular}
[c]{lll}%
$((x,r),t)$ & $\mapsto$ & $(\alpha^{\prime}(x,\frac{t}{\sqrt{(\mu(x,r))^{2}%
+1}}),\beta^{\prime}(r,\frac{\mu(x,r)\cdot t}{\sqrt{(\mu(x,r))^{2}+1}%
}))\smallskip$\\
$(\langle z,\mu\rangle,t)$ & $\mapsto$ & $(\alpha^{\prime}(z,\frac{t}%
{\sqrt{\mu^{2}+1}}),\beta^{\prime}(\infty,\frac{\mu\cdot t}{\sqrt{\mu^{2}+1}%
}))\smallskip$\qquad(if $\mu\geq0$)\\
$(\langle z,\mu\rangle,t)$ & $\mapsto$ & $(\alpha^{\prime}(z,\frac{t}%
{\sqrt{\mu^{2}+1}}),\beta^{\prime}(-\infty,\frac{\mu\cdot t}{\sqrt{\mu^{2}+1}%
}))\medskip$\quad\ (if $\mu<0$)\\
$(\langle\infty\rangle,t)$ & $\mapsto$ & $(x_{0},\beta^{\prime}(\infty
,t))\medskip$\\
$(\langle-\infty\rangle,t)$ & $\mapsto$ & $(x_{0},\beta^{\prime}(-\infty,t))$%
\end{tabular}
\
\]
\bigskip

\noindent Clearly $\gamma^{\prime}$ extends continuously over $X\times{%
\mathbb{R}
}\times\lbrack0,\infty]$ by sending each $((x,r),\infty)$ to $(x,r)$. For
later use, let $\gamma_{(x,r)}^{\prime}=\left.  \gamma^{\prime}\right\vert
_{(x,r)\times\lbrack0,\infty]}$. Note that for each $\left\langle
z,\mu\right\rangle \in SZ$, the map $\gamma_{\left\langle z,\mu\right\rangle
}^{\prime}=\left.  \gamma^{\prime}\right\vert _{\left\langle z,\mu
\right\rangle \times\lbrack0,\infty)}$ is a proper ray in $X\times%
\mathbb{R}
$. We wish to observe that $\gamma_{\left\langle z,\mu\right\rangle }^{\prime
}(t)\rightarrow\left\langle z,\mu\right\rangle $ in $\overline{X\times{%
\mathbb{R}
}}$ as $t\rightarrow\infty$. In the special case that $\mu=\infty$, i.e.
$\left\langle z,\mu\right\rangle =\langle\infty\rangle$, then $\gamma
_{\langle\infty\rangle}^{\prime}(\langle\infty\rangle,t)=(x_{0},\beta^{\prime
}(\infty,t))$. Since $p(x_{0})=0$, then $\mu(x_{0},\beta^{\prime}%
(\infty,t))=\infty$; moreover $\beta^{\prime}(\infty,t)\rightarrow\infty$ as
$t\rightarrow\infty$. So $\gamma_{\langle\infty\rangle}^{\prime}%
(t)\rightarrow\langle\infty\rangle$ as $t\rightarrow\infty$. Similarly for
$\mu=-\infty$. For a generic boundary point $\langle z,\mu\rangle$ with
$0\leq\mu<\infty$, we have%
\begin{align*}
\mu(\gamma^{\prime}(\langle z,\mu\rangle,t))  &  =\frac{\beta^{\prime}%
(\infty,\frac{\mu\cdot t}{\sqrt{\mu^{2}+1}})}{p(\alpha^{\prime}(z,\frac
{t}{\sqrt{\mu^{2}+1}}))}\\
&  \in(\frac{\frac{\mu\cdot t}{\sqrt{\mu^{2}+1}}-2}{\frac{t}{\sqrt{\mu^{2}+1}%
}+3},\frac{\frac{\mu\cdot t}{\sqrt{\mu^{2}+1}}+3}{\frac{t}{\sqrt{\mu^{2}+1}%
}-2})\\
&  =(\frac{\mu\cdot t-2\sqrt{\mu^{2}+1}}{t+3\sqrt{\mu^{2}+1}},\frac{\mu\cdot
t+3\sqrt{\mu^{2}+1}}{t-2\sqrt{\mu^{2}+1}})
\end{align*}
\newline which implies that $\mu(\gamma^{\prime}(\langle z,\mu\rangle
,t))\rightarrow\mu$ as $t\rightarrow\infty$. In addition, $\alpha^{\prime
}(z,\frac{t}{\sqrt{\mu^{2}+1}})\rightarrow\alpha(z,0)=z$ in $\overline{X}$ as
$t\rightarrow\infty$. Thus $\gamma^{\prime}(\langle z,\mu\rangle
,t)\rightarrow\langle z,\mu\rangle$ in $\overline{X\times{%
\mathbb{R}
}}$, and we define $\gamma^{\prime}(\langle z,\mu\rangle,\infty)=\langle
z,\mu\rangle$. Calculations similar to the above show that, for small
$\varepsilon$, rays of the form $\gamma_{(x^{\prime},r)}^{\prime}$ and
$\gamma_{\left\langle z^{\prime},\mu^{\prime}\right\rangle }^{\prime}$ which
end in a basic open set $U(\langle z,\mu\rangle,\varepsilon)$, track together
in $\overline{X\times{%
\mathbb{R}
}}$. As such, we obtain a continuous function
\[
\gamma^{\prime}:\overline{X\times{%
\mathbb{R}
}}\times\lbrack0,\infty]\rightarrow\overline{X\times{%
\mathbb{R}
}}%
\]

Reversing and reparameterizing the interval once more, we get the desired
$\mathcal{Z}$-set homotopy $\gamma:\overline{X\times{%
\mathbb{R}
}}\times\lbrack0,1]\rightarrow\overline{X\times%
\mathbb{R}
}$.


\section{Proof of Theorem \ref{Theorem A}}

\label{sec:proof}

We now turn to the proof of Theorem \ref{Theorem A}. The goal is a
$\mathcal{Z}$-compactification $\overline{Y}=Y\sqcup SZ$ of the $(G\rtimes
_{\phi}%
\mathbb{Z}
)$-space $Y=\operatorname*{Tel}_{f}(X)$ that satisfies Definition
\ref{DefineZStructure}. The approach is to use the map $v:Y\rightarrow X\times%
\mathbb{R}
$ to swap the boundary from a controlled $\mathcal{Z}$-compactification
$\overline{X\times%
\mathbb{R}
}=\left(  X\times%
\mathbb{R}
\right)  \sqcup SZ$ onto $Y$. The key is to incorporate some of the geometry
of the $(G\rtimes_{\phi}%
\mathbb{Z}
)$-action on $Y$ into the choice of control function $\eta:[0,\infty
)\rightarrow\lbrack0,\infty)$ used in Theorem
\ref{Theorem: Controlled Z-compactification of X x R}.

Choose a $(G\rtimes_{\phi}{%
\mathbb{Z}
})$-equivariant metric $\rho$ on $Y$ (see \cite[\S 6]{GuMo19}). The
restriction of $\rho$ to any $X_{i}$, $i\in%
\mathbb{Z}
$, is a $G$-invariant metric on $X$ which we will designate as $d$. Give $%
\mathbb{R}
$ the usual metric, and let $d_{1}$ be the corresponding $\ell_{1}$ metric on
$X\times%
\mathbb{R}
$. As in the previous section, let $\overline{d}$ be a metric for
$\overline{X}$ and $\overline{\sigma}$ a metric for $\left[  -\infty
,\infty\right]  $.

Let $C_{Y}\subseteq\mathcal{M}_{[0,1]}(f)$ be a compact set whose $(
G\rtimes_{\phi}{%
\mathbb{Z}
}) $-translates cover $Y$. For example, let $C_{X}\subseteq X$ be a finite
subcomplex whose $G$-translates cover $X$, then let $C_{Y}$ be the sub-mapping
cylinder $\mathcal{M}_{[0,1]}(\left.  f\right\vert _{C_{X}})$, where the range
is restricted to a finite subcomplex containing $f( C_{X}) $. We will refer to
$C_{X}$ and $C_{Y}$ informally as \emph{fundamental domains} for the $G$- and
$( G\rtimes_{\phi}%
\mathbb{Z}
) $-actions on $X$ and $Y$, respectively.

\begin{definition}
[The control function]\label{DefineEta} Let $\eta:[0,\infty)\rightarrow
\lbrack0,\infty)$ be a function satisfying: For all $k\in\mathbb{%
\mathbb{Z}
}$,%
\[
\eta(\left\vert k\right\vert )\geq\text{max}%
\begin{cases}
\operatorname*{diam}_{d_{1}}(v(H(t^{\pm k}C_{Y}\times\lbrack0,1])))\\
\operatorname*{diam}_{d_{1}}(J(C_{X}\times\left[  k,k+1\right]  )\times
\lbrack0,1]))
\end{cases}
\]
Furthermore, choose $\eta$ to be monotonic and require that $\lim
_{r\rightarrow\infty}\eta(r)=\infty$.
\end{definition}

\begin{remark}
\label{Remark: about distortion function}A priori, $\eta(\left\vert
k\right\vert )$ provides a bound only on the diameters of $h(t^{\pm k}%
C\times\lbrack0,1])$ and $J(v(t^{\pm k}C)\times\lbrack0,1])$; however the fact
that $v$, $H$, and $J$ are $G$-equivariant and level preserving, means that
$\eta(\left\vert k\right\vert )$ bounds the diameters of all $(G\rtimes_{\phi
}{%
\mathbb{Z}
})$-translates of those sets which are contained in the $\left[  k,k+1\right]
$-level. Since $H_{0}$ and $J_{0}$ are identities, $\eta(\left\vert
k\right\vert )$ also bounds the diameters of $t^{\pm k}C$ and $v(t^{\pm k}C)$
and their translates contained in the $\left[  k,k+1\right]  $-level.
\end{remark}

Now apply Theorem \ref{Theorem: Controlled Z-compactification of X x R} to
obtain a $\mathcal{Z}$-compactification $(\overline{X\times{%
\mathbb{R}
}},SZ)$ of $X\times%
\mathbb{R}
$ satisfying the condition:\medskip

\noindent(\ddag)$\quad$\emph{For each open cover } $U$\emph{ of }%
$\overline{X\times{%
\mathbb{R}
}}$\emph{, there exists a compact set }$Q\times\left[  -N,N\right]  \subseteq
X\times%
\mathbb{R}
$\emph{, such that every set }$B_{d}[x,\eta(\left\vert k\right\vert
)]\times\lbrack k,k+1]$\emph{ lying outside }$Q\times\left[  -N,N\right]
$\emph{ is contained in some }$U\in U.\medskip$

Recall the map $v:Y\rightarrow X\times%
\mathbb{R}
$ defined in Section
\ref{Section: Mapping tori and telescopes as classifying spaces}. Let
$\overline{Y}=Y\sqcup SZ$ and define $\overline{v}=v\cup\operatorname*{id}%
_{SZ}:\overline{Y}\rightarrow\overline{X\times{%
\mathbb{R}
}}$, i.e.
\[
\overline{v}(z)=\left\{
\begin{tabular}
[c]{cc}%
$v(z)$ & if $z\in Y$\\
$z$ & if $z\in SZ$%
\end{tabular}
\ \ \ \right.
\]
Give $\overline{Y}$ the topology $\mathcal{T}$ generated by the open subsets
of $Y$ together with
\[
\left\{  \overline{v}^{-1}(\mathcal{U})\mid\mathcal{U}\text{ is open in
}\overline{X\times{%
\mathbb{R}
}}\right\}
\]

\begin{remark}
\label{Remark: pullback compactification}We call $\left(  \overline
{Y},\mathcal{T}\right)  $ the \emph{pull-back compactification} of $Y$ via the
map $v$. This construction can be applied more broadly whenever one has a
proper map $v:Y\rightarrow W$ between locally compact separable metric spaces
and a compactification $\overline{W}=W\sqcup A$. The result is a
compactification $\overline{Y}=Y\sqcup A$ (not always a $\mathcal{Z}%
$-compactification) of $Y$ and a continuous map $\overline{v}:\overline
{Y}\rightarrow\overline{W}$ which extends $v$ via the identity over $A$.
\end{remark}

\begin{proposition}
\label{prop:compactification} \label{Prop: Y-hat is a Z-compactification}%
$\overline{Y}$ is a $\mathcal{Z}$-compactification of $Y$.
\end{proposition}

Our proof will be obtained by applying the following lemma which was based on
\cite[Prop. 1.6]{Fer00}.

\begin{lemma}
[{see \cite[Lemma 3.1]{GuMo19}}]\label{Lemma: Boundary swap} Let $X$ and $Y$
be separable metric spaces and $f:(X,A)\rightarrow(Y,B)$ and
$g:(Y,B)\rightarrow(X,A)$ be continuous maps with $f(X-A)\subseteq Y-B$,
$g(Y-B)\subseteq X-A$, and $\left.  g\circ f\right\vert _{A}%
=\operatorname*{id}_{A}$. Suppose further that there is a homotopy
$K:X\times\lbrack0,1]\rightarrow X$ which is fixed on $A$ and satisfies:
$K_{0}=\operatorname*{id}_{X}$, $K_{1}=g\circ f$, and $K((X-A)\times
\lbrack0,1])\subseteq X-A$. If $B$ is a $\mathcal{Z}$-set in $Y$, then $A$ is
a $\mathcal{Z}$-set in $X$.
\end{lemma}

\begin{proof}
[Proof of Proposition \ref{Prop: Y-hat is a Z-compactification}]Recall the
maps $u:X\times%
\mathbb{R}
\rightarrow Y$ and $H:Y\times\left[  0,1\right]  \rightarrow Y$ defined in
Section \ref{Section: Mapping tori and telescopes as classifying spaces}, and
let $\overline{u}:\overline{X\times%
\mathbb{R}
}\rightarrow\overline{Y}$ and $\overline{H}:\overline{Y}\times\left[
0,1\right]  \rightarrow\overline{Y}$ be extensions via the identity on $SZ$.
In order to apply Lemma \ref{Lemma: Boundary swap}, it suffices to show that
$\overline{u}$ and $\overline{H}$ are continuous.

We will use the following notational convention: Whenever $\overline{V}$
denotes a subset of $\overline{Y}$ [resp., $\overline{X\times%
\mathbb{R}
}$], $V$ will denote $\overline{V}\cap Y$ [resp., $\overline{V}\cap(X\times%
\mathbb{R}
)$]. \medskip

\noindent\emph{Claim 1.}\textit{\emph{ }}$\overline{u}$\textit{ }\emph{is
continuous.}\medskip

It suffices to verify continuity at points of $SZ$. Let $\langle z_{0},\mu
_{0}\rangle\in SZ$ and $\overline{v}^{-1}(\overline{V})$ be a basic open
neighborhood of $\overline{u}(\langle z_{0},\mu_{0}\rangle)=\langle z_{0}%
,\mu_{0}\rangle$ in $\overline{Y}$. The goal is to find a basic open
neighborhood $U(\left\langle z_{0},\mu\right\rangle ,\varepsilon_{0})$ of
$\langle z_{0},\mu_{0}\rangle$ in $\overline{X\times%
\mathbb{R}
}$ such that $\overline{u}(U(\left\langle z_{0},\mu\right\rangle
,\varepsilon_{0}))\subseteq\overline{v}^{-1}(\overline{V})$, i.e.,
$\overline{v}\left(  \overline{u}(U(\left\langle z_{0},\mu\right\rangle
,\varepsilon_{0}))\right)  \subseteq\overline{V}$.

To begin, assume that $\mu_{0}\neq\pm\infty$ and choose $j\in%
\mathbb{Z}
$ so that $j-1<\mu_{0}<j+1$. Let $\overline{W}$ of be an open neighborhood of
$\langle z_{0},\mu_{0}\rangle$ such that $\operatorname*{cl}(\overline
{W})\subseteq\overline{V}$. Then $\left\{  \overline{X\times%
\mathbb{R}
}-\operatorname*{cl}(\overline{W}),\overline{V}\right\}  $ is an open cover of
$\overline{X\times%
\mathbb{R}
}$, so by Lemma \ref{LebesgueNumberOfBoundary}, there exists $\varepsilon
^{\prime}>0$ such that, if $U(\left\langle z^{\prime},\mu^{\prime
}\right\rangle ,\varepsilon^{\prime})\cap\operatorname*{cl}(\overline{W}%
)\neq\varnothing$, then $U(\left\langle z^{\prime},\mu^{\prime}\right\rangle
,\varepsilon^{\prime})\subseteq\overline{V}$. Choose $Q\times\left[
-N,N\right]  $ be sufficiently large that $N\geq j+1$ and every set of the
form $B_{d}\left[  x,\eta(k)\right]  \times\left[  k,k+1\right]  $ which lies
outside $Q\times\left[  -N,N\right]  $ is contained in some $U(\left\langle
z,\mu\right\rangle ,\varepsilon^{\prime})$. Choose $\varepsilon_{0}>0$ so that:

\begin{itemize}
\item $U(\left\langle z_{0},\mu\right\rangle ,\varepsilon_{0})\subseteq
\overline{W}$,

\item $\varepsilon_{0}\leq\operatorname*{dist}( \mu_{0},\left\{
j-1,j+1\right\}  ) $, and

\item $d(x,Q)>2\eta(N)$ for all $(x,r)\in U(\left\langle z_{0},\mu
_{0}\right\rangle ,\varepsilon_{0})$.
\end{itemize}

Let $(x,r)\in U(\left\langle z_{0},\mu\right\rangle ,\varepsilon_{0})$. Then
$r\in\lbrack j-1,j+1]$. Assume $r\in\left[  j,j+1\right]  $ (the case
$r\in\lbrack j-1,j]$ is similar). Choose $g\in G$ so that $x\in gC_{X}$. Then%
\[
(x,r)\in gC_{X}\times\left[  j,j+1\right]  \subseteq J(gC_{X}\times\left[
j,j+1\right]  \times\left[  0,1\right]  )\subseteq B_{d}\left[  gx_{0}%
,\eta(j)\right]  \times\left[  j,j+1\right]
\]
Since $d(x,Q)>2\eta(N)>2\eta(j)$, then $B_{d}\left[  gx_{0},\eta(j)\right]
\times\left[  j,j+1\right]  $ lies outside $Q\times\left[  -N,N\right]  $,
hence entirely in some basic open set $U(\left\langle z^{\prime},\mu^{\prime
}\right\rangle ,\varepsilon^{\prime})$. This basic open set intersects
$\overline{W}$ at $\left(  x,r\right)  $, so by choice of $\varepsilon
^{\prime}$, $U(\left\langle z^{\prime},\mu^{\prime}\right\rangle
,\varepsilon^{\prime})\subseteq\overline{V}$. Furthermore, since $B_{d}\left[
gx_{0},\eta(j)\right]  \times\left[  j,j+1\right]  $ contains $J((x,r)\times
\left[  0,1\right]  )$, then $U(\left\langle z^{\prime},\mu^{\prime
}\right\rangle ,\varepsilon^{\prime})$ contains $\overline{v}\circ\overline
{u}(x)$. This means that $\overline{v}(\overline{u}(x,r))\in\overline{V}$, as
desired. Since $\overline{v}\circ\overline{u}$ is the identity on $SZ$, it
follows that $\overline{v}(\overline{u}(U(\left\langle z_{0},\mu\right\rangle
,\varepsilon_{0})))\subseteq\overline{V}$.

A similar, but easier, argument verifies continuity of $\overline{u}$ at
$\left\langle \pm\infty\right\rangle $.\medskip

\noindent\emph{Claim 2.}\textit{\emph{ }}$\overline{H}$\textit{ }\emph{is
continuous.}\medskip

Let $\langle z_{0},\mu_{0}\rangle\in SZ$ and $\overline{v}^{-1}(\overline{V})$
be is a basic open neighborhood of $\langle z_{0},\mu_{0}\rangle$ in
$\overline{Y}$. Then $\overline{V}$ is an open neighborhood of $\overline
{v}(\langle z_{0},\mu_{0}\rangle)=\langle z_{0},\mu_{0}\rangle$ in
$\overline{X\times%
\mathbb{R}
}$. Arguing in much the same way as above, we may choose a basic open
neighborhood $U(\left\langle z_{0},\mu_{0}\right\rangle ,\varepsilon
_{0})\subseteq\overline{V}$ so small that, if $y\in Y$ and $\overline{v}(y)\in
U(\left\langle z_{0},\mu_{0}\right\rangle ,\varepsilon_{0})$ then
$\overline{v}(H(y\times\left[  0,1\right]  ))\subseteq\overline{V}$. (This
uses the first control condition in Definition \ref{DefineEta}.) As such,
$\overline{v}(\overline{H}(\overline{v}^{-1}(U(\left\langle z_{0},\mu
_{0}\right\rangle ,\varepsilon_{0}))\times\left[  0,1\right]  ))\subseteq$
$\overline{V}$, so $\overline{H}(\overline{v}^{-1}(U(\left\langle z_{0}%
,\mu_{0}\right\rangle ,\varepsilon_{0}))\times\left[  0,1\right]
)\subseteq\overline{v}^{-1}(\overline{V})$, implying that $\overline{H}$ is
continuous at points of $\langle z_{0},\mu_{0}\rangle\times\left[  0,1\right]
$.
\end{proof}

We are now ready to complete the main task of this section.

\begin{proof}
[Proof of Theorem \ref{Theorem A}]{Thus far we }have a geometric action of
$G\rtimes_{\phi}%
\mathbb{Z}
$ on $Y=\operatorname*{Tel}_{f}(X)$, and a $\mathcal{Z}$-compactification
$\overline{Y}=Y\sqcup SZ$. We have already noted that $Y$ is a contractible
locally finite CW complex---hence an AR---and that a $\mathcal{Z}%
$-compactification of an AR is an AR. It remains only to show that
$(\overline{Y},SZ)$ satisfies the nullity condition. It suffices to show that
the $(G\rtimes_{\phi}%
\mathbb{Z}
)$-translates of the fundamental domain $C_{Y}$ form a null family in
$\overline{Y}$.

Let $\mathcal{V}$ be an open cover of $\overline{Y}$. By passing to an open
refinement, we may assume $\mathcal{V}$ consists entirely of basic open sets
and thus contains a subcollection $\left\{  \overline{v}^{-1}(U_{\alpha
})\right\}  _{\alpha\in A}$ which covers $SZ$. It follows that $\left\{
U_{\alpha}\right\}  _{\alpha\in A}$ covers $SZ$ in $\overline{X\times%
\mathbb{R}
}$. Choose a single open set $U_{0}=B_{d}(x_{0},r)\times(-r,r)$ sufficiently
large that $\mathcal{U=}\left\{  U_{0}\right\}  \cup\left\{  U_{\alpha
}\right\}  _{\alpha\in A}$ covers $\overline{X\times%
\mathbb{R}
}$. Then choose a compact set $Q\times\left[  -N,N\right]  \subseteq X\times%
\mathbb{R}
$ satisfying condition (\ddag). Let $Q^{\prime}$ be the closed $\eta
(N+1)$-neighborhood of $Q$ in $(X,d)$ and choose $P\subseteq Y$ to be a
compact set containing both $v^{-1}(Q^{\prime}\times\left[  -(N+1),N+1\right]
)$ and $v^{-1}(B_{d}[x_{0},r]\times\lbrack-r,r])$. Let $a\in G\rtimes_{\phi}%
\mathbb{Z}
$ such that $aC_{Y}\cap P=\varnothing$. If we use the standard presentation
for semidirect products to express $a$ as $gt^{k}$, then $aC_{Y}%
\subseteq\mathcal{M}_{\left[  k,k+1\right]  }(f)$ and $v(aC_{Y})\subseteq
X\times\left[  k,k+1\right]  $. Since $\operatorname*{diam}_{d_{1}}%
(v(aC_{Y}))\leq\eta(\left\vert k\right\vert )$, then $v(aC_{Y})\subseteq
B_{d_{1}}\left[  x,\eta(k)\right]  \times\left[  k,k+1\right]  $ for some
$x\in X$. Moreover, since $v(aC_{Y})\cap(Q^{\prime}\times\left[
-(N+1),N+1\right]  )=\varnothing$, then $(B_{d_{1}}\left[  x,\eta(k)\right]
\times\left[  k,k+1\right]  )\cap(Q\times\left[  -N,N\right]  )=\varnothing$.
Therefore $B_{d_{1}}\left[  x,\eta(k)\right]  \times\left[  k,k+1\right]  $,
and hence $v(aC_{Y})$, lies in some element of $\mathcal{U}$. That element
cannot be $U_{0}$ since $v(aC_{Y})\cap B_{d}[x_{0},r]\times\lbrack
-r,r]=\varnothing$, so $v(aC_{Y})\subseteq U_{\alpha}$ for some $\alpha\in A$.
Thus $aC_{Y}\subseteq v^{-1}(U_{\alpha})\in\mathcal{V}$. Since the action of
$G\rtimes_{\phi}%
\mathbb{Z}
$ on $Y$ is proper, the nullity condition follows.
\end{proof}


\section{$\mathcal{Z}$-structures for $G\rtimes_{\phi}%
\mathbb{Z}
$ when torsion is permitted\label{sec: proof with torsion permitted}}


We now turn to the general case where $G$ is permitted to have torsion. As
before, we assume a $\mathcal{Z}$-structure $(\overline{X},Z)$ on $G$, so
there is a proper, cocompact $G$-action on an AR $X$, but now the action need
not be free. As a result, the quotient map is not a covering projection. This
disables the tricks used at the beginning of Section
\ref{Section: Mapping tori and telescopes as classifying spaces} to obtain a
finite $K(G,1)$ complex, which was used to build: a finite $K(G\rtimes_{\phi}%
\mathbb{Z}
,1)$ complex; a corresponding universal cover; and a variety of maps and
homotopies. Without all the benefits of covering space theory, a more hands-on
approach is required.

Since the trick that allowed us to pass from ARs and ANRs to the category of
CW complexes is no longer available, we will deal directly with A(N)Rs
whenever possible. In several places, however, we can get slightly stronger
conclusions by assuming the existence of cell structures. Sometimes those
stronger conclusions are not needed, but in a few of the more delicate
applications they are crucial. For that reason, we include both A(N)R and
cellular versions in much of what follows.

In this section we will prove:

\begin{theorem}
\label{Theorem: Most general main theorem}Let $G$ be a group admitting a
$\mathcal{Z}$-structure $(\overline{X},Z)$ and let $\phi\in\operatorname*{Aut}%
\left(  G\right)  $. If there exist $\phi$-variant and $\phi^{-1}$-variant
self-maps of $X$ then $G\rtimes_{\phi}\mathbb{Z}$ admits a $\mathcal{Z}%
$-structure of the form $(\overline{Y},SZ)$.
\end{theorem}

In the torsion-free case, the existence of $\phi$-variant and $\phi^{-1}%
$-variant maps was automatic (by covering space theory). As such Theorem
\ref{Theorem A} can be deduced as a corollary of this theorem. When torsion is
present, we are not sure if these maps always exist, but they do exist in a
wide variety of important cases to be discussed in this section. Most notably,
we will see that Theorem \ref{Theorem: Most general main theorem} can be
applied whenever $X$ is an $\underline{E}G$-complex.

\subsection{Constructing $\left(  G\rtimes_{\phi}%
\mathbb{Z}
\right)  $-spaces}

For our purposes, the key to constructing a nice $\left(  G\rtimes_{\phi}%
\mathbb{Z}
\right)  $-space is the existence of a\ $\phi$-variant map from the $G$-space
$X$ to itself.

\begin{theorem}
\label{Theorem: constructing a G-semi-Z AR}Let $X$ be a proper cocompact
$G$-space, $\phi\in\operatorname*{Aut}\left(  G\right)  $, and $f:X\rightarrow
X$ a $\phi$-variant map. Then $Y=\operatorname*{Tel}_{f}\left(  X\right)  $
admits a corresponding proper cocompact $(G\rtimes_{\phi}%
\mathbb{Z}
)$-action. Moreover,

\begin{enumerate}
\item \label{Assertion 1 of constructing a G-semi-Z AR}if $X$ is contractible
then so is $Y$,

\item \label{Assertion 2 of constructing a G-semi-Z AR}if $X$ is an AR then so
is $Y$, and

\item \label{Assertion 3 of constructing a G-semi-Z AR}if $X$ is a [rigid]
$G$-complex and $f$ is cellular, then $Y$ admits a cell structure under which
the $(G\rtimes_{\phi}%
\mathbb{Z}
)$-action is [rigid] cellular.
\end{enumerate}
\end{theorem}

\begin{proof}
Lemma \ref{Lemma: basic facts about cylinders and telescopes} assures that $Y$
is locally compact, separable and metrizable. Applying Notation
\ref{Notation: points in a telescope}, for each $g\in G$, define%
\begin{equation}
g\cdot\lceil x,r\rceil=\lceil\phi^{\left\lfloor r\right\rfloor }\left(
g\right)  \cdot x,r\rceil\label{Defn: G acting on Y}%
\end{equation}
and let%
\begin{equation}
t\cdot\lceil x,r\rceil=\lceil x,r+1\rceil\label{Defn: t acting on Y}%
\end{equation}

Clearly $t$ determines a self-homeomorphism of $Y$ with inverse $t^{-1}%
\cdot\left(  x,r\right)  =\left(  x,r-1\right)  $. To see that each
$g:Y\rightarrow Y$ is continuous, note first that $g$ is continuous on the
individual subcylinders $\mathcal{M}_{[k,k+1)}\left(  f\right)  $.
Furthermore, for each mapping cylinder segment $\left\{  x\right\}
\times\lbrack k,k+1)$ in $\mathcal{M}_{[k,k+1)}\left(  f\right)  $ converging
to $\left(  f\left(  x\right)  ,k+1\right)  $ in $\mathcal{M}_{[k+1,k+2)}%
\left(  f\right)  $, the $g$-translate $\left\{  \phi^{k}\left(  g\right)
\cdot x\right\}  \times\lbrack k,k+1)$ converges to $\left(  f\left(  \phi
^{k}\left(  g\right)  \cdot x\right)  ,k+1\right)  $ in $\mathcal{M}%
_{[k+1,k+2)}\left(  f\right)  $. By $\phi$-variance, $\left(  f\left(
\phi^{k}\left(  g\right)  \cdot x\right)  ,k+1\right)  =\left(  \phi
^{k+1}\left(  g\right)  \cdot f\left(  x\right)  ,k+1\right)  $ which is
precisely $g\cdot\left(  f\left(  x\right)  ,k+1\right)  $. As such, $g$ is
continuous on $Y$.

Next note that
\begin{align*}
t^{-1}gt\cdot\left(  x,r\right)   &  =t^{-1}g\cdot\left(  x,r+1\right)  \\
&  =t^{-1}\cdot(\phi^{\left\lfloor r+1\right\rfloor }\left(  g\right)  \cdot
x,r+1)\\
&  =(\phi^{\left\lfloor r+1\right\rfloor }\left(  g\right)  \cdot x,r)\\
&  =(\phi^{\left\lfloor r\right\rfloor }\phi\left(  g\right)  \cdot
x,r)=\phi\left(  g\right)  \cdot(x,r)
\end{align*}
So the semidirect product relators are satisfied, and we have the desired
action. We leave it to the reader to check that this action is proper and cocompact.

Assertions \ref{Assertion 1 of constructing a G-semi-Z AR}%
)-\ref{Assertion 3 of constructing a G-semi-Z AR}) follow easily from further
application of Lemma \ref{Lemma: basic facts about cylinders and telescopes}.
\end{proof}

\subsection{Fixed sets and the existence $\phi$-variant
maps\label{Subsection: fixed sets and phi-variant maps}}

We now investigate situations where $\phi$-variant (and $\phi^{-1}$-variant)
maps can be shown to exist. Some commonly used assumptions about fixed point
sets will be useful.

\begin{definition}
Let $G$ be a group and $\mathcal{F}_{G}$ the collection of all finite
subgroups of $G$. A $G$-space $X$ is called $\mathcal{F}_{G}$%
\emph{-contractible} if, for every $H\in\mathcal{F}_{G}$, the fixed set
\[
X^{H}=\left\{  x\in X\mid hx=x\text{ for all }h\in H\right\}
\]
is contractible. If, in addition, each $X^{H}$ is an $AR$ the action is called
$\mathcal{F}_{G}^{AR}$\emph{-contractible}
\end{definition}

\begin{definition}
\label{Defn: EbarG spaces and complexes}A proper, rigid, $\mathcal{F}_{G}%
$-contractible $G$-complex is called an $\underline{E}G$\emph{-complex}. A
proper, $\mathcal{F}_{G}^{AR}$-contractible $G$-AR is called an $\underline{E}%
G_{\text{AR}}$\emph{-space}.
\end{definition}

\begin{remark}
By definition an $\underline{E}G$-complex $X$ is contractible, and by rigidity
each $X^{H}$ is a subcomplex. Since a cocompact $\underline{E}G$-complex $X$
is necessarily locally finite, it can be viewed as a special case of a
cocompact $\underline{E}G_{\text{AR}}$-space. (Recall that we require ARs to
be locally compact.)

\end{remark}

The following proposition provides some key examples; it will also be useful
for some of our later applications.

\begin{proposition}
\label{prop:ebarg} If $X$ is a proper cocompact $\mathcal{F}_{G}$-contractible
$G$-space, $\phi\in\operatorname*{Aut}\left(  G\right)  $, and $f:X\rightarrow
X$ is a $\phi$-variant map, then $Y=\operatorname*{Tel}_{f}\left(  X\right)  $
is a proper cocompact $\mathcal{F}_{G\rtimes_{\phi}%
\mathbb{Z}
}$-contractible $\left(  G\rtimes_{\phi}%
\mathbb{Z}
\right)  $-space. Moreover, if $X$ is an $\underline{E}G_{\text{AR}}$-space
then $Y$ an $\underline{E}(G\rtimes_{\phi}%
\mathbb{Z}
)_{\text{AR}}$-space and if $X$ is an $\underline{E}G$-complex and $f$ is
cellular, then $Y$ is an $\underline{E}(G\rtimes_{\phi}%
\mathbb{Z}
)$-complex.
\end{proposition}

\begin{proof}
For the initial assertion, let $H\leq G\rtimes_{\phi}%
\mathbb{Z}
$ be finite. Then $H\leq$ $G\leq G\rtimes_{\phi}%
\mathbb{Z}
$, and by hypothesis, $X^{H}$ is contractible, as is $X^{\phi^{i}(H)}$ for all
$i\in%
\mathbb{Z}
$. Furthermore, if $h\cdot x=x$, then $\phi\left(  h\right)  \cdot f\left(
x\right)  =f\left(  x\right)  $; so $f(X^{\phi^{i}(H)})\subseteq X^{\phi
^{i+1}(H)}$ for all $i\in%
\mathbb{Z}
$. It follows that $Y^{H}$ is the sub-mapping telescope defined by the
following subspaces and restriction maps.
\[
\cdots\overset{\left.  f\right\vert }{\longrightarrow}X^{\phi^{-2}%
(H)}\overset{\left.  f\right\vert }{\longrightarrow}X^{\phi^{-1}%
(H)}\overset{\left.  f\right\vert }{\longrightarrow}X^{H}\overset{\left.
f\right\vert }{\longrightarrow}X^{\phi(H)}\overset{\left.  f\right\vert
}{\longrightarrow}X^{\phi^{2}(H)}\overset{\left.  f\right\vert
}{\longrightarrow}\cdots
\]
By an application of Lemma
\ref{Lemma: basic facts about cylinders and telescopes}, since each
$X^{\phi^{i}(H)}$ is contractible, so is $Y^{H}$.

The additional assertions from similar reasoning.
\end{proof}

See \cite{LeNu03} for examples of groups that do not admit any cocompact
$\underline{E}G$-space but which do act cocompactly on a contractible CW-complex.

We now turn to the task of constructing $\phi$-variant maps.

\begin{proposition}
\label{PhiVariantMap} If $X$ is a proper cocompact $G$-space, $X^{\prime}$ is
an $\mathcal{F}_{G}$\emph{-contractible} $G$-space, and $\phi\in
\operatorname*{Aut}\left(  G\right)  $, then there exists a $\phi$-variant map
$f:X\rightarrow X^{\prime}$. If $X$ and $X^{\prime}$ are $\underline{E}%
G$-complexes, then $f$ can be chosen to be cellular.
\end{proposition}

We will use the following special case of \cite[Theorem A.2]{FaJo93}.

\begin{theorem}
\label{FarrellJonesTheorem} Let $X$ be a rigid $G$-CW complex with finite cell
stabilizers, and let $X^{\prime}$ an $\mathcal{F}_{G}$-contractible $G$-space.
Then there is a $G$-equivariant map from $X$ to $X^{\prime}$, and any two such
maps are homotopic through $G$-maps. If $X^{\prime}$ is a rigid $G$-CW complex
(hence an $\underline{E}G$-complex), then the maps and homotopies can be
chosen to be cellular.
\end{theorem}

\begin{remark}
\label{Remark: Adjusting the action}If $X$ and $X^{\prime}$ are rigid $G$-CW
complexes, Proposition \ref{PhiVariantMap} follows almost immediately from
Theorem \ref{FarrellJonesTheorem}. Let $X^{\prime\prime}$ denote $X^{\prime}$
with the modified $G$-action $g\cdot x\equiv\phi(g)\cdot x$, then note that a
$G$-equivariant map from $X$ to $X^{\prime\prime}$ is a $\phi$-variant map
from $X$ to $X^{\prime}$. In the more general case of Proposition
\ref{PhiVariantMap}, more work is needed. Much of our strategy is borrowed
from \cite{Ont05}.
\end{remark}

\begin{proof}
[Proof of Proposition \ref{PhiVariantMap}]Choose a proper metric $d$ on $X$ so
that $G$ acts by isometries \cite[Proposition 6.3]{GuMo19}. Since the orbit
$Gx$ of any $x\in X$ is discrete, there is a radius $r_{x}$ such that the
closed ball $B_{d}[x,r_{x}]\cap(Gx)=\{x\}$. Then, for all $g\in G$,
$B(x,\frac{r_{x}}{2})\cap gB(x,\frac{r_{x}}{2})=\varnothing$ or $gx=x$, with
the latter implying that $B(x,\frac{r_{x}}{2})=gB(x,\frac{r_{x}}{2})$. By
cocompactness, there is a finite collection $\mathcal{V=}\left\{  B_{d}%
(x_{i},\frac{r_{x_{i}}}{4})\right\}  _{i=1}^{k}$ of balls, with $Gx_{i}\neq
Gx_{j}$ for $i\neq j$, such that $\mathcal{U}:=\{gV|\ g\in G,\ V\in
\mathcal{V}\}$ is an open cover of $X$.\medskip

\noindent\textsc{Claim.}\textbf{ }\emph{Every ball }$U\in\mathcal{U}$\emph{
intersects only finitely many elements in }$\mathcal{U}$\emph{.} \medskip

If not, then there would be some $B(x,\frac{r_{x}}{4})$ and $B(y,\frac{r_{y}%
}{4})$ along with a sequence $\{g_{i}\}\subseteq G$ such that infinitely many
distinct $g_{i}B(x,\frac{r_{x}}{4})$ all have nonempty intersection with
$B(y,\frac{r_{y}}{4})$. It can be assumed that $r_{x}>r_{y}$. Thus, if
$g_{i}B(x,\frac{r_{x}}{4})$ and $g_{j}B(x,\frac{r_{x}}{4})$ both intersect
$B(y,\frac{r_{y}}{4})$, then $y\in g_{i}B(x,\frac{r_{x}}{2})\cap
g_{j}B(x,\frac{r_{x}}{2})$. This contradicts our assumption that
$B(x,\frac{r_{x}}{2})\cap gB(x,\frac{r_{x}}{2})=\varnothing$ or $B(x,\frac
{r_{x}}{2})=gB(x,\frac{r_{x}}{2})$. The claim follows.\medskip

Let $N(\mathcal{U})$ be the nerve of $\mathcal{U}$ and note that, by the above
claim, $N(\mathcal{U})$ is locally finite and finite-dimensional. By
construction $N(\mathcal{U})$ admits a proper, cocompact, simplicial
$G$-action; and since $gU\cap U\neq\varnothing$ implies that $gU=U$, this
action is rigid. Apply Theorem \ref{FarrellJonesTheorem} and the trick
described in Remark \ref{Remark: Adjusting the action} to obtain a $\phi
$-variant map $h:\left\vert N(\mathcal{U)}\right\vert \rightarrow X^{\prime}$,
where $\left\vert N(\mathcal{U)}\right\vert $ is the geometric realization of
$N(\mathcal{U)}$.

Next let $\beta:X\rightarrow\left\vert N(\mathcal{U)}\right\vert $ be the
barycentric map. In other words, $\beta(x)=\sum_{U\in\mathcal{U}}\lambda
_{U}(x)v_{U}$, where $\left\{  \lambda_{U}\right\}  $ is the partition of
unity defined by
\[
\lambda_{U_{0}}(x)=d(x,X-U_{0})/(\sum_{U\in\mathcal{U}}d(x,X-U))
\]
and $v_{U}$ denotes the vertex in $\left\vert N(\mathcal{U)}\right\vert $
defined by $U$. By construction, this map is $G$-equivariant, so
$f=h\circ\beta:X\rightarrow X^{\prime}$ is $\phi$-variant.
\end{proof}

\subsection{Proof of Theorems
\ref{Theorem: main theorem for groups with torsion} and
\ref{Theorem: Most general main theorem}}

We are ready to complete the main task of this section. Roughly speaking, we
aim to mimic, to the extent possible, the proof used in the torsion-free case.
In addition to $\phi$-variant and $\phi^{-1}$-variant maps $f:X\rightarrow X$
and $h:X\rightarrow X$, we will need analogs of the homotopies $A$ and $B$
used in Section
\ref{Section: Mapping tori and telescopes as classifying spaces}.

\begin{proposition}
\label{Prop: homotopies A and B in the torsion case}Let $f:X\rightarrow
X^{\prime}$ and $h:X^{\prime}\rightarrow X$ be $\phi$-variant and $\phi^{-1}%
$-variant maps, respectively, between proper cocompact $G$-ARs. Then there
exist a pair of bounded homotopies $A:X\times\left[  0,1\right]  \rightarrow
X$ and $B:X^{\prime}\times\left[  0,1\right]  \rightarrow X^{\prime}$ with
$A_{0}=\operatorname*{id}_{X}$, $A_{1}=hf$, $B_{0}=\operatorname*{id}%
_{X^{\prime}}$ and $B_{1}=fh$. As a result, $A$ and $B$ are proper homotopies
and $f$ and $h$ are proper homotopy inverses of one another. If $X$ and
$X^{\prime}$ are $\underline{E}G$-complex and $f$ and $h$ are cellular maps,
then $A$ and $B$ can be chosen so that $A_{t}$ and $B_{t}$ are $G$-equivariant
for all $t$.
\end{proposition}

\begin{proof}
As usual, choose proper metrics on $X$ and $X^{\prime}$ so that the actions
are geometric. It follows that $X$ and $X^{\prime}$ are uniformly contractible
and have finite macroscopic dimension. By cocompactness, since $hf$ is
$G$-equivariant, it is boundedly close to $\operatorname*{id}_{X}$. This
assures that $hf$ is a coarse equivalence, so we may apply \cite[Cor.5.3]%
{GuMo19} to obtain the bounded homotopy $A$. The same argument produces $B$.

When $X$ and $X^{\prime}$ are $\underline{E}G$-complexes, this proposition and
the $G$-equivariant conclusion follow from Theorem \ref{FarrellJonesTheorem}.
\end{proof}

\begin{proof}
[Proof of Theorem \ref{Theorem: Most general main theorem}]By hypothesis, we
have $\phi$-variant and $\phi^{-1}$-variant maps $f:X\rightarrow X$ and
$h:X\rightarrow X$, so we may apply Proposition
\ref{Prop: homotopies A and B in the torsion case} to obtain bounded
homotopies $A:X\times\left[  0,1\right]  \rightarrow X$ and $B:X\times\left[
0,1\right]  \rightarrow X$ with $A_{0}=\operatorname*{id}_{X}$, $A_{1}=hf$,
$B_{0}=\operatorname*{id}_{X}$ and $B_{1}=fh$. (If $X$ is an $\underline{E}%
G$-complex and $f$ and $h$ are cellular maps, then $A$ and $B$ can be chosen
so that $A_{t}$ and $B_{t}$ are $G$-equivariant for all $t$.)

Let $Y=\operatorname*{Tel}_{f}(X)$, endowed with the $\left(  G\rtimes_{\phi}%
\mathbb{Z}
\right)  $-action described in Theorem
\ref{Theorem: constructing a G-semi-Z AR}. Then construct a level-preserving
proper homotopy equivalence $v:Y\rightarrow X\times%
\mathbb{R}
$ by using formula (\ref{Definition of v}) directly (as opposed to obtaining
$v$ as a lift). In a similar manner, construct a proper homotopy inverse
$u:X\times%
\mathbb{R}
\rightarrow Y$ and homotopies $H:$ $u\circ v\simeq\operatorname*{id}_{Y}$;
$J:v\circ u\simeq\operatorname*{id}_{X\times%
\mathbb{R}
}$.

From this point on, the proof of Theorem \ref{Theorem A} can be used without changes.
\end{proof}

For the purpose of applications, the following variation on Theorem
\ref{Theorem: Most general main theorem} is probably the most useful.

\begin{theorem}
\label{Theorem: main theorem for groups with torsion}If a group $G$ (possibly
with torsion) admits a $\mathcal{Z}$-structure $(\overline{X},Z)$ where $X$ is
an $\underline{E}G_{AR}$-space, then every semidirect product of the form
$G\rtimes_{\phi}${$%
\mathbb{Z}
$} admits a $\mathcal{Z}$-structure $(\overline{Y},SZ)$ where $Y$ is an
$\underline{E}(G\rtimes_{\phi}{%
\mathbb{Z}
})_{AR}$ space. If $X$ is an $\underline{E}G$-complex, then $Y$ can be chosen
to be an $\underline{E}(G\rtimes_{\phi}{%
\mathbb{Z}
})$-complex.
\end{theorem}

\begin{proof}
The initial assertion combines Theorem
\ref{Theorem: Most general main theorem} and Proposition~\ref{PhiVariantMap}.
The latter assertion adds in Proposition \ref{prop:ebarg}.
\end{proof}

\begin{remark}
Theorem \ref{Theorem: main theorem for groups with torsion} can be applied
whenever $G$ is hyperbolic, CAT(0), or systolic. We save the details of that
discussion for Section \ref{Section: EZ-bar structures} where even stronger
results will be obtained.
\end{remark}

\subsection{Some closing comments on the proof of Theorem
\ref{Theorem: Most general main theorem}}

Unlike the torsion-free case, we did not claim that the map $v:Y\rightarrow
X\times%
\mathbb{R}
$, used in the proof of Theorem \ref{Theorem: Most general main theorem}, is
$G$-equivariant. That is due to the fact that $A:X\times\left[  0,1\right]
\rightarrow X$ was not constructed to be $G$-equivariant---except in cases
where $X$ is a CW-complex (see remark below). That is not an issue for the
proofs just completed. For later applications, however, it is be useful to
note that, since $\left.  A\right\vert _{X\times\left\{  0,1\right\}  }$ is
$G$-equivariant, $v$ is $G$-equivariant when restricted to the integer levels
of $Y$. For completeness, we add a quick proof.

\begin{lemma}
\label{lem:integralequivariance} Whenever $k$ is an integer, $v$ has the
property that
\[
v(g\cdot\lceil x,k\rceil)=g\cdot(v\lceil x,k\rceil).
\]

\end{lemma}

\begin{proof}
First assume $k\geq0$. Then%

\begin{align*}
v(g\cdot\lceil x,k\rceil)  &  =v\lceil\phi^{k}\left(  g\right)  \cdot
x,k\rceil\\
&  =(h^{k}(\phi^{k}(g)\cdot x),k)\\
&  =(\phi^{-k}(\phi^{k}(g))\cdot h^{k}(x),k)\qquad\text{(by }\phi
^{-1}\text{-variance of }h\text{)}\\
&  =(g\cdot h^{k}(x),k)\\
&  =g\cdot v\lceil x,k\rceil.
\end{align*}
When $k\leq0$, the argument is similar.
\end{proof}

\begin{remark}
\label{Remark: Equivariance of maps and homotopies in CW cases}For later use,
we make some additional observations about the above proof in cases where $X$
is an $\underline{E}G$-complex. In that case, $f$ and $h$ can be chosen to be
cellular maps by Theorem \ref{FarrellJonesTheorem}. so, as noted in
Proposition \ref{Prop: homotopies A and B in the torsion case}, we may choose
$A$ and $B$ such that $A_{t}$ and $B_{t}$ are $G$-equivariant for all $t$. In
that case, the maps $v$ and $u$ are $G$-equivariant, so by another application
of Theorem~\ref{FarrellJonesTheorem}, we may choose $H$ and $J$ to be
$G$-equivariant as well.
\end{remark}


\section{$E\mathcal{Z}$-structures\label{Section: EZ-structures}}

\label{sec:equivariant}

We now turn to the study of $E\mathcal{Z}$-structures. Our primary goal is a
proof of the following theorem.

\begin{theorem}
\label{Theorem: EZ-structures}If a group $G$ admits an $E\mathcal{Z}%
$-structure $(\overline{X},Z)$, $\phi\in\operatorname*{Aut}(G)$, and there
exist $\phi$-variant and $\phi^{-1}$-variant maps $f:X\rightarrow X$ and
$g:X\rightarrow X$, respectively, which extend continuously to maps
$\overline{f}:\overline{X}\rightarrow\overline{X}$ and $\overline{g}%
:\overline{X}\rightarrow\overline{X}$, then $G\rtimes_{\phi}${$%
\mathbb{Z}
$} admits an $E\mathcal{Z}$-structure with boundary $SZ$.
\end{theorem}

From here it is easy to deduce Theorem \ref{Theorem D} from the introduction.
We save that discussion for the end of this section.

The following pair of elementary lemmas helps to clarify some technical
questions. The second offers an alternative hypothesis for our main theorem.

\begin{lemma}
\label{Lemma: continuous extensions 1}Let $\overline{X}=X\sqcup Z$ and
$\overline{Y}=Y\sqcup Z^{\prime}$ be controlled compactifications of proper
metric spaces $X$ and $Y$ and let $f:X\rightarrow Y$ be a proper map which
admits a continuous extension $\overline{f}:\overline{X}\rightarrow
\overline{Y}$. Then $\overline{f}\left(  Z\right)  \subseteq Z^{\prime}$ and
every continuous map $f^{\prime}:X\rightarrow Y$ which is boundedly close to
$f$ (as measured in $\left(  Y,d_{Y}\right)  $) extends continuously to
$\overline{f^{\prime}}:\overline{X}\rightarrow\overline{Y}$ by letting
$\left.  \overline{f^{\prime}}\right\vert _{Z}=\left.  \overline{f}\right\vert
_{Z}$. Continuous extensions of this type are unique when they exist.
\end{lemma}

\begin{lemma}
\label{Lemma: continuous extensions 2}Let $\overline{X}=X\sqcup Z$ and
$\overline{Y}=Y\sqcup Z^{\prime}$ be controlled compactifications of proper
metric spaces $X$ and $Y$ and let $f:X\rightarrow Y$ be a coarse equivalence
with coarse inverse $h:Y\rightarrow X$. Then the following are equivalent.

\begin{enumerate}
\item There exists a continuous extension $\overline{f}:\overline
{X}\rightarrow\overline{Y}$ of $f$ which takes $Z$ homeomorphically onto
$Z^{\prime}$.

\item There exist continuous extensions $\overline{f}:\overline{X}%
\rightarrow\overline{Y}$ and $\overline{h}:\overline{Y}\rightarrow\overline
{X}$ of $f$ and $g$.
\end{enumerate}
\end{lemma}

Suppose now that $G$ admits an $E\mathcal{Z}$-structure $\left(  \overline
{X},Z\right)  $ and $\phi\in\operatorname*{Aut}(G)$. Assume also the existence
of $\phi$-variant and $\phi^{-1}$-variant maps $f:X\rightarrow X$ and
$h:X\rightarrow X$ which extend continuously to maps $\overline{f}%
:\overline{X}\rightarrow\overline{X}$ and $\overline{h}:\overline
{X}\rightarrow\overline{X}$. By Lemma \ref{Lemma: continuous extensions 2} the
restrictions $f_{Z}:Z\rightarrow Z$ and $h_{Z}:Z\rightarrow Z$ are
homeomorphisms. As usual, we let $Y=\operatorname*{Tel}_{f}\left(  X\right)  $
with the $\left(  G\rtimes_{\phi}%
\mathbb{Z}
\right)  $-action defined earlier. To prove Theorem
\ref{Theorem: EZ-structures}, we will first define a $\left(  G\rtimes_{\phi}%
\mathbb{Z}
\right)  $-action on $SZ$; then we will show that this action continuously
extends the action on $Y$. The latter of these tasks is surprisingly delicate.

The notational conventions established in Section
\ref{Subsection: Mapping cylinders, mapping tori, and mapping telescopes} will
be especially useful in this section.

\subsection{The $\left(  G\rtimes_{\phi}%
\mathbb{Z}
\right)  $-action on $SZ$}

Working with the presentation
\[
G\rtimes_{\phi}%
\mathbb{Z}
=\langle G,t\ |\ t^{-1}gt=\phi(g)\forall g\in G\}\rangle
\]
it suffices to specify a $G$-action on $SZ$ together with a self-homeomorphism
of $SZ$ corresponding to the generator $t$, and to check that the conjugating
relators hold. To begin with, let $g$ represent both an element of $G$ and the
corresponding self-homeomorphism of $X$; let $\overline{g}:\overline
{X}\rightarrow\overline{X}$ denote the extension of $g$ implied by the
$E\mathcal{Z}$-structure; and let $g_{Z}=\left.  \overline{g}\right\vert _{Z}%
$, a self-homeomorphism of $Z$.

We choose the $G$-action on $SZ$ to be the suspension of the $G$-action on
$Z$. In other words, for each $g\in G$, define $g_{SZ}:SZ\rightarrow SZ$ by
$g_{SZ}\cdot\left\langle z,r\right\rangle =\left\langle g_{Z}\cdot
z,r\right\rangle $. Then define $t_{SZ}:SZ\rightarrow SZ$ to be the suspension
of $h_{Z}$, i.e., $t_{SZ}\cdot\left\langle z,r\right\rangle =\left\langle
h_{Z}\left(  z\right)  ,r\right\rangle $. Since $h:X\rightarrow X$ is
$\phi^{-1}$-equivariant, $h^{-1}gh=\phi\left(  g\right)  $ as a
self-homeomorphism of $X$, for all $g\in G$. Then, since extensions over
$\overline{X}$ are unique when they exist, $h_{Z}^{-1}\circ g_{Z}\circ
h_{Z}=\phi\left(  g\right)  _{Z}$ as self-homeomorphisms of $Z$. It follows
that $t_{SZ}^{-1}\circ g_{SZ}\circ t_{SZ}=\phi\left(  g\right)  _{SZ}$, so the
conjugating relators of $G\rtimes_{\phi}%
\mathbb{Z}
$ are satisfied.

\subsection{The $\left(  G\rtimes_{\phi}%
\mathbb{Z}
\right)  $-action on $\overline{Y}$}

To save on notation, now let $g$ and $t$ denote the self-homeomorphisms of $Y$
defined in equations (\ref{Defn: G acting on Y}) and
(\ref{Defn: t acting on Y}), and let $\overline{g}:=g\sqcup g_{SZ}$ and
$\overline{t}:=t\sqcup t_{SZ}$. Our task is complete if we can show (or
arrange) that these maps are continuous on $\overline{Y}$. Recall that
$\overline{Y}=Y\sqcup SZ$ was defined by first compactifying $X\times%
\mathbb{R}
$ to $\overline{X\times%
\mathbb{R}
}=(X\times%
\mathbb{R}
)\sqcup SZ$ (in a very precise manner) then applying the pullback
compactification (see Remark \ref{Remark: pullback compactification}) to $Y$
using the map $v:Y\rightarrow X\times%
\mathbb{R}
$ described in formula (\ref{Definition of v}). To prove continuity, we will
use a pair of simple facts. The first is a general property of pullback
compactifications; the second follows directly from Definition
\ref{DefineProductBoundaryTopology}.

\begin{itemize}
\item A sequence $\left\{  \left\lceil x_{i},r_{i}\right\rceil \right\}  $ in
$Y$ converges in $\overline{Y}$ to $\left\langle z,\mu\right\rangle \in SX$ if
and only if $\left\{  v\left(  \left\lceil x_{i},r_{i}\right\rceil \right)
\right\}  $ converge to $\left\langle z,\mu\right\rangle $ in $\overline
{X\times%
\mathbb{R}
}$.

\item A sequence $\left\{  \left(  x_{i},r_{i}\right)  \right\}  $ in $X\times%
\mathbb{R}
$ converges in $\overline{X\times%
\mathbb{R}
}$ to $\left\langle z,\mu\right\rangle \in SX$ if and only if $\left\{
x_{i}\right\}  $ converges to $z$ in $\overline{X}$ and the sequence of slopes
$\left\{  \mu\left(  x_{i},r_{i}\right)  \right\}  $ converges to $\mu$.
\end{itemize}

The following lemma allows us to focus on sequences with integral second
coordinates which, in turn, allows us to make use of Lemma
\ref{lem:integralequivariance}.

\begin{lemma}
\label{Lemma: restricting to integers}Let $\left\{  \left\lceil x_{i}%
,r_{i}\right\rceil \right\}  $ be a sequence in $Y=\operatorname*{Tel}%
_{f}\left(  X\right)  $. Then $v\left(  \left\lceil x_{i},r_{i}\right\rceil
\right)  \rightarrow\left\langle z,\mu\right\rangle $ in $\overline{X\times%
\mathbb{R}
}$ if and only if $v\left(  \left\lceil x_{i},\left\lfloor r_{i}\right\rfloor
\right\rceil \right)  \rightarrow\left\langle z,\mu\right\rangle $.
\end{lemma}

\begin{proof}
For each $i$, the entire mapping cylinder line containing $\left\lceil
x_{i},r_{i}\right\rceil $ and $\left\lceil x_{i},\left\lfloor r_{i}%
\right\rfloor \right\rceil $ lies in some translate of the fundamental domain
defined at the beginning of Section \ref{sec:proof}. The nullity condition
arranged in that section assures that the diameters of the $v$-images of these
fundamental domains, measured in $\overline{X\times%
\mathbb{R}
}$, approach $0$ as they are pushed to infinity. The conclusion follows easily.
\end{proof}

\subsection{An important special case}

To complete the proof of Theorem \ref{Theorem: EZ-structures}, it remains to
verify the continuity of $\overline{g}$ and $\overline{t}$. This is a delicate
matter. In fact, without an adjustment to the earlier construction, continuity
could fail. The necessary adjustment involves the choice of slope function
defined in Section \ref{sec:controlled}. To make the new choice as intuitive
as possible, we start with a key special case:\medskip\ 

\noindent Assume that the original proper metric space $\left(  X,d\right)  $,
on which $G$ is acting geometrically, is a \emph{quasi-geodesic} space.
\medskip

This condition holds in nearly all commonly studied cases (it is built-in when
$G$ is hyperbolic or CAT(0) and can be arranged whenever $G$ is torsion-free),
but it is traditionally not required in the definition of a $\mathcal{Z}%
$-structure. After handling the special case, we will return to address the
generic case.

The usefulness of the quasi-geodesic hypothesis is that, by the
\v{S}varc-Milnor Lemma, it ensures that the $\phi$- and $\phi^{-1}$-variant
maps $f:X\rightarrow X$ and $h:X\rightarrow X$ are quasi-isometries. As such,
we may choose a single pair of constants $K\geq1$ and $\varepsilon\geq0$ such
that
\begin{equation}
\frac{1}{K}d\left(  x,y\right)  -\varepsilon\leq d\left(  f\left(  x\right)
,f\left(  y\right)  \right)  \leq Kd\left(  x,y\right)  +\varepsilon
\label{formula: QI inequalities for f}%
\end{equation}
and%
\begin{equation}
\frac{1}{K}d\left(  x,y\right)  -\varepsilon\leq d\left(  h\left(  x\right)
,h\left(  y\right)  \right)  \leq Kd\left(  x,y\right)  +\varepsilon
\label{formula: QI inequalities for h}%
\end{equation}
for all $x,y\in X$.

Now let us return to the continuous, monotone increasing function
$\psi:[0,\infty)\rightarrow\lbrack0,\infty)$ described in item
(\ref{Item: Defn of psi}) of Section \ref{sec:controlled}. The two properties
assigned to $\psi$ were the following:

\begin{enumerate}
\item[a)] $\psi(s)\geq\max\left\{  \eta(s),\lambda(s)\right\}  $ , and

\item[b)] $\psi(s+1)\geq3\psi(s)$ for all $s\geq0$
\end{enumerate}

\noindent For the purposes of this section, we introduce a third requirement
that can easily be added to the above.

\begin{enumerate}
\item[c)] \label{Item c)}$\psi$ is smooth with $\psi^{\prime}\left(  s\right)
\geq1$ for all $s\geq0$.
\end{enumerate}

Now define $\Psi:[0,\infty)\rightarrow\lbrack0,\infty)$ by $\Psi\left(
s\right)  =e^{\psi\left(  s\right)  }$ and notice that $\Psi$ also satisfies
conditions a)-c). As such, we can go back to Section \ref{sec:controlled} and
replace $\psi$ with $\Psi$. That will change the slope function (hence, the
way $SZ$ is glued to $X\times%
\mathbb{R}
$ to obtain $\overline{X\times%
\mathbb{R}
}$) but the proofs that follow remain valid. Given that $\Psi^{-1}\left(
s\right)  =\psi^{-1}\left(  \log s\right)  $, the new slope formula takes the
form%
\begin{align*}
\mu\left(  x,r\right)   &  =\frac{r}{\log\left(  \Psi^{-1}\left(  d\left(
x,x_{0}\right)  +\Psi\left(  0\right)  +1\right)  \right)  }\\
&  =\frac{r}{\log(\psi^{-1}(\log\left(  d\left(  x,x_{0}\right)  +\Psi\left(
0\right)  +1\right)  ))}%
\end{align*}
\label{new slope formula}

Most of our continuity arguments hinge on calculations of limits. The
following lemma will aid in several calculations.

\begin{lemma}
\label{Lemma: technical limit for QI case}Let $A_{1},A_{2},B_{1},B_{2}\in%
\mathbb{R}
$ with $A_{1},A_{2}>0$ and let $\theta:[0,\infty)\rightarrow\lbrack0,\infty)$
be a smooth monotone increasing function such that $\theta\left(  x\right)
\rightarrow\infty$ and $\theta^{\prime}\left(  x\right)  \leq1$ for all
sufficiently large $x$. Then%

\[
\lim_{x\rightarrow\infty}\frac{\theta(\log\left(  A_{1}x+B_{1}\right)
))}{\theta(\log\left(  A_{2}x+B_{2}\right)  ))}=1
\]

\end{lemma}

\begin{proof}
First note that if $\sigma:[0,\infty)\rightarrow%
\mathbb{R}
$ is a smooth monotone increasing function with $\sigma^{\prime}\left(
x\right)  \leq1$ for all sufficiently large $x$, then
\[
\sigma\left(  x\right)  -\left\vert c\right\vert \leq\sigma\left(  x+c\right)
\leq\sigma\left(  x\right)  +\left\vert c\right\vert
\]
for sufficiently large $x$. This fact will be applied to both $\theta$ and the
log function.

In particular, since
\[
\log\left(  A_{i}x+B_{i}\right)  =\log\left(  A_{i}(x+\frac{B_{i}}{A_{i}%
})\right)  =\log\left(  x+\frac{B_{i}}{A_{i}}\right)  +\log A_{i}%
\]
then
\[
\log\left(  x\right)  +\log A_{i}-\left\vert \frac{B_{i}}{A_{i}}\right\vert
\leq\log\left(  A_{i}x+B_{i}\right)  \leq\log\left(  x\right)  +\log
A_{i}+\left\vert \frac{B_{i}}{A_{i}}\right\vert
\]
Letting $d_{i}:=\max\left\{  \left\vert \log A_{i}-\left\vert \frac{B_{i}%
}{A_{i}}\right\vert \right\vert ,\left\vert \log A_{i}+\left\vert \frac{B_{i}%
}{A_{i}}\right\vert \right\vert \right\}  $, we can conclude that
\[
\theta\left(  \log\left(  x\right)  \right)  -d_{i}\leq\theta\left(
\log\left(  A_{i}x+B_{i}\right)  \right)  \leq\theta\left(  \log\left(
x\right)  \right)  +d_{i}%
\]
Applying this inequality multiple times yields%
\[
\frac{\theta\left(  \log\left(  x\right)  \right)  -d_{1}}{\theta\left(
\log\left(  x\right)  \right)  +d_{2}}\leq\frac{\theta(\log\left(
A_{1}x+B_{1}\right)  ))}{\theta(\log\left(  A_{2}x+B_{2}\right)  ))}\leq
\frac{\theta\left(  \log\left(  x\right)  \right)  +d_{1}}{\theta\left(
\log\left(  x\right)  \right)  -d_{2}}%
\]
for sufficiently large $x$. Our main assertion follows easily.
\end{proof}

We are now ready to proceed with proofs of the continuity of $\overline{g}$
and $\overline{t}$ and, hence, the special case of Theorem
\ref{Theorem: EZ-structures}.

\begin{claim}
\label{Claim: continuity of g-bar}For each $g\in G$, the function
$\overline{g}:\overline{Y}\rightarrow\overline{Y}$, defined above, is continuous.
\end{claim}

\begin{proof}
Since $\left.  \overline{g}\right\vert _{Y}=g$ is continuous, it suffices to
check continuity at points of $SZ$, and since $\left.  \overline{g}\right\vert
_{SZ}=g_{SZ}$ is continuous, it suffices consider the effect of $\overline{g}$
on sequences $\left\{  \left\lceil x_{i},r_{i}\right\rceil \right\}  $ in $Y$
which converge in $\overline{Y}$ to a point $\left\langle z,\mu\right\rangle
\in SX$. Specifically, we need to show that
\[
\{\lceil x_{i},r_{i}\rceil\}\rightarrow\langle z,\mu\rangle\implies
\{\overline{g}\cdot\lceil x_{i},r_{i}\rceil\}\rightarrow\langle g_{Z}\cdot
z,\mu\rangle
\]

By Lemma \ref{Lemma: restricting to integers}, we may replace each $r_{i}$
with $\left\lfloor r_{i}\right\rfloor $. To simplify notation, we simply
assume that each $r_{i}$ is an integer.\medskip

\textsc{Case 1.} $0<\mu\leq\infty$\medskip

Then $r_{i}$ is eventually non-negative, so we can assume $r_{i}\geq0$ for all
$i$. Applying formula (\ref{Definition of v}) and the fact that $A_{0}%
=\operatorname*{id}_{X}$, we have $v\left(  \lceil x_{i},r_{i}\rceil\right)
=\left(  h^{r_{i}}\left(  x_{i}\right)  ,r_{i}\right)  $. By the above bullet
points, $h^{r_{i}}\left(  x_{i}\right)  \rightarrow z$ in $\overline{X}$ and
$\mu\left(  h^{r_{i}}\left(  x_{i}\right)  ,r_{i}\right)  \rightarrow\mu$. If
we apply $\overline{g}$ to $\{\lceil x_{i},r_{i}\rceil\}$ and $\left\langle
z,\mu\right\rangle $, we get $\{\overline{g}\cdot\lceil x_{i},r_{i}%
\rceil\}=\left\{  \left\lceil \phi^{r_{i}}\left(  g\right)  \cdot x_{i}%
,r_{i}\right\rceil \right\}  $ and $\overline{g}\cdot\left\langle
z,\mu\right\rangle =\left\langle g_{Z}\cdot z,\mu\right\rangle $. By the first
bullet point, it remains to check that $v\left(  \left\lceil \phi^{r_{i}%
}\left(  g\right)  \cdot x_{i},r_{i}\right\rceil \right)  \rightarrow
\left\langle g_{Z}\cdot z,\mu\right\rangle $ in $\overline{X\times%
\mathbb{R}
}$. Now%
\begin{align*}
v\left(  \left\lceil \phi^{r_{i}}\left(  g\right)  \cdot x_{i},r_{i}%
\right\rceil \right)   &  =\left(  h^{r_{i}}\left(  \phi^{r_{i}}\left(
g\right)  \cdot x_{i}\right)  ,r_{i}\right) \\
&  =(\phi^{-r_{i}}(\phi^{r_{i}}(g))\cdot h^{r_{i}}(x),r_{i})\qquad\text{(by
}\phi^{-1}\text{-variance of }h\text{)}\\
&  =\left(  g\cdot h^{r_{i}}\left(  x_{i}\right)  ,r_{i}\right)
\end{align*}

Since $h^{r_{i}}\left(  x_{i}\right)  \rightarrow z$ and $\overline
{g}:\overline{X}\rightarrow\overline{X}$ is continuous, $g\cdot h^{r_{i}%
}\left(  x_{i}\right)  \rightarrow\overline{g}\cdot z=g_{Z}\cdot z$ in
$\overline{X}$, as desired.

It remains to show that $\mu\left(  g\cdot h^{r_{i}}\left(  x_{i}\right)
,r_{i}\right)  \rightarrow\mu$. We already know that $\mu\left(  h^{r_{i}%
}\left(  x_{i}\right)  ,r_{i}\right)  \rightarrow\mu$, so by comparing these
two sequences and applying formula (\ref{new slope formula}), it suffices to
show that
\begin{equation}
\frac{\log(\psi^{-1}(\log(d(h^{r_{i}}\left(  x_{i}\right)  ,x_{0}%
)+\Psi(0))+1)))}{\log(\psi^{-1}(\log(d(g\cdot h^{r_{i}}\left(  x_{i}\right)
,x_{0})+\Psi(0))+1)))}\rightarrow
1\label{formula: limit for continuity of g-bar}%
\end{equation}
Since $g$ is an isometry of $X$, the triangle inequality assures us that
\[
d\left(  h^{r_{i}}\left(  x_{i}\right)  ,x_{0}\right)  -d\left(  x_{0},g\cdot
x_{0}\right)  \leq d(g\cdot h^{r_{i}}\left(  x_{i}\right)  ,x_{0})\leq
d\left(  h^{r_{i}}\left(  x_{i}\right)  ,x_{0}\right)  +d\left(  x_{0},g\cdot
x_{0}\right)
\]
This allows us to squeeze limit \ref{formula: limit for continuity of g-bar}
between a pair of limits of the type addressed in Lemma
\ref{Lemma: technical limit for QI case}. In both cases $\theta=\log\circ
\psi^{-1}$, $A_{1}=A_{2}=1$, and $B_{1}=\Psi(0))+1$. For the lower limit, let
$B_{2}=\Psi(0)+1+d\left(  x_{0},g\cdot x_{0}\right)  $ and for the upper
limit, let $B_{2}=\Psi(0)+1-d\left(  x_{0},g\cdot x_{0}\right)  $.\medskip

\textsc{Case 2.} $-\infty\leq\mu<0$\medskip

Then $r_{i}$ is eventually negative, so we can assume $r_{i}<0$ for all $i$,
so formula (\ref{Definition of v}) yields $v\left(  \lceil x_{i},r_{i}%
\rceil\right)  =\left(  f^{\left\vert r_{i}\right\vert }\left(  x_{i}\right)
,r_{i}\right)  $. By the earlier bullet points, $f^{\left\vert r_{i}%
\right\vert }\left(  x_{i}\right)  \rightarrow z$ in $\overline{X}$ and
$\mu\left(  f^{\left\vert r_{i}\right\vert }\left(  x_{i}\right)
,r_{i}\right)  \rightarrow\mu$. As in Case 1, $\{\overline{g}\cdot\lceil
x_{i},r_{i}\rceil\}=\left\{  \left\lceil \phi^{r_{i}}\left(  g\right)  \cdot
x_{i},r_{i}\right\rceil \right\}  $ and $\overline{g}\cdot\left\langle
z,\mu\right\rangle =\left\langle g_{Z}\cdot z,\mu\right\rangle $, so it
remains to check that $v\left(  \left\lceil \phi^{r_{i}}\left(  g\right)
\cdot x_{i},r_{i}\right\rceil \right)  \rightarrow\left\langle \overline
{g}\cdot z,\mu\right\rangle $ in $\overline{X\times%
\mathbb{R}
}$. In this case,%
\begin{align*}
v\left(  \left\lceil \phi^{r_{i}}\left(  g\right)  \cdot x_{i},r_{i}%
\right\rceil \right)   &  =\left(  f^{\left\vert r_{i}\right\vert }\left(
\phi^{-\left\vert r_{i}\right\vert }\left(  g\right)  \cdot x_{i}\right)
,r_{i}\right) \\
&  =(\phi^{\left\vert r_{i}\right\vert }(\phi^{-\left\vert r_{i}\right\vert
}(g))\cdot f^{\left\vert r_{i}\right\vert }(x),r_{i})\qquad\text{(by }%
\phi\text{-variance of }g\text{)}\\
&  =\left(  g\cdot f^{\left\vert r_{i}\right\vert }\left(  x_{i}\right)
,r_{i}\right)
\end{align*}

The rest of the proof follows the reasoning used in Case 1, with$f^{\left\vert
r_{i}\right\vert }$ replacing $h^{r_{i}}$.\medskip

\textsc{Case 3.} $\mu=0$\medskip

Split the sequence $\{\lceil x_{i},r_{i}\rceil\}$ into a pair of subsequences,
one with all $r_{i}\geq0$ and the other with $r_{i}<0$. Then apply the
arguments used in Cases 1 and 2 to the subsequence individually.
\end{proof}

\begin{claim}
The function $\overline{t}:\overline{Y}\rightarrow\overline{Y}$, defined
above, is continuous.
\end{claim}

\begin{proof}
Following the same strategy used above, we will show that
\[
\{\lceil x_{i},r_{i}\rceil\}\rightarrow\langle z,\mu\rangle\implies
\{\overline{t}\cdot\lceil x_{i},r_{i}\rceil\}\rightarrow\langle h_{Z}\left(
z\right)  ,\mu\rangle
\]

As before, we may assume each $r_{i}$ is an integer.\medskip

\textsc{Case 1.} $0<\mu\leq\infty$\medskip

Then $r_{i}$ is eventually non-negative, so we can assume $r_{i}\geq0$ for all
$i$. Then $v\left(  \lceil x_{i},r_{i}\rceil\right)  =\left(  h^{r_{i}}\left(
x_{i}\right)  ,r_{i}\right)  $ so $h^{r_{i}}\left(  x_{i}\right)  \rightarrow
z$ in $\overline{X}$ and $\mu\left(  h^{r_{i}}\left(  x_{i}\right)
,r_{i}\right)  \rightarrow\mu$. Applying $\overline{t}$ to $\{\lceil
x_{i},r_{i}\rceil\}$ and $\left\langle z,\mu\right\rangle $, we get $\{\lceil
x_{i},r_{i}+1\rceil\}$ and $\left\langle h_{Z}(z),\mu\right\rangle $,
respectively. By the first bullet point, it remains to show that $v\left(
\left\lceil x_{i},r_{i}+1\right\rceil \right)  \rightarrow\left\langle
h_{Z}(z),\mu\right\rangle $ in $\overline{X\times%
\mathbb{R}
}$. By formula (\ref{Definition of v}), we must show that $\left(  h^{r_{i}%
+1}\left(  x_{i}\right)  ,r_{i}+1\right)  \rightarrow\left\langle h_{Z}%
(z),\mu\right\rangle $ in $\overline{X\times%
\mathbb{R}
}$.

Since $h^{r_{i}}\left(  x_{i}\right)  \rightarrow z$ and $\overline
{h}:\overline{X}\rightarrow\overline{X}$ is continuous, $h^{r_{i}+1}\left(
x_{i}\right)  \rightarrow h_{Z}(z)$ in $\overline{X}$, as desired. It remains
to show that $\mu\left(  h^{r_{i}+1}\left(  x_{i}\right)  ,r_{i}+1\right)
\rightarrow\mu$. We already know that $\mu\left(  h^{r_{i}}\left(
x_{i}\right)  ,r_{i}\right)  \rightarrow\mu$, so by comparing these two
sequences and applying formula (\ref{new slope formula}), it suffices to show
that\
\begin{equation}
\frac{\log(\psi^{-1}(\log(d(h^{r_{i}}\left(  x_{i}\right)  ,x_{0}%
)+\Psi(0))+1)))}{\log(\psi^{-1}(\log(d(h^{r_{i}+1}\left(  x_{i}\right)
,x_{0})+\Psi(0))+1)))}\rightarrow1
\label{formula: limit for continuity of t-bar}%
\end{equation}
By applying (\ref{formula: QI inequalities for h}) we have%
\[
\frac{1}{K}d(h^{r_{i}}\left(  x_{i}\right)  ,x_{0})-\varepsilon\leq
d(h^{r_{i}+1}\left(  x_{i}\right)  ,x_{0})\leq Kd(h^{r_{i}}\left(
x_{i}\right)  ,x_{0})+\varepsilon
\]
which allows us squeeze limit (\ref{formula: limit for continuity of t-bar})
between a pair of limits like those addressed in Lemma
\ref{Lemma: technical limit for QI case}. For both limits $\theta=\log
\circ\psi^{-1}$, $A_{1}=1$, and $B_{1}=\Psi(0))+1$. For the lower limit
$A_{2}=K$ and $B_{2}=\Psi(0))+1+\varepsilon$, while for the upper limit
$A_{2}=1/K$ and $B_{2}=\Psi(0))+1-\varepsilon$.\medskip

\textsc{Case 2.} $0<\mu\leq\infty$\medskip

Then $r_{i}$ is eventually negative so we assume $r_{i}<0$ for all $i$. Now
$v\left(  \left\lceil x_{i},r_{i}\right\rceil \right)  =\left(  f^{\left\vert
r_{i}\right\vert }\left(  x_{i}\right)  ,r_{i}\right)  $ and $v\left(
\left\lceil x_{i},r_{i}+1\right\rceil \right)  =\left(  f^{\left\vert
r_{i}+1\right\vert }\left(  x_{i}\right)  ,r_{i}+1\right)  $, so we rely on
the continuity of $\overline{f}:\overline{X}\rightarrow\overline{X}$ and
inequalities (\ref{formula: QI inequalities for f}) instead of the analogs for
$h$. Everything else follows as in Case 1.\medskip

\textsc{Case 3.} $\mu=0$\medskip

As we did earlier, split $\{\lceil x_{i},r_{i}\rceil\}$ into subsequences to
which the arguments of Cases 1 and 2 can be applied.
\end{proof}

\begin{remark}
\label{Remark: Extension of G-action to XxR-bar}In proving Claim
\ref{Claim: continuity of g-bar}, we showed that if a sequence $\left\{
\left(  x_{i},r_{i}\right)  \right\}  $ converges to $\left\langle
z,\mu\right\rangle $ in $X\times%
\mathbb{R}
$ and $g\in G$, then $\left\{  \left(  g\cdot x_{i},r_{i}\right)  \right\}
\rightarrow\left\langle g_{Z}\cdot z,\mu\right\rangle $. A useful, and not
entirely obvious, corollary is that, when we have an $E\mathcal{Z}$-structure
$\left(  \overline{X},Z\right)  $ on $G$, the product $G$-action on $X\times%
\mathbb{R}
$ extends to our compactification $\overline{X\times%
\mathbb{R}
}$ by suspending the given $G$-action on $Z$.
\end{remark}

\subsection{Spaces that are not quasi-geodesic}

It is possible to meet the hypotheses of Theorem~\ref{Theorem: EZ-structures}
with respect to a space $X$ such that the $\phi$-variant map is not a
quasi-isometry. Indeed, when $(X,d)$ is not quasi-geodesic, we can only
conclude that $\phi$-variant maps are coarse equivalences. See \cite{BDM07}
for a discussion of this topic.

\begin{definition}
Let $X,W$ be metric spaces. A map $\phi:X\rightarrow W$ is a \emph{coarse
embedding} if there exist increasing \emph{control functions} $\rho_{-}%
,\rho_{+}:[0,\infty)\rightarrow\lbrack0,\infty)$ with $\lim_{r\rightarrow
\infty}\rho_{-}(r)=\infty$ such that for all $x,x^{\prime}\in X$%

\[
\rho_{-}(d_{X}(x,x^{\prime}))\leq d_{W}(\phi(x),\phi(x^{\prime}))\leq\rho
_{+}(d_{X}(x,x^{\prime})).
\]
A coarse embedding is called a \emph{coarse equivalence} if it is quasi-onto,
i.e., there exists $C>0$ such that for all $w\in W$, $d_{W}(w,\phi(X))<C$.
\end{definition}

\begin{example}
As a simple illustration, consider $\left(
\mathbb{R}
,d^{\prime}\right)  $ where $d^{\prime}\left(  r,s\right)  =\log\left(
1+\left\vert x-y\right\vert \right)  $. The usual $%
\mathbb{Z}
$-action is still by isometries, but the standard orbit map $\lambda:%
\mathbb{Z}
\rightarrow%
\mathbb{R}
$ is not a quasi-isometry. For a more extreme example, let $d^{\prime\prime
}=\log\left(  1+d^{\prime}\right)  $, etc.
\end{example}

Faced with more general control functions---imagine $\rho_{+}$ being
super-exponential and $\rho_{-}$ growing slower than an iterated
logarithm---the adjustment made to the function $\psi$ in the earlier argument
(where we replaced $\psi$ with $e^{\psi}$) may not suffice. In this more
general context, we will make an adjustment that depends on the \emph{growth
rates }of control functions $\rho_{-}$ and $\rho_{+}$. Without loss of
generality, we assume those functions are continuous. As another
simplification, we may make the following substitution.

\begin{lemma}
\label{lem:control functions} Let $X,Y$ be spaces which are coarsely
equivalent with control functions $\rho_{-},\rho_{+}$. Then there exists a
function $\rho:[0,\infty)\rightarrow\lbrack0,\infty)$ such that $X,Y$ are also
coarsely equivalent with respect to functions $\rho^{-1},\rho$.
\end{lemma}

\begin{proof}
Let $\rho(x):=\text{min}\bigl(\rho_{-}(x),\rho_{+}^{-1}(x)\bigr)$.
\end{proof}

\begin{definition}
Let $\phi:[0,\infty)\rightarrow\lbrack0,\infty)$ be a continuous, increasing
function such that $\phi(x)\leq\frac{x}{2}$ and define $\phi^{\ast}%
:[0,\infty)\rightarrow\lbrack0,\infty)$ by the following rule:%

\[
\phi^{\ast}(x)=%
\begin{cases}
1 & \text{if $x\leq1$}\\
1+\phi^{\ast}(\phi(x)) & \text{else.}%
\end{cases}
\]

\end{definition}

One may view the output $\phi^{\ast}\left(  x\right)  $ to be
\textquotedblleft one more than the number of times one needs to apply $\phi$
to $x$ to achieve a value less than $1$\textquotedblright. This is
well-defined and finite by the assumption that $\phi(x)\leq\frac{x}{2}$. The
construction of the star function is inspired by the \emph{iterated
logarithm}, which is traditionally denoted as $\log^{\ast}$. The $\log^{\ast}$
function (not to be confused with probability's \textquotedblleft Law of the
Iterated Logarithm\textquotedblright) has roots in complexity theory and
logic; for an example see \cite{succinct}. Clearly this function is not
continuous, but rather looks like a floor function which steps one greater at
intervals of length $\phi(k)$, $k\in\mathbb{N}$.

The following is a direct consequence of the definition of $\phi^{*}$.

\begin{lemma}
\label{lem:star function} Let $\phi: [0, \infty) \rightarrow[0, \infty)$ be a
continuous, increasing function such that $\phi(x) \leq\frac{x}{2}$. Then for
all $z \in[0, \infty)$, the following hold:

\begin{itemize}
\item $\phi^{*}(\phi(z)) = \phi^{*}(z) - 1$

\item $\phi^{*}(\phi^{-1}(z)) = \phi^{*}(z) + 1$.
\end{itemize}
\end{lemma}


We are now ready for the general (non-quasi-geodesic) version of Theorem
\ref{Theorem: EZ-structures}.

\begin{proof}
[Proof of Theorem~\ref{Theorem: EZ-structures}]Most of the necessary work has
been done in the preceding subsections, with the sole exception of an analog
of Lemma~\ref{Lemma: technical limit for QI case}. Thus, we need to, for an
arbitrary pair of control functions $\rho_{-},\rho_{+}$, develop a function
$\psi$ satisfying analogous limit laws to the above. Before doing this, we
invoke Lemma~\ref{lem:control functions} to use control functions of the form
$\rho^{-1},\rho$, assuming without loss of generality $\rho^{-1}<\rho$.
Furthermore, we can also assume that $\rho\left(  x\right)  >3x$. Let
$(\rho^{-1})^{\ast}$ be defined as above with respect to $\rho^{-1}$.

To tackle the limits, we first show the following.\medskip

\noindent\textbf{Claim.} For all $A,B\geq0$
\[
\lim_{x\rightarrow\infty}\frac{(\rho^{-1})^{\ast}(x+A)}{(\rho^{-1})^{\ast
}\left(  \rho_{-}(x)+B\right)  }=1\quad(\dagger)
\]

\noindent\textit{Proof of Claim. }By applying Lemma~\ref{lem:star function},
we obtain:
\[
\frac{(\rho^{-1})^{\ast}(x+A)}{(\rho^{-1})^{\ast}\left(  \rho_{-}(x)+B\right)
}\geq\frac{(\rho^{-1})^{\ast}(x+A)}{(\rho^{-1})^{\ast}\left(  \rho
_{-}(x)\right)  +B}=\frac{(\rho^{-1})^{\ast}(x+A)}{(\rho^{-1})^{\ast}%
(x)+B-1}\geq\frac{(\rho^{-1})^{\ast}(x)}{(\rho^{-1})^{\ast}(x)+B-1}%
\]
The last ratio clearly approaches $1$. \qed\medskip

The following inequality, for all $A,B\geq0$, proceeds mutatis mutandis:%

\[
\lim_{x\rightarrow\infty}\frac{(\rho^{-1})^{\ast}(x+A)}{(\rho^{-1})^{\ast
}\left(  \rho_{+}(x)+B\right)  }=1 \quad(\ddagger)
\]

With these limits in hand, the function $(\rho^{-1})^{\ast}$ almost serves our
purpose. Earlier arguments require $\psi$ (and thus $\psi^{-1}$) to be
continuous and bijective. To accomplish that take $\psi^{-1}$ to be the
function which linearly connects the points
\[
(0,0),(1,(\rho^{-1})^{\ast}(1)),(2,(\rho^{-1})^{\ast}(2)),\ldots
\]
Observe that $\psi^{-1}$ is continuous and bijective, and also satisfies
$|\psi^{-1}(x)-(\rho^{-1})^{\ast}(x)|\leq1$. This closeness of $\psi^{-1}$ to
$(\rho^{-1})^{\ast}$ guarantees that inequalities $(\dagger)$ and $(\ddagger)$
are satisfied by $\psi^{-1}$ as well. We note that the original condition on
$\psi$ given by $\psi(s+1)\geq3\psi(s)$ is satisfied by the assumption that
$\rho\left(  x\right)  \geq3x$, as in this case we know that $\rho^{-1}\left(
x\right)  \leq\frac{x}{3}$, so $(\rho^{-1})^{\ast}(x)\leq(\frac{x}{3}^{\ast
})=\log_{3}(x)$. This means that $\psi^{-1}$ is bounded above by $\log_{3}%
(x)$, ensuring that $\psi$ grows at least as fast as $3^{x}$. Requirement
\ref{Item c)}), from the previous section, can arranged by carefully rounding
the corners on the piecewise-linear graph.
\end{proof}

\subsection{Proof of Theorem \ref{Theorem D}}

Theorem \ref{Theorem D} simply identifies some interesting special cases where
the continuous maps in the hypothesis of Theorem \ref{Theorem: EZ-structures}
always exist. In the case of hyperbolic groups, this follows from
\cite{Gro87}. For $G=%
\mathbb{Z}
^{n}$, one can begin with the standard Euclidean $E\mathcal{Z}$-structure
$\left(  \overline{%
\mathbb{R}
^{n}},\mathbb{S}^{n-1}\right)  $ and note that every element of
$\operatorname*{Aut}\left(
\mathbb{Z}
^{n}\right)  $ can be realized by a linear map which has a natural extension
to the sphere at infinity. For abelian groups with torsion, one can simply let
the torsion elements act trivially on $%
\mathbb{R}
^{n}$ and repeat the previous construction. For CAT(0) groups with the
isolated flats property, the existence of these maps is an application of
\cite{HrKl05}. The appendix of that paper includes additional cases that can
be added to this collection.

\section{$\protect\underline{E\mathcal{Z}}$%
-structures\label{Section: EZ-bar structures}}

In Section \ref{sec: proof with torsion permitted} we saw that, when working
on groups that have torsion, it can be useful to use classifying spaces
($\underline{E}G_{AR}$-spaces and $\underline{E}G$-complexes) which allow
fixed points but place restrictions on the fixed-point sets. In that spirit,
we introduce analogous definitions for $\mathcal{Z}$- and $E\mathcal{Z}%
$-structures. Motivation is contained in work by Rosenthal, which allows us to
use these structures to make conclusions about the Novikov Conjecture.
Specific applications of that type will be addressed in Section
\ref{subsec:novikov}.

\begin{definition}
\label{Defn: E-bar-Z structure}An $\underline{E\mathcal{Z}}$\emph{-structure}
on $G$ is an $E$\textbf{$\mathcal{Z}$}-structure $\left(  \overline
{X},Z\right)  $ with the following additional properties:

\begin{enumerate}
\item \label{Item 1 of E-bar-Z definition}$X$ is an $\underline{E}%
G_{\text{AR}}$-space, and

\item \label{Item 2 of E-bar-Z definition}for each $H\in\mathcal{F}_{G}$,
$\overline{X}^{H}$ is a $\mathbf{\mathcal{Z}}$-compactification of $X^{H}$.
\end{enumerate}
\end{definition}

\noindent Notice that, when this definition is satisfied, then

\begin{enumerate}
\item[a)] $\overline{X}^{H}$, i.e., the subset of $\overline{X}$ fixed by $H$,
is equal to the closure of $X^{H}$ in $\overline{X}$,

\item[b)] $\overline{X}^{H}\cap Z=Z^{H}$, and

\item[c)] $\left(  \overline{X}^{H},Z^{H}\right)  $ is an $E$%
\textbf{$\mathcal{Z}$}-structure for $N_{G}\left(  H\right)  $ (the normalizer
of $H$ in $G$) and also for $N_{G}\left(  H\right)  /H$.
\end{enumerate}

\noindent If, in addition to the above, $X$ is an $\underline{E}G$-complex, we
call $\left(  \overline{X},Z\right)  $ a \emph{cellular }%
$\underline{E\mathcal{Z}}$ \emph{-structure}.

For a variation on the above definition, we can relax equivariance by
requiring only that $\left(  \overline{X},Z\right)  $ be a
\textbf{$\mathcal{Z}$}-structure, while keeping conditions
\ref{Item 1 of E-bar-Z definition} and \ref{Item 2 of E-bar-Z definition} in
place. We call this a $\underline{\mathcal{Z}}$\emph{-structure} on $G$. Under
this definition, observations a)-c) remain valid, except that $\left(
\overline{X}^{H},Z^{H}\right)  $ is only a \textbf{$\mathcal{Z}$}-structure
for $N_{G}\left(  H\right)  $ and $N_{G}\left(  H\right)  /H$ rather than an
$E$\textbf{$\mathcal{Z}$}-structure. A \emph{cellular }$\underline{\mathcal{Z}%
}$\emph{-structure} is defined in the obvious way.

For the purposes of this paper, we are especially interested in
$\underline{E\mathcal{Z}}$-structures. That will be the focus of the remainder
of this section.

\begin{example}
\label{Example: EZ-bar for hyperbolic groups}Every hyperbolic group $G$ admits
a cellular $\underline{E\mathcal{Z}}$-structure $\left(  \overline{X},\partial
G\right)  $, where $X$ is an appropriately chosen Rips complex for $G$ and
$\partial G$ is the Gromov boundary. This is precisely the content of
\cite{RoSc05}.
\end{example}

\begin{example}
\label{Example: EZ-bar for systolic groups}Every systolic group $G$ admits a
cellular $\underline{E\mathcal{Z}}$-structure $\left(  \overline{X},\partial
X\right)  $, where $X$ is the implied systolic simplicial complex acted upon
by $G$, and $\partial X$ is the systolic boundary as defined in \cite{OsPr09}.
This observation follows from the main theorem of that paper along with their
Theorem 14.1, Claim 14.2, and its proof.
\end{example}

\begin{example}
\label{Example: EZ-bar for CAT(0) groups}Every CAT(0) group $G$ admits an
$\underline{E\mathcal{Z}}$-structure $\left(  \overline{X},\partial_{\infty
}X\right)  $, where $X$ is the implied proper CAT(0) space acted upon by
geometrically by $G$, and $\partial_{\infty}X$ is its visual boundary. It is
well documented that $\left(  \overline{X},\partial_{\infty}X\right)  $ is an
$E$\textbf{$\mathcal{Z}$}-structure for $G$. Conditions
\ref{Item 1 of E-bar-Z definition}) and \ref{Item 2 of E-bar-Z definition})
hold because the fixed set of every finite subgroup is nonempty and convex in
$X$. See \cite[Cor.II.2.8]{BrHa99}.
\end{example}

Our primary applications of $\underline{E\mathcal{Z}}$-structures (see Section
\ref{subsec:novikov}) requires that they be cellular. For that reason, the
following refinement of Example \ref{Example: EZ-bar for CAT(0) groups} will
be useful.

\begin{theorem}
\label{Theorem: cellular EZ-bar for CAT(0) groups}Every CAT(0) group $G$
admits a cellular $\underline{E\mathcal{Z}}$-structure.
\end{theorem}

\begin{proof}
Let $\left(  \overline{X},\partial_{\infty}X\right)  $ be the $E\mathcal{Z}%
$-structure implied by the definition of CAT(0) group, then apply
\cite[Prop.A]{Ont05} to obtain a rigid $\underline{E}G$-simplicial complex $K$
and a $G$-equivariant map $f:K\rightarrow X$. If we give $K$ the path-length
metric, then $f$ is a quasi-isometry, so we may use the map $f$ and the
$E\mathcal{Z}$-boundary swapping theorem from \cite{GuMo19} to obtain an
$E\mathcal{Z}$-structure of the form $\left(  \overline{K},\partial_{\infty
}X\right)  $ and a continuous extension $\overline{f}:\overline{K}%
\rightarrow\overline{X}$ which is the identity on $\partial_{\infty}X$.

Since $f$ is $G$-equivariant, it maps $K^{H}$ into $X^{H}$ for every
$H\in\mathcal{F}_{G}$. Moreover, since $N_{G}\left(  H\right)  $ acts properly
and cocompactly on both $K^{H}$ and $X^{H}$, $\left.  f\right\vert _{K^{H}%
}:K^{H}\rightarrow X^{H}$ is a quasi-isometry. We know that $\overline{X}%
^{H}=X^{H}\sqcup\partial_{\infty}(X^{H})$, so the boundary swap between $K$
and $X$ restricts to a boundary swap between $K^{H}$ and $X^{H}$. In
particular, $\overline{K}^{H}=K^{H}\sqcup\partial_{\infty}(X^{H})$ is a
$\mathcal{Z}$-compactification. Therefore $\left(  \overline{K},\partial
_{\infty}X\right)  $ is a cellular $\underline{E\mathcal{Z}}$-structure.
\end{proof}

We now state and prove our main theorem about $\underline{E\mathcal{Z}}%
$-structures on groups of the form $G\rtimes_{\phi}%
\mathbb{Z}
$. Due to the delicate nature of the argument, we begin with the assumption of
a \emph{cellular} $\underline{E\mathcal{Z}}$-structure on $G$. That allows us
to choose $G$-equivariant maps and homotopies in a number of places---a
property that will be used in the proof. It also leads to the conclusion of a
\emph{cellular} $\underline{E\mathcal{Z}}$-structure on $G\rtimes_{\phi}%
\mathbb{Z}
$ ---a property that is required for our main applications. It seems likely
that a more generic version of this theorem is true, but the argument would be
even more delicate.

\begin{theorem}
\label{Theorem: Main theorem about EZ-bar-structures}Suppose $G$ admits a
cellular $\underline{E\mathcal{Z}}$-structure $\left(  \overline{X},Z\right)
$, $\phi\in\operatorname*{Aut}\left(  G\right)  $, and the corresponding
$\phi$-variant map(s) $f:X\rightarrow X$ extends to a continuous map
$\overline{f}:\overline{X}\rightarrow\overline{X}$ which is a homeomorphism on
$Z$. Then $G\rtimes_{\phi}${$%
\mathbb{Z}
$} admits a cellular $\underline{E\mathcal{Z}}$-structure with boundary equal
to $SZ$.
\end{theorem}

\begin{proof}
Propositions \ref{PhiVariantMap} and \ref{Theorem: constructing a G-semi-Z AR}
guarantee the existence of a cellular $\phi$-variant map $f$ and a
corresponding rigid $\left(  G\rtimes_{\phi}{%
\mathbb{Z}
}\right)  $-complex $Y=\operatorname*{Tel}_{f}\left(  X\right)  $. Proposition
\ref{prop:ebarg} assures that $Y$ is an $\underline{E}(G\rtimes_{\phi}%
\mathbb{Z}
)$-complex. The assumption of the existence of $\overline{f}:\overline
{X}\rightarrow\overline{X}$ implies a corresponding $E\mathcal{Z}$-structure
$\left(  \overline{Y},SZ\right)  $ for $G\rtimes_{\phi}%
\mathbb{Z}
$, as proved in Theorem \ref{Theorem: EZ-structures}. It remains only to
verify condition \ref{Item 2 of E-bar-Z definition}) of Definition
\ref{Defn: E-bar-Z structure}.

Let $H\in\mathcal{F}_{G\rtimes_{\phi}%
\mathbb{Z}
}$ and recall from our discussion in Section
\ref{Subsection: fixed sets and phi-variant maps} that $H\leq G$, and $Y^{H}$
is the sub-mapping telescope defined by the following subspaces and
restriction maps.
\[
\cdots\overset{\left.  f\right\vert }{\longrightarrow}X^{\phi^{-2}%
(H)}\overset{\left.  f\right\vert }{\longrightarrow}X^{\phi^{-1}%
(H)}\overset{\left.  f\right\vert }{\longrightarrow}X^{H}\overset{\left.
f\right\vert }{\longrightarrow}X^{\phi(H)}\overset{\left.  f\right\vert
}{\longrightarrow}X^{\phi^{2}(H)}\overset{\left.  f\right\vert
}{\longrightarrow}\cdots
\]

Clearly $\left(  X\times%
\mathbb{R}
\right)  ^{H}=X^{H}\times%
\mathbb{R}
$, and since $v:Y\rightarrow X\times%
\mathbb{R}
$ and $u:X\times%
\mathbb{R}
\rightarrow Y$ are $G$-equivariant (see Remark
\ref{Remark: Equivariance of maps and homotopies in CW cases}), $v\left(
Y^{H}\right)  \subseteq X^{H}\times%
\mathbb{R}
$ and $u(X\times%
\mathbb{R}
)\subseteq Y^{H}$. By our hypothesis, $\overline{X}^{H}$ is a $\mathcal{Z}%
$-compactification $X^{H}\sqcup Z^{H}$ of $X^{H}$ where $Z^{H}\subseteq Z$. By
the definition of the topology on $\overline{X\times%
\mathbb{R}
}$, the closure of $X^{H}\times%
\mathbb{R}
$ in $\overline{X\times%
\mathbb{R}
}$ is precisely $\left(  X^{H}\times%
\mathbb{R}
\right)  \sqcup S(Z^{H})$ topologized according to the same rules (i.e., the
restriction of the same slope function used to topologize $\overline{X\times%
\mathbb{R}
}$) where $S(Z^{H})$ is also $SZ^{H}$. From here, the same argument used in
proving \ref{ProductWithJoinIsANR} (See, in particular, Lemma
\ref{ReparameterizedRays} and the remark that follows it.) shows that
$\overline{X\times%
\mathbb{R}
}^{H}=\left(  X^{H}\times%
\mathbb{R}
\right)  \sqcup S(Z^{H})$ is a $\mathcal{Z}$-compactification of $(X\times%
\mathbb{R}
)^{H}$. Recalling that the $\mathcal{Z}$-compactification $\overline
{Y}=Y\sqcup SZ$ was obtained as the pull-back of the compactification
$\overline{X\times%
\mathbb{R}
}=\left(  X\times%
\mathbb{R}
\right)  \sqcup SZ$ via the map $v:Y\rightarrow X\times%
\mathbb{R}
$, it is clear that $\overline{Y}^{H}$ is the pull-back compactification of
$\overline{X\times%
\mathbb{R}
}^{H}=\left(  X^{H}\times%
\mathbb{R}
\right)  \sqcup S(Z^{H})$ via $\left.  v\right\vert _{Y^{H}}:Y^{H}\rightarrow
X^{H}\times%
\mathbb{R}
$. Now the same boundary swapping proof used in Proposition
\ref{Prop: Y-hat is a Z-compactification}\ applies to show that $\overline
{Y}^{H}=Y^{H}\sqcup S(Z^{H})$ is a $\mathcal{Z}$-compactification. Here one
should note that, by $G$-equivariance (see Remark
\ref{Remark: Equivariance of maps and homotopies in CW cases} again), the
homotopy $\overline{H}$ used in the earlier argument restricts to an
appropriate homotopy between self-maps of $\overline{Y}^{H}$.
\end{proof}

\begin{corollary}
Let $G$ be a hyperbolic group and $\phi\in\operatorname*{Aut}\left(  G\right)
$. Then $G\rtimes_{\phi}%
\mathbb{Z}
$ admits a cellular $\underline{E\mathcal{Z}}$-structure with boundary the
suspension of the Gromov boundary of $G$.
\end{corollary}

\begin{proof}
Begin with the cellular $\underline{E\mathcal{Z}}$-structure $\left(
\overline{X},\partial G\right)  $ on $G$ discussed in Example
\ref{Example: EZ-bar for hyperbolic groups}. A $\phi$-variant map
$f:X\rightarrow X$ exists by Theorem \ref{PhiVariantMap}. Since $f$ is a
quasi-isometry, the well-known theory of hyperbolic spaces ensures a
continuous extension $\overline{f}:\overline{X}\rightarrow\overline{X}$ which
takes $\partial G$ homeomorphically onto $\partial G$.
\end{proof}

\begin{corollary}
Let $G$ be a CAT(0) or systolic group with corresponding
$\underline{E\mathcal{Z}}$-structure $(\overline{X},\partial_{\infty}X)$,
$\phi\in\operatorname*{Aut}\left(  G\right)  $, and $f:X\rightarrow X$ a
corresponding $\phi$-variant map. If there exists a continuous extension
$\overline{f}:\overline{X}\rightarrow\overline{X}$ which is a homeomorphism on
$\partial_{\infty}X$, then $G\rtimes_{\phi}%
\mathbb{Z}
$ admits a cellular $\underline{E\mathcal{Z}}$-structure with boundary the
suspension of $\partial_{\infty}X$.
\end{corollary}

\begin{proof}
For systolic $G$, the $\underline{E\mathcal{Z}}$-structure $(\overline
{X},\partial_{\infty}X)$ is automatically cellular, so the conclusion is
immediate. For CAT(0) $G$, we must first swap the $\underline{E\mathcal{Z}}%
$-structure $(\overline{X},\partial_{\infty}X)$ for the cellular version
$\left(  \overline{K},\partial_{\infty}X\right)  $ promised in Theorem
\ref{Theorem: cellular EZ-bar for CAT(0) groups}.
\end{proof}

\section{Applications of the main theorems}

\label{sec:applications}

In this section we look at some concrete applications of the main results of
this paper. We begin with a look at $E\mathcal{Z}$-structures and
$\underline{E\mathcal{Z}}$-structures and their relationship to the Novikov
Conjecture. We point out situations where our methods can add to the
collection of groups for which the Novikov Conjecture is known to be true, and
provide new proofs that other groups belong to that collection.

Next we examine polycyclic groups (a class which contains all finitely
generated nilpotent groups) from the perspective of group boundaries. After
that, we show that fundamental groups of \emph{all} closed 3-manifolds admit
$\mathcal{Z}$-structures. For these latter two applications, we also discuss
$E\mathcal{Z}$-structures and $\underline{E\mathcal{Z}}$-structures.

\subsection{Applications of $\protect\underline{E\mathcal{Z}}$-structures to
the Novikov Conjecture}

\label{subsec:novikov}

Notice that for torsion-free groups there is no difference between a
$E\mathcal{Z}$-structure and an $\underline{E\mathcal{Z}}$-structure. Work by
Carllson and Pedersen \cite{CaPe95} and Farrell and Lafont \cite{FaLa05}
showed that the existence of an $E\mathcal{Z}$-structure on a torsion-free
group $\Gamma$ implies the Novikov Conjecture for $\Gamma$. In fact, Farrell
and Lafont's motivation for defining an $E\mathcal{Z}$-structure was precisely
that application. Rosenthal \cite{Ros04}, \cite{Ros06}, \cite{Ros12} expanded
upon that work to provide a similar approach to the Novikov Conjecture for
groups with torsion. His conditions motivated our definition of an
$\underline{E\mathcal{Z}}$-structure. When the $\underline{E\mathcal{Z}}%
$-structure is cellular, all of Rosenthal's conditions are satisfied, so we have:

\begin{theorem}
[after Rosenthal]\label{thm:rosenthal} If a group $G$ admits a cellular
$\underline{E\mathcal{Z}}$-structure, then the Baum-Connes map, $KK_{i}%
^{G}\left(  C_{0}\left(  \underline{E}G\right)  ;%
\mathbb{C}
\right)  \rightarrow K_{i}\left(  C_{r}^{\ast}G\right)  $, is split injective.
In particular, the Novikov Conjecture holds for $G$.
\end{theorem}

\noindent For example, this theorem, combined with work discussed in the
previous section, implies the Novikov Conjecture for all hyperbolic, CAT(0),
and systolic groups, including those with torsion. (In many cases, other
proofs are known.)

For the purposes of this paper, we are interested in groups of the form
$G\rtimes_{\phi}%
\mathbb{Z}
$. Work presented above yields the following.

\begin{theorem}
\label{Theorem: EZ structures for hyperbolic-by-Z}Let $G$ be a hyperbolic
group and $\phi\in\operatorname*{Aut}\left(  G\right)  $. Then the Novikov
Conjecture holds for $G\rtimes_{\phi}%
\mathbb{Z}
$.
\end{theorem}

\begin{theorem}
\label{Theorem: EZ structures for CAT(0) or systolic-by-Z}Let $G$ be a CAT(0)
or systolic group with corresponding $E\mathcal{Z}$-structure $(\overline
{X},\partial_{\infty}X)$, $\phi\in\operatorname*{Aut}\left(  G\right)  $, and
$f:X\rightarrow X$ a $\phi$-variant map. If $f$ extends continuously to a map
$\overline{f}:\overline{X}\rightarrow\overline{X}$ which is a homeomorphism on
$\partial_{\infty}X$, then the Novikov Conjecture holds for $G\rtimes_{\phi}%
\mathbb{Z}
$.
\end{theorem}

\begin{remark}
Our assertion about the Novikov Conjecture in Theorem
\ref{Theorem: EZ structures for hyperbolic-by-Z} is not new. An existing proof
goes as follows: groups with finite asymptotic dimension satisfy the Novikov
Conjecture; hyperbolic groups have finite asymptotic dimension; and extensions
of groups with finite asymptotic dimension by groups with finite asymptotic
dimension have finite asymptotic dimension. See \cite{BeDr01}.

It is an open question whether CAT(0) or systolic groups have finite
asymptotic dimension. As such, to the best of our knowledge, the assertions
about the Novikov Conjecture in Theorem
\ref{Theorem: EZ structures for CAT(0) or systolic-by-Z} are new. Notice that
the hypothesis about the existence of $\overline{f}:\overline{X}%
\rightarrow\overline{X}$ is not vacuous. See \cite{BoRu96} for a relevant example.
\end{remark}



\subsection{Polycyclic groups, nilpotent groups, and groups of polynomial
growth}

\label{subsec:nilpotent}

A group $G$ is \emph{polycyclic} if it admits a subnormal series
\[
G=G_{k}\vartriangleright G_{n-1}\vartriangleright\cdots\vartriangleright
G_{0}=\left\{  1\right\}
\]
for which each quotient group $G_{i+1}/G_{i}$ is cyclic. If it admits such a
series for which each quotient is infinite cyclic, then $G$ is called
\emph{strongly polycyclic }(or sometimes \emph{poly-}$%
\mathbb{Z}
$). The \emph{Hirsch length} of polycyclic $G$ is the number of infinite
cyclic factors in its subnormal series. It is a standard fact that Hirsch
length is an invariant of $G$.

The following comes from inducting on the Hirsch length, applying Theorem
\ref{Theorem A} at each step.

\begin{theorem}
\label{polycyclic} Every strongly polycyclic group $G$ admits a $\mathcal{Z}%
$-structure. If the Hirsch length of $G$ is $n$, the $\mathcal{Z}$-structure
$(\overline{X},Z)$ can be chosen so that $Z=\mathbb{S}^{n-1}$.
\end{theorem}

A group $G$ is \emph{nilpotent} if there exists a finite sequence of normal
subgroups $G=G_{k}\vartriangleright G_{n-1}\vartriangleright\cdots
\vartriangleright G_{0}=\left\{  1\right\}  $ such that $[G,G_{i+1}]$ is
contained in $G_{i}$ (where brackets indicate the commutator). Observe that a
finitely generated nilpotent group is polycyclic.

\begin{theorem}
\label{nilpotent} Every finitely generated nilpotent group admits a
$\mathcal{Z}$-structure with spherical boundary.
\end{theorem}

\begin{proof}
Let $\Gamma$ be a finitely generated nilpotent group. The set of all torsion
elements forms a finite, characteristic subgroup $T(\Gamma)\vartriangleleft
\Gamma$, which gives us a short exact sequence
\[
1\rightarrow T(\Gamma)\rightarrow\Gamma\rightarrow\Gamma/T(\Gamma
)\rightarrow1
\]
where $\Gamma/T(\Gamma)$ is a torsion-free, nilpotent group \cite[Cor.
1.10]{Seg83}. Recall then that a \emph{torsion-free} nilpotent group is
strongly polycyclic group (see \cite[5.2.20]{Rob96}). Therefore $\Gamma
/T(\Gamma)$ admits a $\mathcal{Z}$ structure $(\overline{X},\mathbb{S}^{k})$
by Theorem~\ref{polycyclic}. The proper cocompact action of $\Gamma/T(\Gamma)$
on $X$ can be extended to a cocompact $\Gamma$-action using the homomorphism
$\Gamma\rightarrow\Gamma/T(\Gamma)$. Since the kernel is finite, this action
is also proper, so $(\overline{X},\mathbb{S}^{k})$ is a $\mathcal{Z}%
$-structure for $\Gamma$.
\end{proof}

Already, the above results expand greatly on the class of groups known to
admit a $\mathcal{Z}$-structures. That is because non-elementary nilpotent
groups are never hyperbolic (they have polynomial growth, \cite{Wol68}) and
seldom CAT(0) (nilpotent groups which are CAT(0) are virtually abelian). See
\cite[Theorem~7.8, p. 249]{BrHa99}. By invoking some powerful theorems, we can
obtain more.

\begin{theorem}
\label{Theorem:polynomial growth theorem}Every group of polynomial growth
admits a $\mathcal{Z}$-structure with spherical boundary.
\end{theorem}

\begin{proof}
Suppose $G$ has polynomial growth. Then, by \cite{Gro81}, $G$ contains a
finite index nilpotent subgroup $H$. Since $G$ is finitely generated, $H$ is
finitely generated, so by Theorem \ref{nilpotent}, $H$ admits a $\mathcal{Z}%
$-structure of the form $\left(  \overline{X},\mathbb{S}^{k}\right)  $ for
some $k$.

Since $G$ contains a finite index nilpotent group, it contains a finite index
normal nilpotent group, assuring us that $G$ is elementary amenable. Since
finitely generated nilpotent groups are polycyclic, $G$ is also
polycyclic-by-finite, so it satisfies condition (x) of \cite[Th.1.1]{KMN09}.
As such, $G$ satisfies condition (i) of that theorem, meaning that there
exists a cocompact $\underline{E}G$-complex $Y$. From here, we may apply the
Generalized Boundary Swapping Theorem \cite[Cor.7.2]{GuMo19} to obtain a
$\mathcal{Z}$-structure for $G$ of the form $\left(  \overline{Y}%
,\mathbb{S}^{k}\right)  $.
\end{proof}

\begin{corollary}
\label{Corollary: consequences of Z-structures}For every group $G$ that is
strongly polycyclic or of polynomial growth, there exists an integer $k\geq-1$
such that $H^{\ast}\left(  G;%
\mathbb{Z}
G\right)  \cong H^{\ast-1}\left(  \mathbb{S}^{k};%
\mathbb{Z}
\right)  $. In addition, $G$ has the same homotopy at infinity as $%
\mathbb{Z}
^{k+1}$, in particular, $G$ is semistable and $\operatorname*{pro}$-$\pi
_{i}\left(  G\right)  $ is stably isomorphic to $\pi_{i}\left(  \mathbb{S}%
^{k}\right)  $ for all $i$.
\end{corollary}

\begin{proof}
These observations follow easily from standard applications of $\mathcal{Z}%
$-structures. See \cite{Bes96} for torsion-free cases and \cite{Dra06},
\cite{GuMo19} and \cite{GuMo21} for extensions to groups with torsion. Earlier
proofs of some of the group cohomology assertions can be found at
\cite[p.213]{Bro82}. For analogous conclusions regarding semistability and
$\operatorname*{pro}$-$\pi_{i}\left(  G\right)  $ for strongly polycyclic
groups one can inductively apply \cite[Prop.17.3.1]{Geo08}.
\end{proof}

\begin{remark}
Unlike the class of groups with polynomial growth, a group $G$ quasi-isometric
to a strongly polycyclic group is not known to be strongly polycyclic or to
admit a cocompact $\underline{E}G$-complex. As such, we cannot mimic the above
strategies to endow $G$ with an $\mathcal{Z}$-structure. We can, however,
place a \emph{coarse} $\mathcal{Z}$\emph{-structure} on $G$ which allows for
many of the same applications as a genuine $\mathcal{Z}$-structure. See
\cite{GuMo21} for details.
\end{remark}

By combining Theorem \ref{Theorem: main theorem for groups with torsion} with
\ref{Theorem:polynomial growth theorem} and its proof, we get a little more.

\begin{theorem}
Every group of the form $G\rtimes_{\phi}%
\mathbb{Z}
$, where $G$ is of polynomial growth, admits a $\mathcal{Z}$-structure with
spherical boundary.
\end{theorem}

Existence of $E\mathcal{Z}$-structures and $\underline{E\mathcal{Z}}%
$-structures for the groups discussed in this section is an interesting open
question. For a strongly polycyclic group, each step in the inductive proof of
Theorem \ref{polycyclic} involves some $\phi_{i}\in\operatorname*{Aut}\left(
G_{i}\right)  $ whose realization as $\phi_{i}$-variant map $f:X_{i}%
\rightarrow X_{i}$ (implicit in that proof) would need to be extended over
$\overline{X}_{i}$ in order to move from an $E\mathcal{Z}$-structure on
$G_{i}$ to an $E\mathcal{Z}$-structure on $G_{i+1}$. By using work found in
\cite{Wei21}, one sees that this is always possible when the Hirsch length is
$\leq3$. For higher Hirsch length, the corresponding sequence of progressively
more complicated groups provides an interesting test case for Question
\ref{Question: phi-variant extensions}.

As for a finitely generated nilpotent group $\Gamma$, the $\mathcal{Z}%
$-structure described above depends entirely on the $\mathcal{Z}$-structure
$\left(  \overline{X},\mathbb{S}^{n}\right)  $ on $\Gamma/T(\Gamma)$, where we
simply allow $T\left(  \Gamma\right)  $ to act trivially on $X$. This means
that whenever $\Gamma/T(\Gamma)$ admits an $E\mathcal{Z}$-structure, so does
$\Gamma$. Moreover, if $H\leq\Gamma$ is finite, then $H\leq T\left(
\Gamma\right)  $, so $X^{H}=X$ and $\overline{X}^{H}=\overline{X}$. Therefore,
we have an $\underline{E\mathcal{Z}}$-structure.

\subsection{$\mathcal{Z}$-structures on 3-manifold groups}

\label{subsec:3mfd}

We now prove Theorem \ref{Theorem C} from the introduction---that the
fundamental groups of every closed 3-manifold admits a $\mathcal{Z}%
$-structure. Our proof brings together a tremendous amount established
knowledge from 3-manifold topology, beginning with the classical theory and
extending through Perelman's proof of the Geometrization Conjecture. It also
uses tools from geometric group theory, such as theorems by Dahmani and Tirel
regarding boundaries of free products, as well as boundary swapping techniques
introduced by Bestvina and expanded upon in \cite{GuMo19}. As for new
ingredients, Theorem \ref{Theorem A} plays a decisive role in handling
fundamental groups of manifolds with $Nil$ and $Sol$ geometries. For a general
discussion of 3-manifold groups, see \cite{AFW15}.

A closed 3-manifold is \emph{prime} if it admits no nontrivial connected sum
decomposition; it is \emph{irreducible} if every tamely embedded 2-sphere
bounds a 3-ball; it is $\mathbb{P}^{2}$\emph{-irreducible }if it is
irreducible and contains no 2-sided projective planes. Note that every
orientable irreducible 3-manifold is $\mathbb{P}^{2}$-irreducible.

\begin{theorem}
\label{3ManifoldTheorem} Every closed 3-manifold group admits a $\mathcal{Z}$-structure.
\end{theorem}

\begin{proof}
Let $M^{3}$ be a closed connected 3-manifold and $G=\pi_{1}(M^{3})$. By
Kneser's existence theorem for prime decompositions \cite{Kne29} (see also
\cite[Th.3.15]{Hem76}) we may express $M^{3}$ and a finite connected sum
$P_{1}^{3}\#P_{2}^{3}\#\cdots\#P_{n}^{3}$ of closed prime 3-manifolds and $G$
as a free product $G_{1}\ast G_{2}\ast\cdots\ast G_{n}$, where $G_{i}=\pi
_{1}(P_{i}^{3})$. By \cite{Dah03} or \cite{Tir11}, it suffices to show that
each $G_{i}$ admits a $\mathcal{Z}$-structure. That will be accomplished by
examining the possible structures for the individual $P_{i}^{3}$. \medskip

\noindent\textbf{Case 1.} $P_{i}^{3}$ \emph{is not irreducible. }\medskip

Then $P_{i}^{3}$ is either the trivial 2-sphere bundle over $\mathbb{S}^{1}$
or the twisted (nonorientable) 2-sphere bundle $\mathbb{S}^{2}%
\widetilde{\times}\mathbb{S}^{1}$. In either case, $G_{i}\cong%
\mathbb{Z}
$ which admits the $\mathcal{Z}$-structure $\left(  \overline{%
\mathbb{R}
},\left\{  \pm\infty\right\}  \right)  $.\medskip

\noindent\textbf{Case 2.} $P_{i}^{3}$ \emph{is }$\mathbb{P}^{2}$%
\emph{-irreducible.}\medskip

By the Geometrization Conjecture, this class of 3-manifolds can be divided
into three disjoint subclasses: geometric manifolds; non-geometric mixed
manifolds; and non-geometric graph manifolds. We will discuss each of these in
some detail.\smallskip

\noindent\textsc{subcase 2a).} $P_{i}^{3}$ \emph{is a geometric manifold.}%
\smallskip

Of the eight 3-dimensional geometries allowed by the Geometrization Theorem,
seven remain possible (the geometry $\mathbb{S}^{2}\times%
\mathbb{R}
$ having been taken care of in Case 1).

If $P_{i}^{3}$ admits the geometry of $\mathbb{S}^{3}$ then $G_{i}$ is finite,
and we can obtain a $\mathcal{Z}$-structure, with empty boundary, by letting
$G_{i}$ act on a one-point space.

The geometries $\mathbb{E}^{3},\ \mathbb{H}^{3},\ \mathbb{H}^{2}\times%
\mathbb{R}
$ are all CAT(0), so if $P_{i}^{3}$ admits one of these geometries, then
$G_{i}$ admits a $\mathcal{Z}$-structure with the corresponding visual
2-sphere at infinity serving as its $\mathcal{Z}$-boundary.

The geometry $\widetilde{SL_{2}(%
\mathbb{R}
)}$ is not CAT(0), but by \cite{Ger92} (see also \cite{Rie01} for a short
elegant proof) it is quasi-isometric to $\mathbb{H}^{2}\times%
\mathbb{R}
$. So by the boundary swapping trick described in \cite{GuMo19},
$\widetilde{SL_{2}(%
\mathbb{R}
)}$ admits a controlled $\mathcal{Z}$-compactification. Since $G_{i}$ acts
geometrically on $\widetilde{SL_{2}(%
\mathbb{R}
)}$, this is a $\mathcal{Z}$-structure for $G_{i}$.

The remaining geometries $Nil$ and $Sol$ are Lie groups homeomorphic to $%
\mathbb{R}
^{3}$ which contain cocompact lattices of the form $%
\mathbb{Z}
^{2}\rtimes_{\phi_{1}}%
\mathbb{Z}
$ and $%
\mathbb{Z}
^{2}\rtimes_{\phi_{2}}%
\mathbb{Z}
$. (For example, let $\phi_{1}$ and $\phi_{2}$ be induced by matrices $%
\begin{bmatrix}
1 & 1\\
0 & 1
\end{bmatrix}
$ and $%
\begin{bmatrix}
2 & 1\\
1 & 1
\end{bmatrix}
$, respectively.) For simplicity, let us begin with $Nil$. Since $%
\mathbb{R}
^{2}$ with the Euclidean metric and the visual $\mathbb{S}^{1}$ at infinity
provide a $\mathcal{Z}$-structure for $%
\mathbb{Z}
^{2}$, Theorem \ref{Theorem A} allows us to build a $\mathcal{Z}$-structures
$\left(  \overline{Y},\mathbb{S}^{2}\right)  $ for the lattice $%
\mathbb{Z}
^{2}\rtimes_{\phi_{1}}%
\mathbb{Z}
$. (With a little effort, we can arrange that $Y$ is a geodesic space
homeomorphic to $%
\mathbb{R}
^{3}$, but that is not essential.) By the \v{S}varc-Milnor Lemma, $Nil$ is
quasi-isometric to $Y$, so we can again use the boundary swapping trick to
obtain a controlled $\mathcal{Z}$-compactification $\left(  \overline
{Nil},\mathbb{S}^{2}\right)  $ which serves as a $\mathcal{Z}$-structure for
all cocompact lattices in $Nil$. A similar argument yields a controlled
$\mathcal{Z}$-compactification $\left(  \overline{Sol},\mathbb{S}^{2}\right)
$ which handles the case of $Sol$.\smallskip\ 

\noindent\textsc{subcase 2b). }$P_{i}^{3}$ \emph{is a non-geometric mixed
manifold.}\smallskip

Recall that a \emph{mixed manifold} is one whose prime JSJ-decomposition
includes at least one hyperbolic block. Leeb proved that Haken mixed manifolds
admit nonpositively curved Riemannian metrics \cite{Lee95}. Since $P_{i}^{3}$
is prime and non-geometric, it contains at least one JSJ-torus, and is
therefore Haken. So Leeb's theorem tells us that $G_{i}$ is a CAT(0) group
with a 2-sphere boundary.\smallskip

\noindent\textsc{subcase 2c).} $P_{i}^{3}$ \emph{is a non-geometric graph
manifold.}\smallskip

A \emph{graph manifold} is one whose prime JSJ decomposition contains no
hyperbolic blocks. Reasoning as above, $P_{i}^{3}$ is Haken, so by
Kapovich-Leeb \cite{KaLe98}, there exists a nonpositively curved 3-manifold
$N^{3}$ and a bi-Lipschitz homeomorphism $h:\widetilde{P_{i}^{3}}%
\rightarrow\widetilde{N^{3}}$. Another boundary swap places the visual
boundary of $\widetilde{N^{3}}$ (necessarily a topological 2-sphere) onto
$\widetilde{P_{i}^{3}}$ thereby giving $G_{i}$ a $\mathcal{Z}$%
-structure.\medskip

\noindent\textbf{Case 3.} $P_{i}^{3}$ \emph{is irreducible but not
}$\mathbb{P}^{2}$\emph{-irreducible.}\medskip

Let $p:Q_{i}^{3}\rightarrow P_{i}^{3}$ be the orientable double covering. Then
$%
\mathbb{Z}
_{2}$ acts on $Q_{i}^{3}$ by covering transformations, and by \cite{MSY82} or
\cite{Dun85}, $Q_{i}^{3}$ contains a $%
\mathbb{Z}
_{2}$-equivariant collection of pairwise disjoint essential 2-spheres
$\left\{  \Sigma_{k}\right\}  _{k=1}^{n}$ which generates $\pi_{2}( Q_{i}^{3})
$ as a $%
\mathbb{Z}
\pi_{1}$-module.\smallskip

\noindent\textsc{Claim.} No $\Sigma_{k}$ separates $Q_{i}^{3}$. \smallskip

The claim breaks into two cases, depending on whether $p( \Sigma_{k}) $ is a
2-sphere or a projective plane.

First assume that $p( \Sigma_{k}) =\Pi$ is a projective plane. Since
$P_{i}^{3}$ is prime and nonorientable, $\Pi$ is 2-sided (see \cite[Lemma~2]%
{Heil}). By an Euler characteristic argument, a single projective plane cannot
be the boundary of a compact 3-manifold, so $\Pi$ cannot separate $P_{i}^{3}$.
Let $\gamma$ be a path in $P_{i}^{3}-\Pi$ which begins at a point on one side
of a product neighborhood of $\Pi$ and ends at a point on the other side. A
lift of $\gamma$ will lie in $Q_{i}^{3}-\Sigma_{k}$ and connect points on
opposite sides of a collar neighborhood of $\Sigma_{k}$, therefore $\Sigma
_{k}$ does not separate $Q_{i}^{3}$.

Next suppose that $p(\Sigma_{k})$ is a 2-sphere $\Sigma$, and assume
$\Sigma_{k}$ separates $Q_{i}^{3}$. Write $Q_{i}^{3}=A\cup B$ where $A$ and
$B$ are connected codimension 0 submanifolds of $Q_{i}^{3}$ intersecting in a
common boundary $\Sigma_{k}$. Without loss of generality, assume $B$ contains
the $%
\mathbb{Z}
_{2}$ translate of $\Sigma_{k}$. Notice that $\left.  p\right\vert
_{A}:A\rightarrow p(A)$ is a covering map. (Just check the definition
locally.) Since points of $\Sigma$ have only one preimage in $A$ then $\left.
p\right\vert _{A}:A\rightarrow p(A)$ is a homeomorphism. It follows that
$\Sigma$ separates $P_{i}^{3}$ so, by irreducibility, $\Sigma$ bounds a 3-ball
in $P_{i}^{3}$. That 3-ball can be lifted to a 3-ball in $P_{i}^{3}$ bounded
by $\Sigma_{k}$, contradicting the inessentiality of $\Sigma_{k}$. The claim
follows.\smallskip

Since no $\Sigma_{k}$ separates it, $Q_{i}^{3}$ is a prime manifold. By its
definition, $Q_{i}^{3}$ is orientable and contains at least one essential
2-sphere. It follows that $Q_{i}^{3}\approx\mathbb{S}^{1}\times\mathbb{S}^{2}%
$. The only nonorientable manifold double covered by $\mathbb{S}^{1}%
\times\mathbb{S}^{2}$ is $\mathbb{P}^{2}\times\mathbb{S}^{1}$. It follows that
$G_{i}\cong%
\mathbb{Z}
\times%
\mathbb{Z}
_{2}$, a group which admits a geometric action on $%
\mathbb{R}
$ and a corresponding $\mathcal{Z}$-structure $\left(  \overline{%
\mathbb{R}
},\left\{  \pm\infty\right\}  \right)  $
\end{proof}

\begin{remark}
By the same strategy applied above (but using the equivariant versions of
\cite{Dah03} or \cite{Tir11}), one can prove that a given closed 3-manifold
group $G=\pi_{1}(M^{3})$ admits an $E\mathcal{Z}$-structure by showing that
the fundamental group $G_{i}=\pi_{1}(P_{i}^{3})$ of each of its prime factors
does. Revisiting the above proof, we see that in many cases, the work has
already been done. Cases 1 and 3 involve the groups $%
\mathbb{Z}
$ and $%
\mathbb{Z}
\times%
\mathbb{Z}
_{2}$, both of which admit $E\mathcal{Z}$-structures. Of the many groups
$G_{i}$ that arise in Subcase 2a), the $\mathcal{Z}$-structures associated to
fundamental groups of manifolds with geometries of $\mathbb{S}^{3}$,
$\mathbb{E}^{3},\ \mathbb{H}^{3},\ $and $\mathbb{H}^{2}\times%
\mathbb{R}
$ are immediately $E\mathcal{Z}$-structures. As for the remaining geometries
$\widetilde{SL_{2}(%
\mathbb{R}
)}$, $Nil$, and $Sol$, many of the $Nil$ and $Sol$ groups (those isomorphic to
a semidirect product of the form $%
\mathbb{Z}
^{2}\rtimes_{\phi}%
\mathbb{Z}
$) have been shown, in this paper, to admit $E\mathcal{Z}$-structures.
Unfortunately, there may be cocompact lattices in $Nil$ and $Sol$ that do not
fall into this category. We are not yet sure if they admit $E\mathcal{Z}%
$-structures. For similar reasons, the existence of $E\mathcal{Z}$-structures
for groups arising in Subcase 2c)---fundamental groups of non-geometric graph
manifolds---is still an open question. On the other hand, all groups arising
in Subcase 2b)---fundamental groups of non-geometric mixed manifolds---admit
$E\mathcal{Z}$-structures by virtue of being CAT(0).

Without formulating a detailed statement, we can say that \emph{many }(even
most) closed 3-manifold groups admit $E\mathcal{Z}$-structures. For the reader
interested in a specific 3-manifold or collection of 3-manifolds, the above
discussion provides a roadmap for checking whether an $E\mathcal{Z}$-structure
is known to exist.
\end{remark}

\section{Further Questions}

The work presented in this paper raises a number of questions. We close by
highlighting two of them.

\begin{question}
\label{Question: phi-variant extensions}Given a $\mathcal{Z}$-structure
$\left(  \overline{X},Z\right)  $ on a group $G$ and $\phi\in
\operatorname*{Aut}\left(  G\right)  $, when does there exist a $\phi$-variant
map $f:X\rightarrow X$ which can be continuously extended to a map
$\overline{f}:\overline{X}\rightarrow\overline{X}$. Given a $\phi$ for which
the answer is no, is there a different $\mathcal{Z}$-structure for which the
answer is yes?
\end{question}

When $\left(  \overline{X},Z\right)  $ is a canonical $\mathcal{Z}$-structure
on a hyperbolic group, the answer is \textquotedblleft
always\textquotedblright. The same holds for finitely generated abelian
groups. Work by Hruska, Kleiner and Hindawi \cite{HrKl05} identifies other
interesting classes of CAT(0) groups (e.g., those with the isolated flats
property) for which the answer is \textquotedblleft always\textquotedblright.
But for general CAT(0) groups, such as $F_{2}\times%
\mathbb{Z}
^{n}$, there are difficulties. See \cite{BoRu96}. What more can be said about
about these groups? What about strongly polycyclic and finitely generated
nilpotent groups? Systolic groups? Baumslag-Solitar groups?\medskip

In a different direction, we ask for generalizations of our main theorems.

\begin{question}
\label{Question: more general group extensions}Given a short exact sequence of
groups
\[
1\rightarrow N\rightarrow G\rightarrow Q\rightarrow1
\]
where $N$ and $Q$ admit $\mathcal{Z}$- or $E\mathcal{Z}$- or
$\underline{E\mathcal{Z}}$-structures, what can be said about $G$? In
particular, when does $G$ admit an analogous structure? Does it help to assume
that $G$ is a semidirect product?
\end{question}

Special cases of this question are of interest. For example, what if $N$ and
$Q$ are hyperbolic? CAT(0)? free? As a starting point, one might look at
\cite{Gui14} where it is shown that, whenever $N$ and $Q$ are nontrivial and
of type F, then $G$ admits a \emph{weak }$\mathcal{Z}$\emph{-structure},
meaning that all conditions for a $\mathcal{Z}$-structure are satisfied,
except for the nullity condition.

\appendix

\section{Example: A $\mathcal{Z}$-structure for the Discrete Heisenberg Group}


In this appendix we take a quick look at the main construction, within the
narrow context of a well-known group realizable as an infinite cyclic
extension of {$%
\mathbb{Z}
$}$^{2}$. In particular, we analyze the issues involved in placing a
$\mathcal{Z}$-structure on the discrete Heisenberg group $H_{3}(${$%
\mathbb{Z}
$}$)$. The reader who might otherwise be put off by the abstractions found in
the body of this paper should consider first reading this appendix as a
warm-up exercise (while using Section \ref{sec:background}, as needed, for
definitions and notation).

Begin with {$%
\mathbb{Z}
$}$^{2}$, its standard geometric action on $%
\mathbb{R}
^{2}$, and the well-known $\mathcal{Z}$-structure obtained by adding the a
circle at infinity. Let $\phi:$ {$%
\mathbb{Z}
$}$^{2}\rightarrow${$%
\mathbb{Z}
$}$^{2}$ be the automorphism induced by $A=%
\begin{bmatrix}
1 & 1\\
0 & 1
\end{bmatrix}
$ and construct the semidirect product
\[
{%
\mathbb{Z}
}^{2}\rtimes_{\phi}%
\mathbb{Z}
=\left\langle x,y,t\mid xy=yx,t^{-1}xt=x,t^{-1}yt=xy\right\rangle
\]
This is one realization of \ the discrete Heisenberg group $H_{3}(${$%
\mathbb{Z}
$}$)$. The torus $T^{2}$ is a $K(%
\mathbb{Z}
^{2},1)$ and the standard Dehn twist homeomorphism $f:T^{2}\rightarrow T^{2}$
induces $\phi$ on fundamental groups, so $\operatorname*{Tor}_{f}(T^{2})$ has
fundamental group {$%
\mathbb{Z}
$}$^{2}\rtimes_{\phi}%
\mathbb{Z}
$. The infinite cyclic cover of $\operatorname*{Tor}_{f}(T^{2})$ is the
mapping telescope $\operatorname*{Tel}_{f}(T^{2})$. Since $f$ is a
homeomorphism, each $\mathcal{M}_{[n,n+1]}(f)$ is homeomorphic to $T^{2}%
\times\left[  n,n+1\right]  $ and $\operatorname*{Tel}_{f}(T^{2})\approx
T^{2}\times%
\mathbb{R}
$. The universal cover is homeomorphic to $%
\mathbb{R}
^{2}\times%
\mathbb{R}
$, but for the purposes of geometry should also be viewed as
$\operatorname*{Tel}_{f}(%
\mathbb{R}
^{2})$, where $f:${$%
\mathbb{R}
$}$^{2}\rightarrow${$%
\mathbb{R}
$}$^{2}$ (the lift of $f$) is the linear isomorphism defined by matrix $A$.

Associate each $\mathcal{M}_{[n,n+1]}(f)$ with $%
\mathbb{R}
^{2}\times\left[  n,n+1\right]  $ by sending the domain and range copies of $%
\mathbb{R}
^{2}$ to $%
\mathbb{R}
^{2}\times n$ and $%
\mathbb{R}
^{2}\times(n+1)$, respectively, via identity maps, and each interval
$x\times\left[  n,n+1\right]  $ (as in the definition of mapping cylinder)
linearly to the line segment from $(x,n)$ to $(f(x),n+1)$. A little linear
algebra shows that no two of these segments intersect, so we have a
homeomorphism. Under this realization of $\operatorname*{Tel}_{f}(%
\mathbb{R}
^{2})$, $\operatorname*{Cay}({%
\mathbb{Z}
}^{2}\rtimes_{\phi}%
\mathbb{Z}
)$ (for generating set $\left\{  x,y,t\right\}  $) is made up of the standard
$1\times1$ grids in each hyperplane $%
\mathbb{R}
^{2}\times n$ together with the line segments connecting each $((i,j),n)$ to
$(f(i,j),n+1)$. Here the vertex $((i,j),n)$ represents the group element
$t^{n}x^{i}y^{j}=\phi^{n}(x^{i}y^{j})t^{n}$, and the action of $t$ is by
translation of $%
\mathbb{R}
^{2}\times%
\mathbb{R}
$ along the $z$-axis.

It is tempting to use the Euclidean metric on $%
\mathbb{R}
^{2}\times%
\mathbb{R}
$ and its corresponding visual boundary in an attempt at a $\mathcal{Z}%
$-structure for {$%
\mathbb{Z}
$}$^{2}\rtimes_{\phi}%
\mathbb{Z}
$. Although this gives a $\mathcal{Z}$-compactification of
$\operatorname*{Tel}_{f}(%
\mathbb{R}
^{2})$, the nullity condition fails badly. For example, $y^{n}$-translations
of the segment connecting $(0,0,0)$ and $(0,0,1)$ grow linearly with $n$.
Those segments look small when viewed from the origin since they lie within $%
\mathbb{R}
^{2}\times\left[  0,1\right]  $; however, if we now translate them by powers
of $t$, they cast large shadows in the $2$-sphere at infinity.

It is possible to \textquotedblleft straighten\textquotedblright\ the $%
\mathbb{Z}
^{2}$-portion of the above {$%
\mathbb{Z}
$}$^{2}\rtimes_{\phi}%
\mathbb{Z}
$ action on $%
\mathbb{R}
^{2}\times%
\mathbb{R}
$ by conjugating with a homeomorphism $v:%
\mathbb{R}
^{2}\times%
\mathbb{R}
\rightarrow%
\mathbb{R}
^{2}\times%
\mathbb{R}
$ that sends the cosets of $\left\langle t\right\rangle \leq${$%
\mathbb{Z}
$}$^{2}\rtimes_{\phi}%
\mathbb{Z}
$ to vertical lines. (That is roughly what the map $v$ in the body of this
paper does.) This solves the problem uncovered in the previous paragraph. But
now, if we translate the unit square $\left[  0,1\right]  \times\left[
0,1\right]  \times0$ upward or downward using $t^{n}$, the effect is to apply
larger and larger (positive and negative) powers of $f$ to the 2-dimensional
hyperplanes. Since the diameters of the resulting parallelograms grow linearly
with $n$, their shadows in the $2$-sphere at infinity are large.

At this point, we have a pair of (non-geometric) proper cocompact $({%
\mathbb{Z}
}^{2}\rtimes_{\phi}%
\mathbb{Z}
)$-action on $%
\mathbb{R}
^{2}\times%
\mathbb{R}
$ for which the natural $\mathcal{Z}$-compactification fails the nullity
condition. To solve this problem, we look for an alternative way to attach a
2-sphere at infinity. Determining the right gluing is the delicate task at the
heart of Theorem \ref{Theorem A}.

\begin{remark}
Due to the simplicity of this example, we were able to realize
$\operatorname*{Tel}_{f}(%
\mathbb{R}
^{2})$ topologically as $%
\mathbb{R}
^{2}\times\mathbb{R}$. In general, however, the relevant mapping cylinders and
telescopes will only be \emph{(proper) homotopy equivalent }to, and not
\emph{homeomorphic to}, products. Our special case does generalize to
semidirect products of the form $%
\mathbb{Z}
^{n}\rtimes_{\phi}%
\mathbb{Z}
$. But for a more typical situation, consider free-by-$%
\mathbb{Z}
$ groups, illustrated in Figure~\ref{fig:freeauto}.
\end{remark}

\begin{figure}[th]
\begin{center}
\includegraphics[scale=0.5]{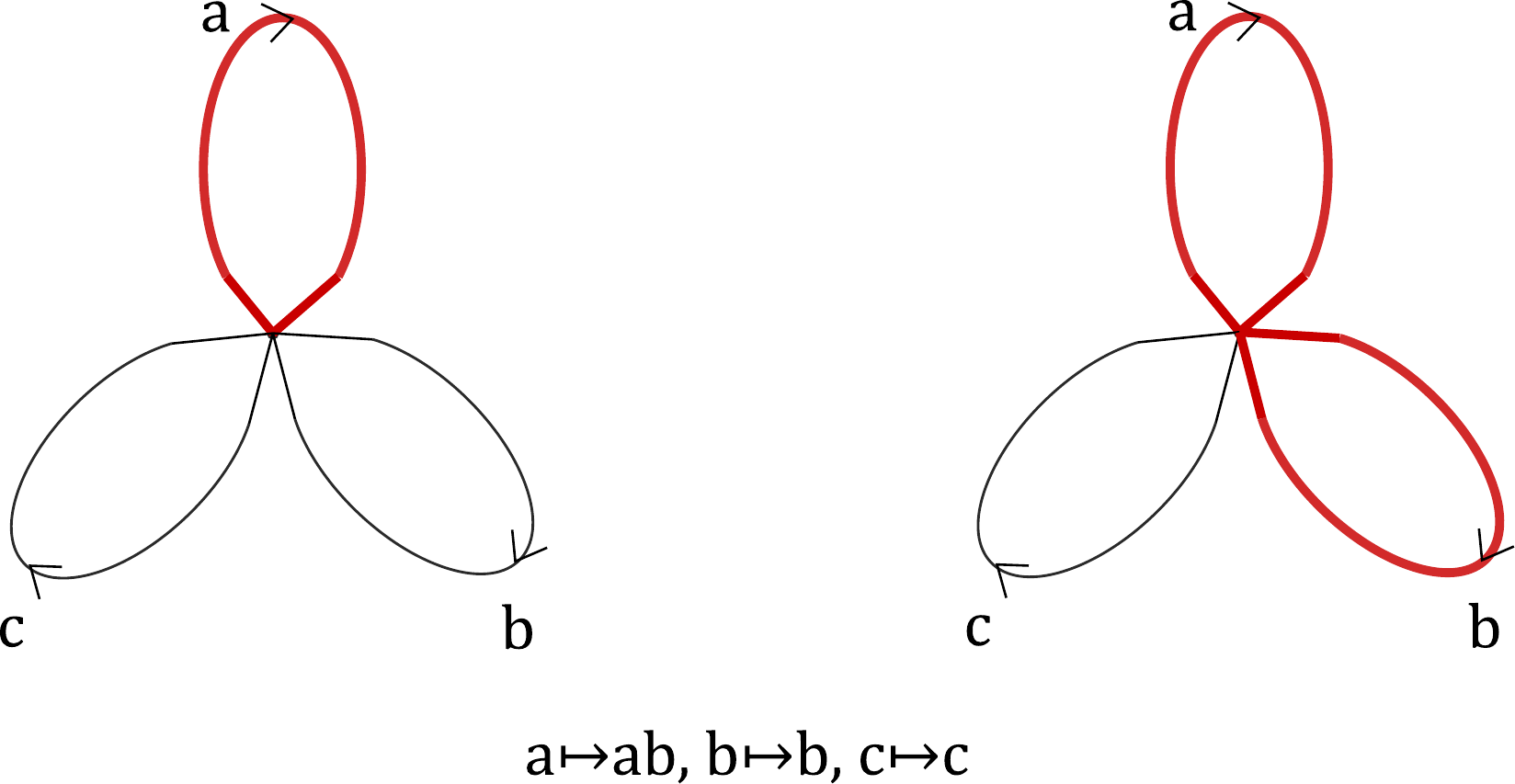} \label{fig:rose}
\end{center}
\caption{A free group automorphism need not be induced by a homemorphism of a
$K(F_{3},1)$.}%
\label{fig:freeauto}%
\end{figure}
\bibliographystyle{amsalpha}
\bibliography{Biblio}

\end{document}